\documentclass[12pt,letterpaper]{amsart}

\usepackage{datetime2}
\usepackage{enumitem}
\usepackage{fancyhdr}
\usepackage[margin=1.35in]{geometry}
\usepackage{graphicx}
\usepackage{hyperref}
\usepackage{mathtools}
\usepackage{verbatim}

\usepackage{tikz}
\usetikzlibrary{cd}

\newcommand{\fancymarker}[1]{
\fcolorbox{black}{black!05}{#1}
}

\newcommand{\ph}{\varphi}
\newcommand{\eps}{\varepsilon}
\newcommand{\ulim}{\varlimsup}

\newcommand{\EE}{\mathbb{E}}

\newcommand{\NN}{\mathbb{N}}
\newcommand{\RR}{\mathbb{R}}
\newcommand{\bbS}{\mathbb{S}}
\newcommand{\bbI}{\mathbb{I}}
\newcommand{\TT}{\mathbb{T}}
\newcommand{\ZZ}{\mathbb{Z}}

\newcommand{\AAA}{\mathcal{A}}
\newcommand{\BBB}{\mathcal{B}}
\newcommand{\CCC}{\mathcal{C}}
\newcommand{\DDD}{\mathcal{D}}
\newcommand{\EEE}{\mathcal{E}}
\newcommand{\FFF}{\mathcal{F}}
\newcommand{\GGG}{\mathcal{G}}
\newcommand{\HHH}{\mathcal{H}}
\newcommand{\III}{\mathcal{I}}
\newcommand{\LLL}{\mathcal{L}}
\newcommand{\MMM}{\mathcal{M}}
\newcommand{\OOO}{\mathcal{O}}
\newcommand{\PPP}{\mathcal{P}}
\newcommand{\QQQ}{\mathcal{Q}}
\newcommand{\RRR}{\mathcal{R}}
\newcommand{\SSS}{\mathcal{S}}
\newcommand{\UUU}{\mathcal{U}}
\newcommand{\VVV}{\mathcal{V}}

\newcommand{\tX}{\widetilde{X}}
\newcommand{\htop}{h_{\mathrm{top}}}
\newcommand{\one}{\mathbf{1}}
\newcommand{\Ws}{W^\mathrm{s}}
\newcommand{\Wu}{W^\mathrm{u}}
\newcommand{\Wsl}{\Ws_\mathrm{loc}}
\newcommand{\Wul}{\Wu_\mathrm{loc}}
\newcommand{\Vu}{\VVV^\mathrm{u}}
\newcommand{\Vuq}{\VVV^{1,\mathrm{u}}}
\newcommand{\Vs}{\VVV^{\mathrm{s}}}
\newcommand{\Vud}{\Vu_\delta}
\newcommand{\Vuqd}{\Vuq_\delta}
\newcommand{\Vsd}{\Vs_\delta}

\newcommand{\Kmin}{\mathcal{K}_{\mathrm{min}}}
\newcommand{\Kmax}{\mathcal{K}_{\mathrm{max}}}
\newcommand{\tmin}{\tau_{\mathrm{min}}}

\newcommand{\Breg}{\BBB^{\mathrm{reg}}}
\newcommand{\Mreg}{M^{\mathrm{reg}}}
\newcommand{\Xreg}{X^{\mathrm{reg}}}
\newcommand{\Yreg}{Y^{\mathrm{reg}}}

\newcommand{\pip}{\tau^{+}}
\newcommand{\pim}{\tau^{-}}

\newcommand{\cyl}[2]{ [{#1}_\bullet{#2}] }
\newcommand{\cylu}[1]{[{#1}]}
\newcommand{\bill}[1]{\overline{#1}}
\newcommand{\pbl}{[} 
\newcommand{\pbr}{]^+} 
\newcommand{\mbl}{[} 
\newcommand{\mbr}{]^-}
\newcommand{\markzero}[1]{\boxed{#1}\,}

\newcommand{\mf}{m^+}
\newcommand{\mb}{m^-}
\newcommand{\mt}{\widetilde{m}}
\newcommand{\mtf}{\mt^+}

\newcommand{\tnu}{\widetilde{\nu}}

\newcommand{\bx}{\bill{x}}

\newcommand{\trunc}{\phi}
\newcommand{\uf}{\overline{f}}

\newtheorem{theorem}{Theorem}[section]
\newtheorem{lemma}[theorem]{Lemma}
\newtheorem{proposition}[theorem]{Proposition}
\newtheorem{corollary}[theorem]{Corollary}
\newtheorem{thma}{Theorem}

\newtheorem*{thma*}{Theorem}
\theoremstyle{remark}
\newtheorem{remark}[theorem]{Remark}
\theoremstyle{definition}
\newtheorem{definition}[theorem]{Definition}

\numberwithin{equation}{section}
\numberwithin{figure}{section}

\title[Sinai billiard maps have unique MMEs]
{Every finite horizon Sinai billiard map has a unique measure of maximal entropy}

\author{Vaughn Climenhaga}
\address{Dept. of Mathematics, University of Houston, Houston, TX 77204, USA}
\email{climenha@math.uh.edu}

\author{Jason Day}
\address{Dept. of Mathematics, Hiroshima University, Higashihiroshima, 739-8526, Japan}
\email{jjday@hiroshima-u.ac.jp}

\date{April 28, 2026}
\thanks{This material is based upon work supported by the National Science Foundation under Award No.\ DMS-2154378 and DMS-2453314.}

\begin{document}
\begin{abstract}
Finite horizon Sinai billiard maps are examples of uniformly hyperbolic systems with singularities. These discontinuities make it more difficult to develop the classical theory of thermodynamic formalism. 
Nevertheless, Baladi and Demers established a variational principle for these systems, and proved that if the billiard table satisfies a certain sparse recurrence condition, then there is a unique measure of maximal entropy. We extend this existence and uniqueness result to all finite horizon Sinai billiard maps by giving a new proof that does not rely on the sparse recurrence condition.
Our construction is very concrete: the unique MME is obtained as the product of the Hausdorff measures on the one-sided subshifts associated to the billiard map. 
\end{abstract}

\maketitle

	
\section{Introduction}\label{sec:intro}

\subsection{Main results}

Sinai billiards are a class of billiards introduced in \cite{yS70}.  The billiard table is a flat torus with a finite collection of disjoint convex scatterers, and the system models the flight of a massless particle on the table.  The particle travels in a straight line until it encounters a scatterer; its post-collision trajectory is governed by elastic collision, as shown in Figure \ref{fig:billiard-flow}. Our main result is:

\begin{thma}\label{thm:mme-unique}
Every finite horizon Sinai billiard map has a unique measure of maximal entropy. This measure is multiply mixing, has full support (is positive on every open set), and has local product structure in the sense of Definition \ref{def:lps}.
\end{thma}

A more detailed version of Theorem \ref{thm:mme-unique} is given as Theorem \ref{thm:strategy} below. The proof includes the following bound (see \S\ref{sec:qsm}), which may be of independent interest.

\begin{thma}\label{thm:length-growth}
Let $T \colon M\to M$ be the collision map for a finite horizon Sinai billiard. Given any $\delta>0$, there exists $Q>0$ such that for every $C^1$ curve $V \subset M$ with length $|V|  \geq \delta$ whose tangent vectors lie in the unstable cone, we have
\begin{equation}\label{eqn:TnV}
Q^{-1} e^{n\htop(T)} \leq |T^n V| \leq Q e^{n\htop(T)},
\end{equation}
where $|T^n V|$ denotes the sum of the lengths of the connected components of $T^n V$.
\end{thma}

Theorems \ref{thm:mme-unique} and \ref{thm:length-growth} extend results of Baladi and Demers \cite{BD20}; see \S\ref{sec:related-literature}.

\subsection{Sinai billiards}\label{sec:background}

The Sinai billiard flow described informally in the previous section is illustrated in Figure \ref{fig:billiard-flow}.
Formally, we make the following definitions.  
\begin{itemize}
\item A \emph{scatterer} is a closed subset of $\TT^2 = \RR^2 / \ZZ^2$ whose boundary is $C^3$ and has strictly positive curvature.
\item Given a finite collection of pairwise disjoint scatterers $B_1,\dots, B_D \subset \TT^2$, the corresponding \emph{billiard table} is $\Omega := \TT^2 \setminus (\bigcup_{i=1}^D B_i)$.
\item A billiard table has \emph{finite horizon} if every infinite ray on $\TT^2$ eventually intersects the interior of some scatterer.
\end{itemize}  

\begin{remark}
The type of table we are considering is referred to as a ``category A'' table in \cite{CM06}.  Tables with infinite horizon or intersecting scatterers have different properties; in particular, infinite horizon Sinai billiards can have multiple measures of maximal (infinite) entropy \cite{CT96}.  More details about Sinai billiards and the different types of tables one can study can be found in \cite{CM06}.
\end{remark}

\begin{figure}
	\centering
\hfill\includegraphics[width=0.33\textwidth]{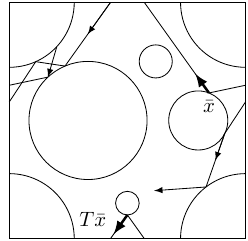}
\hfill
\includegraphics[width=0.33\textwidth]{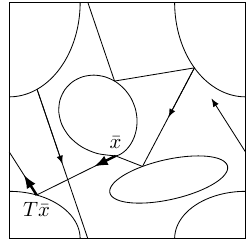}\hfill{}
	\caption{Trajectories of a particle on a billiard table.}\label{fig:billiard-flow}
\end{figure}

We will study the billiard map associated to $\Omega$, which only records the position and direction of the particle as it leaves each of the collisions it undergoes.\footnote{One could also study the \emph{billiard flow}, but we do not do so in this paper. See \cite{BCD24} for the extension of the Baladi--Demers approach to this continuous-time setting.}
Thus, the phase space of the system is
\[
M := \partial \Omega \times \left[-\frac{\pi}{2},\frac{\pi}{2}\right],
\]
where a point $\bar{x} = (r,\ph) \in M$ encodes the post-collision information as follows:
\begin{itemize}
\item $r\in \partial \Omega = \bigcup_{i=1}^D \partial B_i$ represents the location of the collision;
\item $\ph \in [-\frac\pi2,\frac\pi2]$ represents the outgoing angle with respect to the outward-pointing normal vector $\mathbf{n}_r$ to the scatterer (see Figure \ref{fig:phase-space}).
\end{itemize}
We denote the connected components of $M$ by $M_i = \partial B_i \times \left[-\frac{\pi}{2},\frac{\pi}{2}\right]$, and parametrize $\partial B_i$ by arclength, so $M_i$, which is topologically a cylinder (annulus), is represented as a rectangle whose vertical sides are identified.

\begin{figure}[htbp]
\includegraphics[width=.9\textwidth]{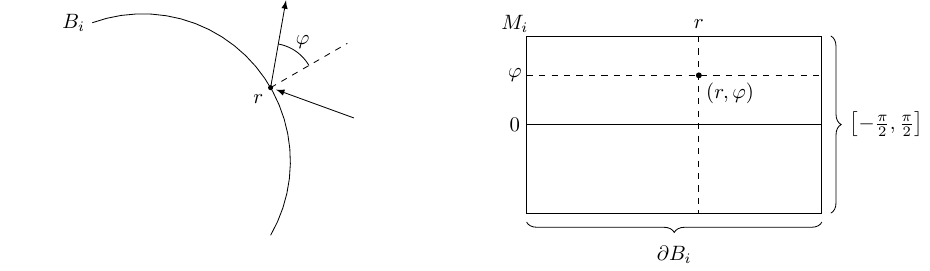}
\caption{Phase space for the billiard map.}
\label{fig:phase-space}
\end{figure}

The billiard map $T \colon M \to M$ is illustrated in Figure \ref{fig:billiard-map}. Formally, it can be defined as follows: given $(r,\ph) \in M$, let $\mathbf{v} \in \RR^2$ be the unique unit vector with the property that the signed angle from the normal vector $\mathbf{n}_r$ to $\mathbf{v}$ is equal to $\ph$. Let $\tau > 0$ be minimal such that $r' := r + \tau\mathbf{v} \in \partial \Omega$; this is the \emph{flight time}. Let $\ph'$ be the signed angle from $-\mathbf{v}$ to $\mathbf{n}_{r'}$ (note the order). This represents the outgoing angle after the collision at the point $r'$ (as measured from $\mathbf{n}_{r'}$ to the outgoing vector), and we write $T(r,\ph) = (r',\ph')$.
 
\begin{figure}[htbp]
\hfill
\includegraphics[width=.4\textwidth]{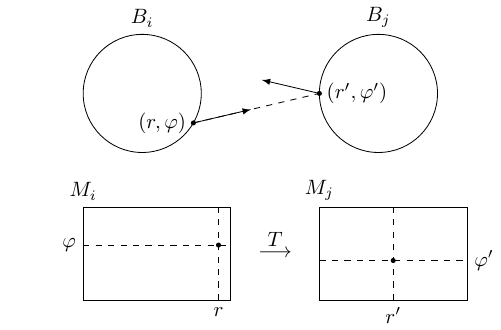}\hfill
\includegraphics[width=.3\textwidth]{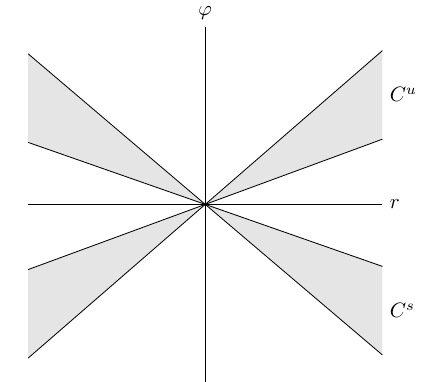}\hfill{}
\caption{The billiard map and its cones.}
\label{fig:billiard-map}
\end{figure}

Throughout the arguments, we will use the hyperbolicity properties of the billiard map, especially the invariant cones for $DT \colon TM \to TM$.
Since elements of $M$ are represented by pairs $(r,\ph)$, it is common to represent elements of $T_{\bx}M$ by pairs $(dr,d\ph) \in \RR^2$. In these coordinates, one can present invariant cones for $DT$. Start by letting $\Kmin$ and $\Kmax$ denote the minimum and maximum curvatures of $\partial \Omega$, and write $\tmin$ for the minimum flight time. Then consider the cones
\begin{equation}\label{eqn:cones}
\begin{aligned}
C^u &:= \{ (dr,d\ph) \in \RR^2 : \Kmin \leq d\ph/dr \leq \Kmax + 1/\tmin \}, \\
C^s &:= \{ (dr,d\ph) \in \RR^2 : -\Kmin \geq d\ph/dr \geq -\Kmax - 1/\tmin \}.
\end{aligned}
\end{equation}
The cones are invariant under $T$ in the sense that
\[
D_{\bx} TC^u) \subset C^{u}
\quad\text{and}\quad
D_{\bx} T^{-1}C^s \subset C^{s}.
\]
(In fact, the images lie in the interiors of the cones, but we will not need this fact.)
More details about the cones can be found in \cite[\S4.4]{CM06}; see also \cite[\S3]{BD20}.

\subsection{Measures of maximal entropy}\label{sec:mme}

Let $(X,\AAA)$ be a measurable space and $T\colon X\to X$ a measurable transformation. 
Let $\MMM(X)$ denote the set of all probability measures on $(X,\AAA)$, and write
\begin{align*}
\MMM_T(X) &:= \{ \mu \in \MMM(X) : \mu \text{ is $T$-invariant} \}, \\
\MMM_T^e(X) &:= \{ \mu \in \MMM_T(X) : \mu \text{ is ergodic} \},
\end{align*}
where we recall that $T$-invariance means that $T_* \mu = \mu$, and ergodicity means that every $E\in \AAA$ with $T^{-1}E = E$ must satisfy either $\mu(E) = 0$ or $\mu(E) = 1$.

We now review some basic definitions concerning entropy; see \cite{pW82} for more details and background.
Given $\nu \in \MMM_T(X)$ and a finite partition $\PPP$ of $X$ into measurable sets, the \emph{entropy} of $\nu$ with respect to $\PPP$ is
\begin{equation}\label{eqn:hnu}
h_\nu(T,\PPP) := \lim_{n\to\infty} \frac 1n \sum_{P\in \bigvee_{k=0}^{n-1} T^{-k}\PPP} -\nu(P) \log \nu(P).
\end{equation}
Here $\bigvee_{k=0}^{n-1} T^{-k}\PPP$ is the partition of $X$ into sets of the form $P = \bigcap_{k=0}^{n-1} T^{-k} P_k$, where $P_k$ is an element of $\PPP$.
The entropy of $\nu$ is
\[
h_\nu(T) := \sup \{ h_\nu(T,\PPP) : \PPP \text{ a finite partition of $X$ into measurable sets} \}.
\]

In later parts of our arguments, we will use the fact that the entropy $h_\nu(T,\PPP)$ can also be described in terms of the conditional measures of $\nu$ with respect to the partition $\bigvee_{k=0}^\infty T^k \PPP$; see  \S\ref{sec:cond-ent}.

When $X$ is a compact metric space and $T\colon X\to X$ is continuous, the \emph{topological entropy} of $T$ is
\begin{equation}\label{eqn:h-sep}
\htop(T) := \lim_{\eps\to 0} \ulim_{n\to\infty} \frac 1n \log \max \{ \# E : E \subset X \text{ is $(n,\eps)$-separated} \},
\end{equation}
where \emph{$(n,\eps)$-separated} means that given any $x,y\in E$ with $x\neq y$, there exists $k \in \{0,1,\dots, n-1\}$ such that $d(T^k x, T^k y) \geq \eps$. In this setting, the \emph{variational principle} shows that
\begin{equation}\label{eqn:var-princ}
\htop(T) = \sup \{ h_\nu(T) : \nu \in \MMM_T(X) \}
= \sup \{ h_\nu(T) : \nu \in \MMM_T^e(X) \}.
\end{equation}
A measure $\nu \in \MMM_T(X)$ that achieves the supremum in \eqref{eqn:var-princ} is called a \emph{measure of maximal entropy} (MME) for $T$.

When $(X,T)$ is a smooth uniformly hyperbolic system such as a transitive Anosov diffeomorphism, or a symbolic model of such a system, namely an irreducible subshift of finite type (SFT), there is a unique MME \cite{wP64,Bow08}.

A substantial amount of work has gone into extending this result beyond the classical setting. In particular, 
the standard proof of the variational principle \eqref{eqn:var-princ} does not apply to systems with singularities, such as Sinai billiards, and the usual arguments for existence and uniqueness of the MME also fail, so new techniques are required. 
In \S\ref{sec:coding-bill}, we will describe our approach to studying MMEs for Sinai billiards. In \S\ref{sec:related-literature}, we will describe some related results in the literature, especially the work of Baladi and Demers \cite{BD20}.

\subsection{Coding the billiard map}\label{sec:coding-bill}

Our proofs of Theorems \ref{thm:mme-unique} and \ref{thm:length-growth} will use a natural coding of the billiard map $T$ by a shift space on a finite alphabet. Here we give a brief description of this shift space, which will \emph{not} be a subshift of finite type. More complete details are given in \S\ref{sec:bill} below.

\begin{figure}[htbp]
\centering
\includegraphics{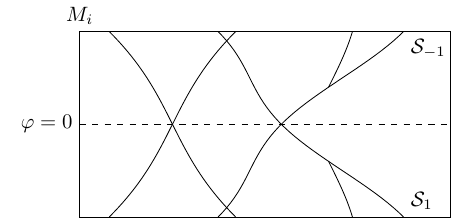}
\caption{The partition $\PPP$ on a connected component $M_i \subset M$.}
\label{fig:M-partition}
\end{figure}

Let $\mathcal{P}$ be the finite partition of $M$ into the maximal connected sets on which both $T$ and $T^{-1}$ are continuous; part of this partition is shown in Figure \ref{fig:M-partition}.
Writing 
\begin{equation}\label{eqn:S0}
\SSS_0 := \{ (r,\ph) \in M : |\ph| = \pi / 2 \}
\end{equation} 
for the set of points in $M$ that represent grazing collisions, which we call the \emph{singularity set}, 
elements of $\mathcal{P}$ are bounded by curves in $T^{\pm 1}\SSS_0$. As we will see in \S\ref{sec:bill}, these have tangent vectors lying in the cones $C^{u,s}$.

Index the elements of $\PPP$ by a set $A$, and let $c_0 \colon M\to A$ assign to each $\bx \in M$ the index of the partition element containing it. 
Writing $A^\ZZ$ for the space of bi-infinite sequences of symbols in $A$,
define a \emph{coding map}
\begin{equation}\label{eqn:coding}
c \colon M\to A^\ZZ
\end{equation}
by the condition that $(c(\bx))_n = c_0 (T^n \bx)$ for every $n\in \ZZ$. 
Equipping $A^\ZZ$ with the product topology, let $X = \overline{c(M)} \subset A^\ZZ$.
The diagram
\begin{equation}\label{eqn:comm-c}
\begin{tikzcd}
M \arrow[r, "T"] \arrow[d, "c"] & M \arrow[d, "c"] \\
X \arrow[r, "\sigma"] & X
\end{tikzcd}
\end{equation}
commutes, where $\sigma$ is the shift map on $X$ (see \S\ref{sec:shift-spaces-and-languages}). From this it follows that the pushforward map $c_* \colon \MMM(M) \to \MMM(X)$ has the property that invariant measures map to invariant measures:
\begin{equation}\label{eqn:c*}
c_*(\MMM_T(M)) \subset \MMM_\sigma(X).
\end{equation}
A vital fact, which we will prove in \S\ref{sec:coding}, is that:

\begin{proposition}\label{prop:1-1}
The coding map $c\colon M\to X$ is injective.
\end{proposition}

\begin{corollary}\label{cor:ent-pres}
The pushforward map $c_* \colon \MMM_T(M) \to \MMM_\sigma(X)$ preserves entropy: $h_{c_*\nu}(\sigma) = h_\nu(T)$ for every $\nu \in \MMM_T(M)$.
\end{corollary}

In \S\ref{sec:construction}, we will construct a ``candidate MME'' $\mu$ on the shift space $X$: this will be a $\sigma$-invariant measure that could \emph{a priori} be $0$ or infinite, but whose normalization must be an MME if it is nonzero and finite.
In \S\S\ref{sec:bill}--\ref{sec:bill-construction}, we will prove the following result, which implies Theorem \ref{thm:mme-unique}:

\begin{theorem}\label{thm:strategy}
Every ergodic $\nu \in \MMM_\sigma^e(X)$ with $h_\nu(\sigma)>0$ satisfies $\nu(c(M))=1$, so $c_*$ gives an entropy-preserving bijection between the sets of positive entropy ergodic measures for $(M,T)$ and $(X,\sigma)$.

Moreover, writing $\mu$ for the invariant measure on $X$ constructed in \S\ref{sec:construction}, the following are true:
\begin{enumerate}[label=\upshape{(\alph{*})}]
\item\label{mu-fin-pos}
The measure $\mu$ is nonzero and finite. 
\item\label{mu-lps-supp}
The normalized measure $\bar\mu = \mu/\mu(X)$ is an MME for $X$ that has local product structure in the sense of Definition \ref{def:lps}, and satisfies $\mu(c(U))>0$ for every open set $U\subset M$.
\item\label{mu-erg}
The measure-preserving system $(X,\sigma,\bar\mu)$ is multiply mixing.
\item\label{mu-uniq} 
Given any ergodic
$\nu \in \MMM_\sigma^e(X)$ such that $\nu \perp \bar\mu$, we have $h_\nu(\sigma) < h_{\bar\mu}(\sigma)$.
\end{enumerate}
\end{theorem}

The claim in the first paragraph of Theorem \ref{thm:strategy} is proved in \S\ref{sec:bill}. The remaining claims are proved in \S\ref{sec:bill-construction}, using estimates established in \S\ref{sec:counting}. The multiple mixing property of $\bar\mu$ is proved using the Hopf argument-based results from \cite{CHT16}. It seems reasonable to expect that $\bar\mu$ satisfies the stronger Kolmogorov and Bernoulli properties as well; but we do not address this question. (These properties are indeed satisfied under the sparse recurrence condition used by Baladi and Demers \cite{BD20}, which we discuss next.)

\subsection{Related literature}\label{sec:related-literature}

In \cite{BD20}, Baladi and Demers studied the measure of maximal entropy for Sinai billiard maps. They proved that for any finite horizon Sinai billiard map, the definition of topological entropy in \eqref{eqn:h-sep} is equivalent to
\begin{equation}\label{eqn:BD-ent}
\htop(T) = \lim_{n\to\infty} \frac 1n \log \# \MMM_0^n,
\end{equation}
where $\MMM_0^n$ is the partition of $M$ into maximal open connected sets on which $T^n$ is continuous. They also proved that $h_\nu(T) \leq \htop(T)$ for every $\nu \in \MMM_T(M)$, establishing half of the variational principle.

For Sinai billiards satisfying a certain ``sparse recurrence condition'', Baladi and Demers went further, establishing the full variational principle, existence of a unique MME, which satisfies the Bernoulli property,
and the uniform growth estimates stated in \eqref{eqn:TnV}; see \cite[Theorem 2.4 and Corollary 5.10]{BD20}.

The sparse recurrence condition can be described as follows. Given a fixed angle $\ph_0$ close to $\pi/2$ and a fixed $n_0 \in \NN$, let $s_0 \in (0,1]$ be such that any orbit of length $n_0$ has at most $s_0n_0$ collisions at which the absolute value of the post-collision angle with respect to the outward normal is greater than $\ph_0$.  The sparse recurrence condition in \cite{BD20} requires the existence of such $n_0, \ph_0, s_0$ such that $\htop(T) > s_0 \log(2)$. It appears to be an open problem to construct a billiard table that violates this sparse recurrence condition.

The approach used by Baladi and Demers involves constructing a certain Banach space of anisotropic distributions, on which they are able to prove a Ruelle--Perron--Frobenius-type theorem for the transfer operator, obtaining the unique MME as the product of the left and right eigenvectors. 
We follow \cite{BD20} for some of our arguments, most notably the uniform counting bounds in \S\ref{sec:counting}. However, our approach to the MME itself is very different. 
The heart of our construction is in \S\ref{sec:one-sided} and \S\ref{sec:exterior-product}, which describe a $\sigma$-invariant measure $\mu$ in terms of the Hausdorff measures on the one-sided shift spaces associated to $X$.

The construction in \S\ref{sec:construction} builds on ideas from \cite{CPZ}, in which leaf measures for MMEs (and other equilibrium measures) of uniformly hyperbolic systems were constructed as analogues of Hausdorff measure. We previously developed this approach in a (non-Markov) symbolic setting \cite{CD25}, and an overview of many of these ideas can be found in \cite{vC24}. An important innovation of the construction in this paper is that our definition of $\mu$ does \emph{not} rely on any ``pushforward and average'' procedure, or any local product structure within the shift space $X$, as was the case in \cite{CPZ,vC24,CD25}. Rather, our product construction is in some sense extrinsic, and our approach bears certain similarities to Patterson--Sullivan constructions on manifolds with some degree of hyperbolicity.

\begin{remark}
Beyond the fact that we are able to recover a uniqueness result with no sparse recurrence condition, another motivation for this approach is the hope that it may eventually be extended to a broader setting, such as non-uniformly hyperbolic billiards including the Bunimovich stadium, where the anisotropic Banach space machinery has not been developed.
\end{remark}

Under the sparse recurrence condition, \cite{BD20} also guarantees that the unique MME $\mu$ is \emph{adapted}, in the sense that it does not give too much weight to neighborhoods of the singularity set: $\int_M |\log d(\bx,T\SSS_0 \cup \SSS_0 \cup T^{-1}\SSS_0)| \,d\mu(\bx) < \infty$.
As shown by Lima and Matheus \cite{LM18}, this implies that there exists a countable Markov partition that codes $\mu$-a.e.\ point in $M$; the crucial ingredient is that adaptedness allows the Pesin theory developed in \cite{KSLP86} to be applied so that the ideas from Sarig's work on surface diffeomorphisms \cite{oS13} can be used.

Unlike \cite{BD20}, our approach does not provide any information about adaptedness: in particular, we do not know whether the MME would continue to be adapted on a (hypothetical) Sinai billiard that violates the sparse recurrence condition. In other settings where there is also a natural notion of adaptedness, nonadapted MMEs have been constructed for certain birational maps \cite[Example 4.6]{DDG11} and for piecewise expanding Markov interval maps where the analogue of the sparse recurrence condition fails \cite{lK25}.
Moreover, certain Sinai billiards have been shown to possess non-adapted measures with positive entropy \cite{CDLZ24}, raising the natural question of which entropies can be achieved by non-adapted measures. Recent results for interval maps \cite{lK26} suggest a conjecture that for every $0\leq \alpha < \htop(T)$, there exists an ergodic non-adapted $\nu \in \MMM_T^e(M)$ such that $h_\nu(T) = \alpha$.
This, however, lies beyond the scope of this paper.

\subsection{Outline of the paper}\label{sec:outline}

The remainder of the paper is organized as follows.

\subsubsection*{Section \ref{sec:construction}:}
Building on ideas from \cite{CD25}, we give a new method to construct an MME for a two-sided shift space $X$ via an ``external'' product construction: writing $\mb$ and $\mf$ for the Hausdorff measures on the one-sided shift spaces $X^-$ and $X^+$ (whose Hausdorff dimension is equal to the log of the topological entropy), we let $\mu$ be the restriction of $\mb\otimes\mf$ to $X \subset X^-\times X^+$.
We give general conditions -- \emph{uniform global counting bounds} -- under which $\mu$ is finite, and prove that if it is nonzero (which is harder to verify), then its normalization is an MME.

\subsubsection*{Section \ref{sec:bill}:}
We obtain further information about the natural coding of a Sinai billiard described in \S\ref{sec:coding-bill}, recalling various known billiard results from the literature, and establishing the existence of enough local product structure within $M$ that \emph{if} we can show that the measure $\mu$ from \S\ref{sec:construction} is finite and positive for the billiard shift space, then it will have local product structure. In \S\ref{sec:pos-ent}, we prove 
the first part of Theorem \ref{thm:strategy}, showing that every positive entropy ergodic measure on the billiard shift space is supported on $c(M)$.

\subsubsection*{Section \ref{sec:counting}:}
We establish the counting bounds on the billiard coding space that are necessary in order to prove that $\mu$ is positive and finite.
These arguments follow those used in \cite{BD20}, with some simplifications.
Along the way, we  establish the estimate on $|T^n V|$ in Theorem \ref{thm:length-growth}; see \S\ref{sec:qsm}.

\subsubsection*{Section \ref{sec:bill-construction}:}
First we prove that $\mu$ is positive and fully supported, completing the proof of Theorem \ref{thm:strategy}\ref{mu-fin-pos}--\ref{mu-lps-supp}. 
Then we use the local product structure of $\mu$ to show that
the normalized measure $\bar\mu$ is multiply mixing by applying results from \cite{CHT16}, which proves Theorem \ref{thm:strategy}\ref{mu-erg}. Finally, we show that the leaf measures satisfy a nonuniform lower Gibbs property, which allows us to show uniqueness by combining ideas from the arguments in \cite{CPZ2} and \cite{BD20}, completing the proof of Theorem \ref{thm:strategy}\ref{mu-uniq}, and thus the proof of Theorem \ref{thm:mme-unique} as well.

\section{General construction for shift spaces}\label{sec:construction}

In this section, we give a general procedure for constructing a measure of maximal entropy on a shift space, under fairly mild assumptions. This builds on our previous work in \cite{CD25}, which followed the general approach from \cite{CPZ} of using a dynamical analogue of Hausdorff measure to construct leaf measures with appropriate scaling properties under the dynamics; see also \cite{vC24}.

\subsection{Shift spaces and languages}\label{sec:shift-spaces-and-languages}

See \cite{LM21,BH86} for a detailed treatment of shift spaces and their associated languages; here we recall only those concepts that we will need.
For a finite set $A$, which we call the \emph{alphabet}, we consider the set $A^\ZZ$ of all bi-infinite sequences. The \emph{shift map} $\sigma \colon A^\ZZ \to A^\ZZ$ is defined by $(\sigma x)_n = x_{n+1}$ for all $n \in \ZZ$.  This is a homeomorphism in the product topology, which is induced by the metric
\begin{equation}\label{eqn:metric}
d(x,y) = 2^{-\min\{|n|: x_n\neq y_n\}}.
\end{equation}
The choice of $2$ as the base in \eqref{eqn:metric} is somewhat arbitrary; one can use any base $\theta > 1$ and produce the same topology. 
The only change would be that we replace $\log_2$ by $\log_\theta$ in the relationship between Hausdorff dimension and topological entropy at the beginning of \S\ref{sec:one-sided}.

A closed, $\sigma$-invariant set $X \subset A^\ZZ$ is called a (two-sided) \emph{subshift}, or a \emph{shift space}.  We will also work with \emph{one-sided} shift spaces defined by starting with $A^{\NN_0}$ or $A^{-\NN_0}$ and making appropriate modifications to the dynamics and the metric.\footnote{We write $\NN = \{1,2,3,\dots\}$ and $\NN_0 = \{0\} \cup \NN = \{0,1,2,3,\dots\}$.}

Given $n \in \NN_0$, we call each $w\in A^n$ a \emph{word} of \emph{length} $|w| = n$. The collection of all finite words over $A$ is denoted $A^* := \bigcup_{n=0}^\infty A^n$.

Given $x\in A^\ZZ$ and $i,j\in \ZZ$ with $i\leq j$, we write
\[
x_{[i,j]} := x_i x_{i+1} \cdots x_j \in A^{j-i+1}.
\]
We define $x_{[i,j)}$, $x_{(i,j]}$, and $x_{(i,j)}$ similarly, making the obvious changes. The \emph{language} of a shift space $X \subset A^\ZZ$ is
\begin{equation}\label{eqn:L}
\LLL := \{ x_{[i,j]} : x\in X,\ i,j\in \ZZ,\ i\leq j \} \subset A^*.
\end{equation}
The language $\LLL$ defined in \eqref{eqn:L} has the following properties.
\begin{itemize}
\item It is \emph{factorial}: for all $w\in \LLL$ and $i,j \in \{1,\dots, |w|\}$ with $i\leq j$, we have $w_{[i,j]} \in \LLL$.
\item It is \emph{extendable}: for all $w\in \LLL$, there exist $a,b\in A$ such that $aw\in \LLL$ and $wb\in \LLL$.
\end{itemize}
Conversely, a shift space can be defined in terms of its language. Indeed, given any factorial and extendable $\LLL \subset A^*$, the following set is a shift space whose language is $\LLL$:
\begin{equation}\label{eqn:X-from-L}
X = \{ x\in A^\ZZ : x_{[i,j]} \in \LLL \text{ for all } i,j\in \ZZ, i\leq j \}.
\end{equation}
For each $n\in \NN$, we write $\LLL_n := \LLL \cap A^n = \{ w\in \LLL : |w| = n \}$. More generally, given any $\DDD \subset \LLL$, we will write
\[
\DDD_n := \DDD \cap A^n = \{ w \in \DDD : |w| = n \}.
\]
We will also write
\[
\LLL_{\geq N} := \bigcup_{n\geq N} \LLL_n = \{ w\in \LLL : |w| \geq N \}.
\]
Given $n\in \NN$ and $w\in \LLL_n$, the \emph{forward and backward cylinders} associated to $w$ are
\begin{equation}\label{eqn:cyl}
\cyl{}{w} := \{ x\in X : x_{[0,n)} = w \}
\quad\text{and}\quad
\cyl{w}{} := \{ x\in X : x_{(-n,0]} = w \}.
\end{equation}
When $|w| = 1$, so that $w=a$ for some $a\in A$, we will simply write
\[
[a] = \cyl{}{a} = \cyl{a}{} = \{ x \in X : x_0 = a\}.
\]
Given $v,w\in \LLL$, we will write
\begin{equation}\label{eqn:vw-cyl}
[v{}_\bullet{} w] := \cyl{v}{} \cap \cyl{}{w}.
\end{equation}
Observe that this can only be nonempty if we have $v_{|v|} = w_1$.
It will sometimes be convenient to use the following notation:
\begin{equation}\label{eqn:Lkn}
\begin{aligned}
\LLL_{k,n} &:= \{ (v,w) \in \LLL_k\times \LLL_n : \cyl{v}{w} \neq \emptyset \} \\
&= \{ (v,w) \in \LLL_k \times \LLL_n : v_k = w_1 \text{ and } v'w\in \LLL \},
\quad\text{where } v' := v_{[1,|v|)}.
\end{aligned}
\end{equation}
For shift spaces, the definition of topological entropy in \eqref{eqn:h-sep} admits a simpler description, as the exponential growth rate of the number of words in the language:
\begin{equation}\label{eqn:htop}
\htop(X,\sigma) =  \lim_{n\to\infty} \frac 1n\log(\#\LLL_n).
\end{equation}
For convenience, we write $h = \htop(X,\sigma)$ throughout the paper.  The sequence $\#\LLL_n$ is  submultiplicative, so by Fekete's lemma, the limit in \eqref{eqn:htop} exists and is equal to $\inf_n \frac 1n \log (\#\LLL_n)$. In particular, we have
\begin{equation}\label{eqn:lcb}
\#\LLL_n \geq e^{nh}
\end{equation}
for all $n$.  
The corresponding upper bound is more subtle: for a general shift space, it is possible to have $\sup_n \#\LLL_n e^{-nh} = \infty$, although this quantity must be subexponential in $n$. An important step in our arguments below will be to prove that for the shift space modeling a Sinai billiard, this quantity is in fact bounded: there exists $C_1>0$ such that $\# \LLL_n \leq C_1 e^{nh}$ holds for all $n$. See Proposition \ref{prop:nondegen} for the importance of this condition in our general construction, and Proposition \ref{prop:ucb} for the result establishing this condition for the billiard shift space.

The partition into $1$-cylinders is generating for a shift space, so the measure-theoretic entropy of a $\sigma$-invariant measure also admits the simpler version
\begin{equation}\label{eqn:hnu-shift}
h_\nu(\sigma) = \lim_{n\to\infty} \frac 1n 
\sum_{w\in \LLL_n}
-\nu\cyl{}{w} \log \nu\cyl{}{w}.
\end{equation}
In \S\ref{sec:cond-ent}, we see how this entropy can be expressed in terms of the conditional measures of $\nu$, which plays an important role in the proof of uniqueness in \S\ref{sec:bill-construction}.

\subsection{One-sided shifts and stable/unstable sets}
Let $X \subset A^{\ZZ}$ be a two-sided shift space.  Define $\pip \colon X \to A^{\NN_0}$ and $\pim \colon X \to A^{-\NN_0}$ by
\[
\pip(x) = x_0 x_1 x_2\dots
\quad\text{and}\quad
\pim(x) = \dots x_{-2} x_{-1}x_0.
\]
The one-sided shift spaces associated to $X$ are
\[
X^+ := \pip(X) \subset A^{\NN_0}
\quad\text{and}\quad
X^- := \pim(X) \subset A^{-\NN_0},
\]
which comprise all forward and backward one-sided sequences associated to points in $X$.
We will denote cylinders in $X^\pm$ by
\[
\pbl w \pbr = \{x \in X^+ : x_{[0,|w|)} = w\}
\quad\text{and}\quad
\mbl w \mbr = \{x \in X^- : x_{(-|w|,0]} = w\}.
\]
Given $x\in X$, the local stable and unstable sets of $x$ are
\begin{align*}
\Wsl(x) & := \{y \in X : y_n = x_n \text{ for all } n \geq 0\}, \\
\Wul(x) & := \{y \in X : y_n = x_n \text{ for all } n \leq 0\}.
\end{align*}
Observe that if $y \in \Wsl(x)$, then 
\begin{equation}\label{eqn:d-to-0}
\lim_{n\to\infty} d(\sigma^n y,\sigma^n x) = 0.
\end{equation}
Similarly, for every $y\in\Wul(x)$, we have $d(\sigma^{-n}y, \sigma^{-n}x) \to 0$ as $n\to\infty$.

Given $x\in X$, the set $\Wsl(x)$ is completely determined by $\pip(x) \in X^+$. Thus without fear of ambiguity, we can also use the notation $\Wsl(x)$ when $x\in X^+$, and $\Wul(x)$ when $x\in X^-$.

Fixing $x\in X$ or $x\in X^+$, the map $\pim|_{\Wsl(x)} \colon \Wsl(x) \to X^-$ is injective, since each $y\in \Wsl(x)$ is completely determined by $y_{(-\infty,0]} = \pim(y)$ and $y_{[0,\infty)} = x_{[0,\infty)}$. Similarly, given $x\in X$ or $x\in X^-$, the map $\pip|_{\Wul(x)} \colon \Wul(x) \to X^+$ is injective. This will become important in \S\ref{sec:exterior-product}, when we produce measures on $\Wsl(x)$ and $\Wul(x)$ from measures on $X^-$ and $X^+$.

\subsection{One-sided measures}\label{sec:one-sided}

The notation and results given here for one-sided shifts have analogues in \cite{CD25}, which is written in the context of two-sided shifts.

Let $X$ be a two-sided shift space over a finite alphabet $A$ with language $\LLL$. Assume that $X$ has positive topological entropy $h>0$.  The one-sided shift spaces $X^\pm$ both have Hausdorff dimension 
$d = \log_2(h)$ in the metric \eqref{eqn:metric}; see \cite[Theorem 4.1]{SS15}.\footnote{The proof in \cite{SS15} is based on  \cite[Proposition III.1]{hF67}, which uses a different metric. The result can also be deduced by combining \cite[Proposition 1]{rB73} and \cite[Example 4.1]{vC11}.}

Let $m^{\pm}$ be the associated Hausdorff measures on $X^{\pm}$.
Every ball of radius $r = 2^{-n}$ in $X^+$ is of the form $\pbl w \pbr$ for some $w\in \LLL_n$, 
so that the weight $r^d$ in the definition of Hausdorff measure satisfies $r^d = 2^{-n\log_2 h} = e^{-nh}$, and we see that an equivalent definition of $\mf$ is as follows: given $Z \subset X^+$, we define
\begin{align}
\EE^+(Z,N) & := \left\{\EEE \subset \LLL_{\geq N} : Z \subset \bigcup_{w \in \EEE} \pbl w \pbr \right\}, \\
\mf(Z) & := \lim_{N\to\infty} \inf \left\{ \sum_{w \in \EEE} e^{-|w|h} : \EEE \in \EE^+(Z,N)\right\}.
\end{align}
We will refer to each $\EEE \in \EE^+(Z,N)$ as a \emph{cover} of $Z$.
Similarly, given $Z \subset X^-$, let
\begin{align}
\EE^-(Z,N) & := \left\{\EEE \subset \LLL_{\geq N} : Z \subset \bigcup_{w \in \EEE} \mbl w \mbr \right\}, \\
\mb(Z)  &:= \lim_{N\to\infty} \inf \left\{ \sum_{w \in \EEE} e^{-|w|h} : \EEE \in \EE^-(Z,N)\right\}.
\end{align}
A priori, there is no guarantee that the measures $m^\pm$ are nonzero and/or finite. However, these properties will hold under a fairly general condition on $\LLL$.  We will state some of the fundamental results here and reserve the proofs for \S\ref{sec:sym-pfs}.

\begin{proposition}\label{prop:nondegen}
Suppose that there exists $C_1 > 0$ such that
\begin{equation}\label{eqn:ucb}
\#\LLL_n \leq C_1 e^{nh} \quad\text{for all } n \in \NN.
\end{equation}
Then $C_1^{-1} \leq \mf(X^+) \leq C_1$ and $C_1^{-1} \leq  \mb(X^-) \leq C_1$.
\end{proposition}

The measures $m^{\pm}$ are \emph{not} generally shift-invariant, but do scale nicely:

\begin{proposition}\label{prop:scaling}
For every $a\in A$ and $Z\subset \pbl a\pbr$, we have
\begin{equation}\label{eqn:+scale}
\mf(\sigma Z) = e^h \mf(Z).
\end{equation}
Equivalently, for every nonnegative measurable $\psi \colon \pbl a \pbr \to [0,\infty)$, we have
\begin{equation}\label{eqn:+scale2}
\int_{\sigma\pbl a \pbr} \psi(ay) \,d\mf(y)
= e^h \int_{\pbl a \pbr} \psi(x) \,d\mf(x).
\end{equation}
Similarly, given any $Z \subset \mbl a\mbr$, we have
\begin{equation}\label{eqn:-scale}
\mb(\sigma^{-1}Z) = e^{h} \mb(Z),
\end{equation}
and for every nonnegative measurable $\psi \colon \mbl a\mbr \to [0,\infty)$, we have
\begin{equation}\label{eqn:-scale2}
\int_{\sigma^{-1}\mbl a\mbr} \psi(ya) \,d\mb(y)
= e^h \int_{\mbl a\mbr} \psi(x) \,d\mb(x).
\end{equation}
\end{proposition}

Observe that \eqref{eqn:+scale} can be deduced from \eqref{eqn:+scale2} by putting $\psi = \one_Z$. Conversely, \eqref{eqn:+scale2} can be deduced from \eqref{eqn:+scale} 
via the standard procedure of passing from characteristic functions to simple functions by linearity and then taking monotone limits. A similar argument shows the equivalence of \eqref{eqn:-scale} and \eqref{eqn:-scale2}.

Iterating \eqref{eqn:+scale}, we see that given any $n\in \NN$, $w\in \LLL_n$, and $Z \subset \pbl w\pbr$, we have
\begin{equation}\label{eqn:w+scale}
\mf(\sigma^n Z) = e^{nh} \mf(Z).
\end{equation}
In particular, taking $Z = [w]^{\pm}$ and recalling
Proposition \ref{prop:nondegen}, we deduce:

\begin{corollary}\label{cor:Gibbs}
Suppose there exists $C_1>0$ such that \eqref{eqn:ucb} holds. Then for every $n\in \NN$ and $w\in \LLL_n$, we have the upper Gibbs estimates
\begin{equation}\label{eqn:up-Gibbs}
\mf(\pbl w\pbr) \leq C_1 e^{-nh}
\quad\text{and}\quad
\mb(\mbl w \mbr) \leq C_1 e^{-nh}.
\end{equation}
Since we assumed that $h>0$, this implies that $\mf$ and $\mb$ are nonatomic.
\end{corollary}

Eventually, it will be important to have a lower bound that complements the upper bound in \eqref{eqn:up-Gibbs}. This is a more subtle issue: Corollary \ref{cor:Gibbs} follows from \eqref{eqn:w+scale} because $\sigma^n \pbl w \pbr \subset X^+$ implies that $\mf(\sigma^n \pbl w \pbr) \leq \mf(X^+)$, but to obtain a lower bound requires more information on $\sigma^n \pbl w \pbr$, since in general we may have $\sigma^n \pbl w \pbr \neq X^+$. For now, we content ourselves with the following result.

\begin{proposition}\label{prop:Z-compact-lower}
Suppose there exists $C_1>0$ satisfying \eqref{eqn:ucb}.
Let $Z \subset X^+$ be compact, and suppose there exist $b_Z>0$ and 
an infinite set $J \subset \NN$ such that
\begin{equation}\label{eqn:Z-lcb}
\#\{w \in \LLL_n : \pbl w \pbr \cap Z \neq \emptyset\} \geq b_Z e^{nh}
\quad\text{for every } n\in J.
\end{equation}
Then we have
\begin{equation}\label{eqn:Z-compact-lower}
\mf(Z) \geq b_Z /C_1.
\end{equation} 
An analogous result holds for $\mb$.
\end{proposition}

\begin{remark}\label{rmk:Z=X+}
When $Z = X^+$, we can take $b_Z=1$ and $J=\NN$ by \eqref{eqn:lcb}, so \eqref{eqn:Z-compact-lower} is a generalization of the lower bound in Proposition \ref{prop:nondegen}.
\end{remark}

\begin{remark}
Proposition \ref{prop:Z-compact-lower} fails for noncompact sets.  Indeed, if $Z$ is a countable dense subset of $X^+$, then we can verify \eqref{eqn:Z-lcb} just as in Remark \ref{rmk:Z=X+} by taking $b_Z=1$ and $J=\NN$, but since $\mf$ is nonatomic (Corollary \ref{cor:Gibbs}), we have $\mf(Z) = 0$.
\end{remark}

\subsection{Product construction and conditions for an MME}\label{sec:exterior-product}

Now the idea is to construct an MME on the two-sided shift $X$ as the product of the measures $\mf$ and $\mb$ on the one-sided shifts $X^{\pm}$, deducing invariance by combining \eqref{eqn:+scale} and \eqref{eqn:-scale}, and maximal entropy via the upper Gibbs bounds in \eqref{eqn:up-Gibbs}. 

This does not work out as simply as one might hope, because $X$ is not the product of $X^+$ and $X^-$. To deal with the resulting subtleties, start by writing
\[
\Delta := \{ (x,y) : x\in X^-, y\in X^+, x_0 = y_0 \} \subset X^- \times X^+,
\]
and define a concatenation map $\Pi \colon \Delta \to A^\ZZ$ by
\begin{equation}\label{eqn:Pi}
\Pi(x,y) := \cdots x_{-2}x_{-1} \markzero{y_0} y_1 y_2 \cdots,
\end{equation}
where the box indicates the position of the $0$ index.
The closed set
\[
\tX := \Pi(\Delta) \subset A^\ZZ,
\]
contains $X$, but is typically not a shift space itself. Indeed, one can see that:
\begin{itemize}
\item $\tX$ is invariant if and only if $X$ is a one-step Markov shift, in which case $X = \tX$;
\item in general, $X = \bigcap_{n\in \ZZ} \sigma^{-n}\tX$ is the largest invariant subset of $\tX$.
\end{itemize}
Now we push forward the product measure $\mb \otimes \mf$ on $X^- \times X^+$ to
\begin{equation}\label{eqn:mt-prod}
\mt := \Pi_*((\mb \otimes \mf)|_{\Delta})
\end{equation}
on $\tX$.
By Fubini's theorem, $\mt$ can also be described as
\begin{equation}\label{eqn:mt}
\begin{aligned}
\mt(E) &= 
\int_{X^-} \mf \{ y \in X^+ : (x,y) \in \Pi^{-1}(E) \} \, d\mb(x) \\
&= \int_{X^+} \mb \{ x\in X^- : (x,y) \in \Pi^{-1}(E) \} \, d \mf(y).
\end{aligned}
\end{equation}
In order to obtain a $\sigma$-invariant measure, we must restrict to $X \subset \tX$, and we put
\begin{equation}\label{eqn:mu}
\mu := \mt |_{X}.
\end{equation}

\begin{remark}\label{rmk:need-nonzero}
Even if \eqref{eqn:ucb} holds, so that Proposition \ref{prop:nondegen} gives $0 < \mt(\tX) < \infty$, there is no guarantee that $\mt(X) > 0$, so $\mu$ could be the $0$ measure. Proving that $\mu(X)>0$ will be an important part of our arguments for Sinai billiards. 
\end{remark}

\begin{proposition}\label{prop:inv-MME}
The measure $\mu$ given by \eqref{eqn:mu} is a $\sigma$-invariant Borel measure on $X$. If there exists $C_1$ such that the upper counting bound \eqref{eqn:ucb} holds, 
then the measure $\mu$ has the following additional properties.
\begin{enumerate}[leftmargin=*]
\item Finiteness: $\mu(X) < \infty$.
\item Upper Gibbs bound: $\mu(\cyl{}{w}) \leq C_1^2 e^{-|w| h}$ for all $w\in \LLL$.
\item If $\mu(X) > 0$, then the normalized measure $\bar\mu := \mu/\mu(X)$ is an MME for $X$. \label{item:shift-mme}
\end{enumerate}
\end{proposition}

Given $x\in X$, we define a measure $\mf_x$ on $X$, supported on $\Wul(x)$, by
\begin{equation}\label{eqn:mfx}
	\mf_x(Z) := \mf(\pip(Z \cap \Wul(x)))
\end{equation}
for every Borel $Z \subset X$. Since
$\Wul(x)$ only depends on $\pim(x)$, we can also write $\Wul(x)$ and $\mf_x(Z)$ when $x\in X^-$, without fear of ambiguity.
From \eqref{eqn:mt}, we have
\begin{equation}\label{eqn:u-cond}
\mu(Z) = \int_{X^-} \mf_x(Z) \,d\mb(x).
\end{equation}
Similarly, defining $\mb_x$ supported on $\Wsl(x)$ by
\begin{equation}\label{eqn:mbx}
\mb_x(Z) := \mb(\pim(Z \cap \Wsl(x))),
\end{equation}
where again we take either $x\in X$ or $x\in X^+$, we have
\begin{equation}\label{eqn:s-cond}
\mu(Z) = \int_{X^+} \mb_x(Z) \,d\mf(x).
\end{equation}

\begin{lemma}\label{lem:nonzero-mu-ae}
The measures $\mf_x$ and $\mb_x$ are nonzero for $\mu$-a.e.\ $x\in X$. In fact, the sets
\begin{equation}\label{eqn:X*su}
\begin{gathered}
X_*^u := \{ x \in X : \mf_x(\Wul(x)) = 0 \}, \\
X_*^s := \{ x\in X : \mb_x(\Wsl(x)) = 0\}
\end{gathered}
\end{equation}
satisfy the following:
\begin{align}
\label{eqn:m-X*}
\mf_x(X_*^u) &= \mb_x(X_*^s) = 0
\quad\text{ for every } x\in X, \text{ and}\\
\label{eqn:mu-X*}
\mu(X_*^u) &= \mu(X_*^s) = 0.
\end{align}
\end{lemma}
\begin{proof}
We claim that $\mf_x(X_*^u) = 0$ for every $x\in X$.
If $x\in X_*^u$, then this follows since $\mf_x \equiv 0$ by \eqref{eqn:X*su}. If $x\in X \setminus X_*^u$, then for every $y\in \Wul(x)$, we have $\mf_y(\Wul(y)) = \mf_x(\Wul(x)) > 0$, so (again by \eqref{eqn:X*su}) $y\in X \setminus X_*^u$; we conclude that 
$\Wul(x) \cap X_*^u = 0$, and thus $\mf_x(X_*^u) = 0$, proving the claim.

Using the claim together with the disintegration of $\mu$ in  \eqref{eqn:u-cond}, we obtain
\[
\mu(X_*^u) = \int_{X^-} \mf_x(X_*^u) \,d\mb(x)
= \int_{X^-} 0 \,d\mb(x) = 0.
\]
A similar argument holds for $X_*^s$.
\end{proof}

\begin{remark}
Conclusion \eqref{eqn:mu-X*} of Lemma \ref{lem:nonzero-mu-ae} is vacuous if $\mu(X)=0$, so as pointed out in Remark \ref{rmk:need-nonzero}, it will be crucial to prove that $\mt(X)>0$.
\end{remark}

\begin{definition}
A set $R \subset X$ is a \emph{rectangle} if for every $x,y\in R$, the intersection $\Wsl(x) \cap \Wul(y)$ consists of exactly one point, and this point lies in $R$. Equivalently, $R$ is a rectangle if we have $R = \Pi(R^- \times R^+)$ for some $R^\pm \subset X^\pm$ with $R^- \times R^+ \subset \Delta$; in this case we must have $R^- = \pim(R)$ and $R^+ = \pip(R)$.
\end{definition}

\begin{definition}\label{def:lps}
A \emph{product measure} on a rectangle $R = \Pi(R^-\times R^+)$ is any measure of the form $\Pi_*(m_1\otimes m_2)$ for some measures $m_1$ on $R^-$ and $m_2$ on $R^+$.
A measure $\nu$ on $X$ has \emph{local product structure} if for $\nu$-a.e.\ $x \in X$, there is a rectangle $R_x$ and product measure $\nu_{R_x}$ on $R_x$ such that $\nu_{R_x}(R_x)>0$ and the measures $\nu|_{R_x}$ and $\nu_{R_x}$ are equivalent (mutually absolutely continuous).
\end{definition}

The following is a direct consequence of the definitions:

\begin{lemma}\label{lem:lps}
Let $\mu$ be the invariant measure on $X$ constructed in \eqref{eqn:mt-prod} and \eqref{eqn:mu}. If there exist rectangles $\{R_n \subset X : n\in \NN\}$ such that $\mu(X \setminus \bigcup_{n\in \NN} R_n) = 0$, then $\mu$ has local product structure.
\end{lemma}

Given a rectangle $R\subset X$ and two points $x,y\in R$, 
the \emph{stable holonomy map}
$\pi_{x,y}^s \colon R \cap \Wul(x) \to R \cap \Wul(y)$ is defined by $\pi_{x,y}^s(z) = \Pi(y,z)$.  The following result is also stated in \cite[Proposition 1.10]{CD25}.

\begin{lemma}\label{lem:holonomy}
Let $R \subset X$ be a set with product structure. Given any $x,y \in R$ and any $Z \subset R \cap \Wul(x)$, we have
\begin{equation}
\mf_x(Z) = \mf_y(\pi_{x,y}^s(Z)). 
\end{equation}
\end{lemma}
\begin{proof}
Fix $R,x,y,Z$ be as in the statement.
Given  $N\in \NN$ and every $\EEE \in \EE^+(Z,N)$, 
observe that every $z\in Z$ has $z\in \cyl{}{w}$ for some $w\in \EEE$,
and since $\Wsl(z) \subset \cyl{}{w}$ by definition, we also have
$\pi_{x,y}^s(z) \in \cyl{}{w}$. We conclude that $\EEE \in \EE^+(\pi_{x,y}^s(Z),N)$, so
\[
\lim_{N\to\infty} \inf \left\{ \sum_{w \in \EEE} e^{-|w|h} : \EE^+(Z,N)\right\} \geq \lim_{N\to\infty} \inf \left\{ \sum_{w \in \EEE} e^{-|w|h} : \EE^+(\pi_{x,y}^s(Z),N)\right\},
\]
giving $\mf_x(Z) \geq \mf_y(\pi_{x,y}^s(Z))$. The reverse inequality follows by symmetry.
\end{proof}

\subsection{Conditional measures and entropy}\label{sec:cond-ent}

The disintegration formulas \eqref{eqn:u-cond} and \eqref{eqn:s-cond} show that up to scaling, the leaf measures $\mf_x$ and $\mb_x$ are the conditional measures of $\mu$ on unstable and stable leaves. The usual version of this disintegration formula for an arbitrary Borel probability measure involves conditional measures that are \emph{probabilities} on each individual leaf. We recall the details of this now, together with an interpretation of entropy in terms of these conditional probability measures.

We will use the following notation: given
a Borel probability measure $\nu$ on $X$, 
and words $v,w\in \LLL$, let
\begin{equation}\label{eqn:cond-prob}
P^\nu(w\mid v)
:= \frac{\nu( \cyl{v}{} \cap \sigma^{-1}\cyl{}{w} )}
{\nu(\cyl{v}{})}
\end{equation}
denote the conditional probability that we will see the word $w$ starting in position $1$, given that we have just seen the word $v$ ending in position $0$. The following result uses these probabilities to describe the disintegration of a measure into its conditional measures, which goes back to Rokhlin \cite{Roh}.

\begin{lemma}\label{lem:disint}
Given a Borel probability measure $\nu \in \MMM(X)$, let $\nu^- := (\pim)_*\nu$ be its pushforward measure on $X^-$. 
Then for $\nu^-$-a.e.\ $x\in X^-$, the following limit exists for every $w\in \LLL$:
\begin{equation}\label{eqn:lim-cond}
P^\nu(w \mid x) :=
\lim_{k\to\infty} P^\nu(w \mid x_{(-k,0]}).
\end{equation}
Moreover, \eqref{eqn:lim-cond} defines a probability measure on $\Wul(x)$: there exists a unique Borel probability measure $\nu_x$ on $\Wul(x)$ such that $\nu_x(\Wul(x) \cap \cyl{}{w}) = P^\nu(w\mid x)$ for all $w\in \LLL$. Finally, we have the following analogue of \eqref{eqn:u-cond}:
\begin{equation}\label{eqn:disint}
\nu(Z) = \int_{X^-} \nu_x(Z) \,d\nu^-(x)
\quad\text{for every Borel } Z\subset X.
\end{equation}
\end{lemma}
\begin{proof}
See \cite[\S5.3]{EW11}, especially Corollary 5.21 and Theorem 5.14, which provide \eqref{eqn:lim-cond} and \eqref{eqn:disint}, respectively. This result is also proved in \cite[Chapter 5]{VO16}; see in particular Definition 5.1.4, Theorem 5.1.11, and Lemma 5.2.1.
\end{proof}

Let $\CCC_\nu$ denote the set of $x\in X$ such that the limit in \eqref{eqn:lim-cond} exists for every $w\in \LLL$. 
We will abuse notation slightly by writing $\nu_x$ to mean the same thing as $\nu_{\pim(x)}$.
By Lemma \ref{lem:disint}, we have $\nu(\CCC_\nu)=1$, and we refer to the measures $\{ \nu_x : x\in \CCC_\nu\}$ as the \emph{conditional measures} of $\nu$ on unstable sets. 

A similar construction gives conditional measures on stable sets. The following discussion has a completely analogous version for these stable conditionals, but we will only write down (and use) the details for the unstable sets and measures.

Let $\nu$ be an \emph{invariant} Borel probability measure on $X$, and define a sequence of functions $p_n^\nu \colon X \to [0,1]$ by
\begin{equation}\label{eqn:pnx}
p_n^\nu(x) := \nu_x (\cyl{}{x_{[1,n]}}) = P^\nu( x_{[1,n]} \mid \pim(x) ).
\end{equation}
For each $n\in \NN$, we also consider the function
\begin{equation}\label{eqn:In}
I_n^\nu(x) := -\log p_n^\nu(x),
\end{equation}
which gives the information content of $x_{[1,n]}$, conditioned on $x_{(-\infty,0]}$. We will use the following characterization of $h_\nu$ in terms of these functions:

\begin{lemma}\label{lem:h-int}
Given any invariant measure $\nu \in \MMM_\sigma(X)$ and any $n\in \NN$, we have
\begin{equation}\label{eqn:h-int}
n h_\nu = \int_X I_n^\nu(x) \,d\nu(x).
\end{equation}
\end{lemma}
\begin{proof}
A standard result from entropy theory (see \cite[Theorem 14.29]{Glasner}) says that given a probability space $(X,\nu)$, an invertible measure-preserving transformation $T\colon X\to X$, and a finite partition $\QQQ$ of $X$, we have
\begin{equation}\label{eqn:H-cond}
h_\nu(T,\QQQ) = H_\nu(\QQQ \mid \QQQ^-),
\qquad
\text{where } \QQQ^- := \bigvee_{k=1}^\infty T^k\QQQ.
\end{equation}
Here the conditional entropy on the right-hand side is defined by
\[
H_\nu(\QQQ \mid \QQQ^-) := -\sum_{A \in \QQQ} 
\int_X \one_A(x) \log \nu_x^{\QQQ^-}(A) \,d\nu(x),
\]
where $\{\nu_x^{\QQQ^-} : x\in X \}$ is the system of conditional measures of $\nu$ with respect to $\QQQ^-$ \cite[p.\ 254]{Glasner}.
Taking $X$ to be our shift space, $\QQQ$ to be the partition into cylinders of the form $\sigma^{-1}\cyl{}{x_{[1,n]}}$, and $T = \sigma^n$, we see that $\QQQ^-$ is the partition into unstable sets, and $\nu_x^{\QQQ^-} = \nu_x$ (the conditional measures on unstable sets introduced above), so recalling \eqref{eqn:pnx}, we get
\[
H_\nu(\QQQ \mid \QQQ^-) = \int_X \sum_{w\in \LLL_n} -\log \nu_x(\cyl{}{w}) \,d\nu(x)
= \int_X \int_{\Wul(x)} -\log p_n^\nu(y) \,d\nu_x(y) \,d\nu(x).
\]
Using \eqref{eqn:In}, we can rewrite this as
\[
H_\nu(\QQQ \mid \QQQ^-) = \int_X I_n^\nu(x) \,d\nu(x),
\]
and now \eqref{eqn:h-int} follows from \eqref{eqn:H-cond} and the fact that $h_\nu(T,\QQQ) = h_\nu(\sigma^n) = n h_\nu$.
\end{proof}

Using \eqref{eqn:lim-cond}, \eqref{eqn:pnx}, and invariance of $\nu$, we see that
\[
p_n^\nu(x) = \lim_{k\to-\infty} P^\nu(x_{[1,n]} \mid x_{(-k,0]})
= \lim_{k\to-\infty} \prod_{j=0}^{n-1}
P^\nu(x_{j+1} \mid x_{(-k,j]})
= \prod_{j=0}^{n-1} p_1^\nu(\sigma^j(x)),
\]
so the conditional information function can be written as a Birkhoff sum:
\begin{equation}\label{eqn:InSn}
I_n^\nu(x) = S_n I_1^\nu(x).
\end{equation}
By the Birkhoff ergodic theorem, we conclude that the \emph{pointwise unstable entropy} $h_\nu^u(x) := \lim_{n\to\infty} \frac 1n I_n^\nu(x)$ exists for $\nu$-a.e.\ $x\in X$, and
\begin{equation}\label{eqn:h-int=}
\int_X h_\nu^u(x) \,d\nu(x)
= \int_X I_1^\nu(x) \,d\nu(x)
= h_\nu,
\end{equation}
where the last equality uses Lemma \ref{lem:h-int}. When $\nu\in \MMM_\sigma^e(X)$ is an \emph{ergodic} invariant measure, we have moreover that $h_\nu^u$ is constant $\nu$-a.e., and obtain the following version of the Shannon--McMillan--Breiman Theorem for conditional measures:
\begin{equation}\label{eqn:SMB}
h_\nu^u(x) = \lim_{n\to\infty} \frac 1n I_n^\nu(x)
= \lim_{n\to\infty} -\frac 1n \log \nu_x(\cyl{}{x_{[1,n]}})
= h_\nu \quad\text{for $\nu$-a.e.\ $x\in X$.}
\end{equation}
In particular, for any ergodic measure $\nu$, the set
\begin{equation}\label{eqn:Snu}
\SSS_\nu := \Big\{ x\in X : \frac 1n I_n^\nu(x) \to h_\nu \Big\}
\end{equation}
satisfies $\nu(\SSS_\nu) = 1$, and using \eqref{eqn:disint}, we conclude that
\begin{equation}\label{eqn:S'nu}
\SSS'_\nu := \{ x \in \CCC_\nu : \nu_x(\SSS_\nu) = 1 \}
\end{equation}
also has $\nu(\SSS'_\nu)=1$.

\begin{lemma}\label{lem:non-atomic}
Let $\nu$ be an ergodic measure with $h_\nu > 0$. Then for every $x\in \SSS'_\nu$, the conditional measure $\nu_x$ is non-atomic: $\nu_x(\{y\}) = 0$ for every $y\in \Wul(x)$.
\end{lemma}
\begin{proof}
By \eqref{eqn:S'nu}, it suffices to show that $\nu_x(\{y\}) = 0$ for every $x\in \SSS'_\nu$ and $y\in \SSS_\nu$. 
Each such $y$ satisfies $\nu_y = \nu_x$, and since $h_\nu>0$, \eqref{eqn:SMB} implies that $\nu_y(y_{[1,n]}) \to 0$, from which we conclude that $\nu_x(\{y\}) = \nu_y(\{y\}) = 0$.
\end{proof}

\subsection{Proofs of basic properties}\label{sec:sym-pfs}

Now we prove the remaining unproved results from the preceding sections.

\subsubsection{Proof of Proposition \ref{prop:nondegen}}

We prove that $C_1^{-1} \leq \mf(X^+) \leq C_1$; the proof for $\mb(X^-)$ is similar.
For the upper bound, observe that $\LLL_N \subset \EE^+(X^+,N)$, so
\[
m^+(X^+) \leq \lim_{N\to\infty} \sum_{w\in \LLL_N} e^{-|w|h}
\leq \lim_{N\to\infty} (\#\LLL_N) e^{-Nh} \leq C_1.
\]
For the lower bound, 
we use the following notation: given $w\in \LLL$ and $k\in \NN$, denote the set of words of length $k$ that can follow $w$ by
\[
\FFF_k(w) := \{ v\in \LLL_k : wv \in \LLL \}.
\]
Each $\EEE \in \EE^+(X^+,N)$ is an open cover of $X^+$, and has a finite subcover $\EEE' \subset \EEE$ by compactness. Let $n := \max\{ |w| : w\in \EEE' \}$. For every $w\in \EEE'$, we have
\[
\#\FFF_{n-|w|}(w) \leq \#\LLL_{n-|w|} \leq C_1 e^{(n-|w|)h},
\]
which implies that
\begin{equation}\label{eqn:Fnw}
e^{-|w|h} \geq C_1^{-1} e^{-nh} \#\FFF_{n-|w|}(w).
\end{equation}
Moreover, $\LLL_n = \bigsqcup_{w\in \EEE'} \{ wv : v\in \FFF_{n-|w|}(w) \}$, so
\begin{equation}\label{eqn:EE'}
\sum_{w\in \EEE} e^{-|w| h}
\geq \sum_{w\in \EEE'} e^{-|w| h}
\geq \sum_{w\in \EEE'} C_1^{-1} e^{-n h} \#\FFF_{n-|w|}(w)
= C_1^{-1} e^{-nh} \#\LLL_n.
\end{equation}
By \eqref{eqn:lcb}, this is at least $C_1^{-1}$. Taking an infimum over all $\EEE \in \EE^+(X^+,N)$ and sending $N\to\infty$ gives $\mf(X^+) \geq C_1^{-1}$, as claimed.

\subsubsection{Proof of Proposition \ref{prop:scaling}}

Let $a$ be a symbol in the alphabet, and consider a subset $Z \subset X^+ \cap \pbl a\pbr$.  If $\EEE \in \EE^+(Z,N)$, then $\sigma \EEE = \{\sigma w : w \in \EEE\}$ is a cover of $\sigma Z$, where we write $\sigma w = w_{[2,|w|]}$ for the word $w$ with its first symbol removed.  

Conversely, if $\FFF$ is a cover of $\sigma Z$, then the set $\{aw : w \in \FFF\}$ is a cover of $Z$.  Thus, $\sigma$ maps covers of $Z$ to covers of $\sigma Z$, and $\sigma|_{\pbl a\pbr}^{-1}$ maps covers of $\sigma Z$ to covers of $Z$. We conclude that $\sigma \colon \EE^+(Z,N) \to \EE^+(Z,N-1)$ is a bijection, and since
\begin{equation*}
\sum_{w\in\EEE} e^{-|w|h} = e^{-h} \sum_{w \in \EEE} e^{-(|w|-1)h} = e^{-h} \sum_{u \in \sigma \EEE} e^{-|u|h},
\end{equation*}
we can now deduce that
\begin{equation}
\inf\left\{\sum_{w \in \EEE} e^{-|w|h} : \EE^+(Z,N) \right\} = e^{-h} \inf\left\{ \sum_{u \in \FFF} e^{-|u|h} : \FFF \in \EE^+(\sigma Z, N-1)\right\}.
\end{equation}
Sending $N \to\infty$ gives $\mf(Z) = e^{-h} \mf(\sigma Z)$, and the proof for $\mb$ is similar.

\subsubsection{Proof of Proposition \ref{prop:Z-compact-lower}}

This proof follows the same idea as the lower bound for $\mf(X^+)$; we only need to modify \eqref{eqn:EE'} slightly. Given $n\in \NN$, let
\[
\LLL_n^Z := \{w \in \LLL_n : \pbl w \pbr \cap Z \neq \emptyset\},
\]
so that by the hypothesis, we have
\begin{equation}\label{eqn:LnZ-geq}
\#\LLL_n^Z \geq b_Z e^{nh} \quad\text{for all } n\in J.
\end{equation}
Now consider an arbitrary $N\in \NN$ and $\EEE \in \EE^+(Z,N)$.
Since $Z$ is compact, there exists a finite subcover $\EEE' \subset \EEE$.
Since $J \subset \NN$ is infinite, there exists $n\in J$ such that $n \geq |w|$ for all $w\in \EEE'$. Moreover, we have
\[
\LLL_n^Z \subset \bigsqcup_{w\in \EEE'} \{ wv : v\in \FFF_{n-|w|}(w) \},
\]
so we can use \eqref{eqn:Fnw} to obtain the following modification of \eqref{eqn:EE'}:
\[
\sum_{w\in \EEE} e^{-|w|h}
\geq \sum_{w\in \EEE'} e^{-|w| h}
\geq \sum_{w\in \EEE'} C_1^{-1} e^{-n h} \#\FFF_{n-|w|}(w)
= C_1^{-1} e^{-nh} \#\LLL_n^Z
\geq C_1^{-1} b_Z,
\]
where the last inequality uses \eqref{eqn:LnZ-geq}. Taking an infimum over all $\EEE \in \EE^+(Z,N)$ and sending $N\to\infty$ proves \eqref{eqn:Z-compact-lower}.

\subsubsection{Proof of Proposition \ref{prop:inv-MME}}

First, we prove that $\mu$ is $\sigma$-invariant by showing that $\mu(\sigma^{-1}Z) = \mu(Z)$ for every measurable $Z\subset X$. It suffices to consider the case when there exists $a\in A$ such that $Z \subset [a]$; the general case follows by linearity.

Recalling \eqref{eqn:u-cond}, we have
\begin{equation}\label{eqn:mu-sZ}
\mu(\sigma^{-1} Z) = \int_{X^-} \mf_y(\sigma^{-1} Z) \,d\mb(y).
\end{equation}
By the definition of $\mf_y$ in \eqref{eqn:mfx}, the integrand is
\begin{equation}\label{eqn:mfy}
\mf_y(\sigma^{-1} Z) = \mf(\pip( \Wul(y) \cap \sigma^{-1} Z)).
\end{equation}
Given $y\in X^-$, we see that 
\[
\pip( \Wul(y) \cap \sigma^{-1} Z) \subset \pbl y_0 \pbr,
\]
so we can use \eqref{eqn:+scale} from Proposition \ref{prop:scaling} together with \eqref{eqn:mfy} to obtain
\begin{equation}\label{eqn:mfy2}
\mf_y(\sigma^{-1} Z) = e^{-h} \mf(\sigma \pip(\Wul(y) \cap \sigma^{-1} Z)).
\end{equation}
Since $Z \subset [a]$, we have
\[
\sigma \pip(\Wul(y) \cap \sigma^{-1} Z)
= \pip(\sigma(\Wul(y) \cap \sigma^{-1} Z))
= \pip(\Wul(ya) \cap Z),
\]
which together with \eqref{eqn:mfy2} gives
\begin{equation}\label{eqn:mfy3}
\mf_y(\sigma^{-1} Z) = e^{-h} \mf(\pip(\Wul(ya) \cap Z))
= e^{-h} \mf_{ya}(Z).
\end{equation}
Substituting this in \eqref{eqn:mu-sZ}, we get
\begin{equation}\label{eqn:mu-sZ2}
\mu(\sigma^{-1} Z) = \int_{X^-} e^{-h} \mf_{ya}(Z) \,d\mb(y).
\end{equation}
We will show that $\mu(Z)$ is given by this same integral.
Define $\psi \colon X^- \to [0,\infty)$ by $\psi(x) = \mf_x(Z)$, and observe that $\psi(x) = 0$ unless $x \in \mbl a \mbr$,
so we can use the disintegration formula \eqref{eqn:mfx} and the scaling formula \eqref{eqn:-scale2} to obtain
\[
\mu(Z) = \int_{X^-} \mf_x(Z) \,d\mb(x)
= \int_{\mbl a \mbr} \psi(x) \,d\mb(x)
= e^{-h} \int_{\sigma^{-1}\mbl a\mbr} \psi(ya) \,d\mb(y).
\]
Since $\psi$ is supported on $\mbl a \mbr$, we see that
\[
\mu(Z) = e^{-h} \int_{X^-} \,\psi(ya) \,d\mb(y)
= e^{-h} \int_{X^-} \mf_{ya}(Z) \,d\mb(y),
\]
and comparing this with the right-hand side of \eqref{eqn:mu-sZ2} shows that $\mu(\sigma^{-1}Z) = \mu(Z)$, so $\mu$ is $\sigma$-invariant.

Finiteness of $\mu$, and the upper Gibbs bound, both follow quickly from the corresponding results in Proposition \ref{prop:nondegen} and Corollary \ref{cor:Gibbs}. Indeed, the first of these gives
\[
\mu(X) = \mt(X) \leq \mt(\tilde X) = \mb(X^-) \cdot \mf(X^+) \leq C_1^2,
\]
and more generally, for any $w\in \LLL_n$, 
we have $\cyl{}{w} = X \cap \Pi(X^-, \pbl w \pbr)$, so 
Corollary \ref{cor:Gibbs} gives
\[
\mu(\cyl{}{w}) = \mt(\cyl{}{w}) 
\leq \mb(X^-) \cdot \mf(\pbl w \pbr)
\leq C_1^2 e^{-nh}.
\]
Finally, to prove that $\bar\mu$ is a measure of maximal entropy when $\mu(X)>0$, we first observe that for every $n\in \NN$ and $w\in \LLL_n$, we have
\begin{equation}\label{eqn:bmu-Gibbs}
\bar\mu(\cyl{}{w}) = \mu(\cyl{}{w}) \mu(X)^{-1}
\leq \mu(X)^{-1} C_1^2 e^{-nh} = Q e^{-nh},
\end{equation}
where $Q := C_1^2 / \mu(X)$.
The function $t \mapsto -\log t$ is decreasing, so \eqref{eqn:bmu-Gibbs} gives
\[
\sum_{w\in \LLL_n} - \bar\mu\cyl{}{w} \log \bar\mu\cyl{}{w}
 \geq \sum_{w\in \LLL_n }-\bar\mu\cyl{}{w} \log (Q e^{-nh})
= nh -\log Q.
\]
Dividing both sides by $n$ and sending $n\to\infty$ proves that $h_{\bar\mu} \geq h$, so $\bar\mu$ is an MME.

\section{Application to dispersing billiards}\label{sec:bill}

Now we return to the billiard setting. Let $T\colon M\to M$ be the billiard map defined in \S\ref{sec:background}, 
let $\PPP = \{P_a : a\in A \}$ be the partition of $M = \partial\Omega\times[-\frac\pi2,\frac\pi2]$ into maximal connected sets on which $T$ and $T^{-1}$ are continuous, and let $c \colon M \to A^\ZZ$ be the associated coding map as in \S\ref{sec:coding-bill}.

In this section, we will consider the shift space $X = \overline{c(M)} \subset A^\ZZ$ and the measures $\mf$, $\mb$, and $\mu$ on $X$ defined in \S\S\ref{sec:one-sided}--\ref{sec:exterior-product}. We will prove the following.
\begin{itemize}
\item Proposition \ref{prop:1-1}: the coding map $c$ is injective; see Proposition \ref{prop:P-gen} in \S\ref{sec:coding}.
\item Theorem \ref{thm:strategy}, first part: every positive entropy ergodic measure on $X$ gives full weight to $c(M) \subset X$. We prove this in \S\ref{sec:pos-ent}, and actually establish a stronger statement: each such measure gives full weight to the subset $\Xreg \subset c(M)$ consisting of \emph{regular} points, defined in \eqref{eqn:Xreg-1}.
\end{itemize}

\subsection{Admissible curves and singularity curves}\label{sec:sing-sets}

Let $C^u$ and $C^s$ be the cones defined in \eqref{eqn:cones}. We will work with \emph{admissible curves} that are tangent to these cones. More precisely, an \emph{$s$-curve} is a continuous and piecewise $C^1$ curve $V \subset M$ whose tangent vectors all lie in the stable cone $C^s$. We write $\Vs$ for the set of all $s$-curves. Replacing ``stable'' with ``unstable'' gives $\Vu$, the set of all \emph{$u$-curves}.

In the $(r,\ph)$-coordinates on $M_i$, $s$-curves are graphs of decreasing functions, and $u$-curves are graphs of increasing functions.  Since our cones are bounded away from the horizontal, there is a uniform bound on the length of an $u/s$-curve: writing $L := \pi/\Kmin \in (0,\infty)$, we see that every $V\in \Vs \cup \Vu$ has length satisfying $|V| \leq L$.  Additionally, 
if $I \subset \partial \Omega$ is (homeomorphic to) an interval, then a $u$-curve and an $s$-curve in $I \times [-\frac\pi2,\frac\pi2] \subset M$ can intersect in at most one point.

Since $C^u$ is invariant under $DT$, the image $TV$ of a $u$-curve $V$ will again have tangent vectors in $C^u$, but it might consist of multiple $C^1$ curves if $V$ crosses $T^{-1}\SSS_0$, where we recall that $\SSS_0 = \partial\Omega \times \{\pm\frac\pi2\}$ is the set of grazing collisions from \eqref{eqn:S0}.

In general, given any $u$-curve $V\in \Vu$, the uniform length bound in the previous paragraph guarantees that for sufficiently large $n\geq 0$, the images $T^n V$ will eventually be ``cut" by singularities and fragment into a collection of $u$-curves, whose length can be arbitrarily small.  Likewise, the images $T^{-n}V$ of an $s$-curve $V \in \Vs$ will fragment.  We will need to guarantee that a proportion of these fragments are long.\footnote{This was important in \cite{BD20} as well, and we use many of the same tools and concepts.} 
We will denote the $u/s$-curves that have length at least $\delta$ by
\begin{equation}\label{eqn:Vud}
	\Vud := \{V \in \Vu : |V| \geq \delta\}
\quad\text{and}\quad
	\Vsd := \{V \in \Vs : |V| \geq \delta\}.
\end{equation}
The following result is standard \cite[\S4.4]{CM06}.

\begin{lemma}[Uniform expansion and contraction]\label{lem:s-contr}
There exist $C_2,\lambda>0$ such that given any $V \in \Vu$ and $n\geq 0$, we have
\begin{equation}\label{eqn:u-contr}
|V| \leq C_2 e^{-\lambda n} |T^n V|,
\end{equation}
where $|T^n V|$ denotes the sum of the lengths of the component curves of $T^n V$. An analogous bound holds for $|T^{-n} V|$ when $V\in \Vs$.
\end{lemma}

We will also use the fact that writing $\SSS_0$ for the singularity set as in \eqref{eqn:S0}, $\SSS_1 := T^{-1} \SSS_0$ is a union of $s$-curves, and $\SSS_{-1} := T \SSS_0$ is a union of $u$-curves \cite[Proposition 4.45, Exercise 4.46]{CM06}.

The complement of the singularity set represents the non-grazing collisions:
\begin{equation}
M^\circ := M \setminus \SSS_0 = \partial\Omega \times \Big( -\frac\pi2, \frac\pi2 \Big).
\end{equation}
Given $n\in \NN$, the singular sets of $T^n$ and of $T^{-n}$ are
\[
\SSS_n := \bigcup_{k=0}^n T^{-k} \SSS_0
\quad\text{and}\quad
\SSS_{-n} := \bigcup_{k=0}^n T^k \SSS_0,
\]
respectively. It will also be convenient to write
\begin{equation}\label{eqn:Sn'}
\SSS_n' := \bigcup_{k=1}^n T^{-k} \SSS_0
= \overline{\SSS_n \setminus \SSS_0}
\quad\text{and}\quad
\SSS_{-n}' := \bigcup_{k=1}^n T^{k} \SSS_0
= \overline{\SSS_{-n} \setminus \SSS_0}.
\end{equation}
The set $\SSS_n'$ consists of $s$-curves, and $\SSS_{-n}'$ consists of $u$-curves.

\begin{definition}
Given $k\in \NN$, we say that a (connected) $u$-curve $V \in \Vu$ is a \emph{component curve} of $T^{k}\SSS_0$ if it can be written as $V = \overline{T^{k}S}$ for some curve $S \subset \SSS_0$ on which $T^k$ is continuous,
and if it is not properly contained in any longer curve of this form. Let $\bbS_k$ denote the set of component curves of $T^{k}\SSS_0$; see Figure \ref{fig:Sk-components}. Equivalently, we have
\begin{equation}\label{eqn:Sk}
\bbS_k = \{ \overline{T^kS} : S \text{ is a connected component of } \SSS_0 \setminus \SSS_k' \}.
\end{equation}
By a ``component of $\SSS_{-n}$'', we will mean a component curve of $T^{k}\SSS_0$ for some $0\leq k \leq n$, and we will refer to $k$ as the \emph{order} of the curve.  Note that for $k=0$, the entire circular boundary is the component.
\end{definition}
Similar definitions can be formulated for components of $T^{-k}\SSS_0$, with the appropriate modifications.

Components of singularity curves play an important role in the counting estimates in \S\ref{sec:counting}. Here we will present some results about the singularity curves, starting with the following fact from \cite[Proposition 4.45 and Exercise 4.46]{CM06}:

\begin{lemma}\label{lem:fin}
For every $k\in \NN$, the sets $\bbS_k$ and $\bbS_{-k}$ are finite.
\end{lemma}

\begin{figure}[htbp]
\centering
\includegraphics{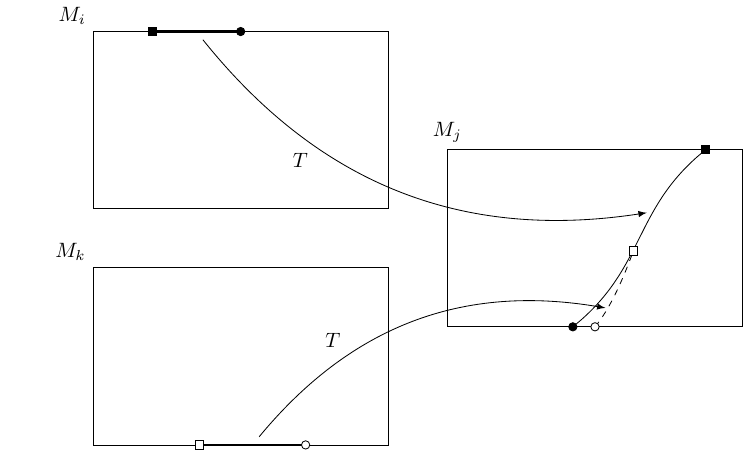}
\caption{Elements of $\bbS_1$ are component curves of $T\SSS_0$.}
\label{fig:Sk-components}
\end{figure}

The next facts about $\bbS_k$ are used in \cite{BD20}, but we highlight them more explicitly here. In what follows, given a curve $W$, we will write $\partial W$ for the set consisting of the endpoints of $W$, and $W^\circ := W \setminus \partial W$ for the `interior' of the curve.

The first lemma restricts the ways in which components of $\SSS_{-n}$ can intersect: informally, it says that
\begin{center}
\emph{higher-order curves terminate at intersections}.
\end{center}
Formally, we have:

\begin{lemma}[Noncrossing]\label{lem:noncrossing}\hfill
\begin{enumerate}[label=(\alph{*}),leftmargin=*]
\item If $V,W \in \bbS_n$ are components of $T^n\SSS_0$, then $V \cap W \subset \partial V \cup \partial W$: any intersection point $\bx$ must be an endpoint of either $V$ or $W$.
\item
If $W \in \bbS_n$ is a component of $T^{n}\SSS_0$ and $\bx \in W \cap T^{k} \SSS_0$ for some $0 \leq k < n$, then $\bx \in \partial W$.
\end{enumerate}
\end{lemma}
\begin{proof}
For the first claim, suppose that $\bx \in W^\circ \cap V$. Since $W^\circ = T^nS$ for some connected component $S$ of $\SSS_0 \setminus \SSS_n'$, there exists an open set $U\subset M$ containing $T^{-n}\bx$ such that $U \cap \SSS_n' = \emptyset$. At the same time, $V = \overline{T^nS'}$ for some other connected component $S'$ of $\SSS_0 \setminus \SSS_n'$, which must be disjoint from $U$, so $V^\circ = T^nS' \not\ni \bx$.

For the second claim, observe that $T^{-k}\bx \in \SSS_0$, and $T^{-k}W$ is an unstable curve.  This implies that $T^{-k} W$ terminates on $\SSS_0$, so $\bx \in \partial W$.
\end{proof}

The following result is a consequence of the fact that the scatterers are mutually disjoint: informally, it says that
\begin{center}
\emph{endpoints of a component curve are interior points of another component curve.}
\end{center}

\begin{lemma}[Continuation]\label{lem:continuation}
If $W\in \bbS_n$ is a component of $T^{n}\SSS_0$ and $\bx \in \partial W$, then $\bx \in V^\circ$ for some component $V$ of $\SSS_{-n}$.
In particular, for every $W\in \bbS_n$, there exists a $u$-curve $V \subset \SSS_{-n}$ such that $W\subset V$ and $\partial V \subset \SSS_0$.
\end{lemma}
\begin{proof}
Consider the case when $n=1$.  Suppose $\bx \in M_j^\circ$ for some $j$.  By definition, there is an $i$ such that $T^{-1}W^\circ \subset \partial M_i$.  If $W$ terminates in the interior of $M_j$, then there must be another scatterer $B_k$ lying between $B_i$ and $B_j$ that blocks some of the tangential collisions coming from $B_i$ (there are no cusps or corners), as shown in Figure \ref{fig:multiple-singularity}.  

There could be multiple scatterers that are tangent to the trajectory from $B_i$ to $B_j$ associated to $T^{-1}\bx$, but one of them must be closest to $B_j$; we will call this scatterer $B_k$. Because $\bx \in \partial W$, it follows that $T^{-1}\bx \in \partial M_k$.  There must be a $V \subset \SSS_0$ such that $T^{-1}\bx \in V^\circ$ and we get $TV \subset M_j$ and $\bx \in (TV)^\circ$.

We proceed by induction, and assume the statement holds for $n\geq 1$.  Suppose $W$ is a component of $T^{n+1}\SSS_0$ and $\bx \in \partial W$.  Thus, $T^{-n}\bx \in \partial T^{-n}W$.  If $T^{-n}\bx \in M^\circ$, then we are done, as the argument reduces to the $n=1$ case.

If $T^{-n}\bx \in \partial M$, then there is a component $V$ of $\SSS_{-n}$ such that either $\bx \in V^\circ$ or $\bx \in \partial V$.  If $\bx \in V^\circ$, then we are done.  If $\bx \in \partial V$, then by the inductive hypothesis, there must be some other component $V'$ such that $\bx$ is in the interior of $V'$.
\end{proof}

\begin{definition}\label{def:multiple}
A point in $M$ is a \emph{multiple singularity point} for $\SSS_n$ if it lies in multiple components of $\SSS_n$. We denote the set of such points by $\SSS^*_n$.
\end{definition}

By Lemmas \ref{lem:noncrossing} and \ref{lem:continuation}, the set of multiple singularity points for $\SSS_n$ is exactly the set of endpoints of component curves:
\begin{equation}\label{eqn:Sn*}
\SSS^*_n = \partial \SSS_n := \bigcup_{k=0}^n \bigcup_{W\in \bbS_k} \partial W.
\end{equation}
By Lemma \ref{lem:fin}, the set $\SSS^*_n$ is finite.

\begin{figure}
\includegraphics[width=0.95\textwidth]{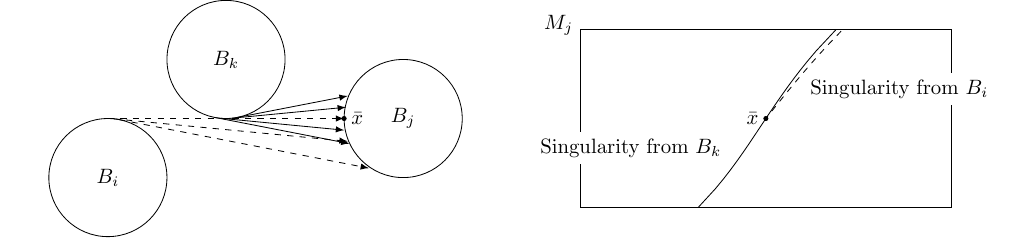}
\caption{Tangential collisions blocked by a scatterer.}\label{fig:multiple-singularity}
\end{figure}

Following \cite{BD20}, we will use the next result from \cite{BCS91,nC01} without repeating the proof here:

\begin{lemma}[Linear complexity]\label{lem:lin-comp}
There exists $K>0$ such that for every $n\geq 0$, the number of curves in $\SSS_n$ intersecting at a single point is at most $Kn$, and similarly for $\SSS_{-n}$.
\end{lemma}

\begin{corollary}\label{cor:local-lin}
Let $K$ be as in Lemma \ref{lem:lin-comp}.
For every $n\in \NN$, there exists $\delta>0$ such that for every $\bx \in M$, the set $B(\bx,\delta) \setminus \SSS_n$ has at most $Kn+1$ connected components, and similarly for $B(\bx,\delta)\setminus \SSS_{-n}$.
\end{corollary}
\begin{proof}
Let $r_n$ be the shortest distance between a multiple singularity point for $\SSS_n$ and a component curve of $\SSS_n$ that does not contain it; observe that $r_n>0$ because $\SSS_n^*$ and $\bigcup_{k=0}^n \bbS_{-k}$ are finite by Lemma \ref{lem:fin}.

Let $\delta_n \in (0,r_n/3)$, so that a $\delta_n$-neighborhood of a multiple singularity point $\bx$ only intersects components of $\SSS_n$ that are coincident at $\bx$. 

There exists $\delta \in (0, \delta_n]$ such that for any
$\bill{y} \in M \setminus \bigcup_{\bx \in \SSS_n^*} B(\bx,\delta_n)$, the set
$B(\bill{y},\delta) \setminus \SSS_n$ has at most $2$ connected components.

Now suppose $\bill{y} \in B(\bx,\delta_n)$ for some $\bx \in \SSS_n^*$.  Observe that because $\delta \leq \delta_n < r_n/3$, the open ball $B(\bill{y},\delta)$ can only intersect the singularity curves that are coincident at $\bx$.  Since there are at most $Kn$ of these curves, $B(\bill{y},\delta)\setminus \SSS_n$ has at most $Kn+1$ connected components.
\end{proof}

\subsection{Injectivity of the coding map}\label{sec:coding}

Given $k,n \in \NN_0 \cup \{\infty\}$, consider the set
\begin{equation}\label{eqn:Mkn}
\begin{aligned}
M_{-k}^n &:= M \setminus (\SSS_{-k} \cup \SSS_n)
= M \setminus \bigcup_{j=-k}^n T^{-j} \SSS_0
= \bigcap_{j=-k}^n T^{-j} M^\circ
 \\
&= \{ \bx \in M : T^j \bx \notin \SSS_0 \text{ for all } j\in \ZZ \text{ such that } -k \leq j \leq n \}
\end{aligned}
\end{equation}
of all points that have no grazing collision between times $-k$ and $n$.  Equivalently, this is the set of points at which $T^{-k}$ and $T^n$ are continuous. 
Following \cite{BD20}, we write $\MMM_{-k}^n$ for the partition of $M_{-k}^n$ into its connected components, which we call \emph{cells}.

Recall from \S\ref{sec:coding-bill} that we code trajectories of the billiard system in terms of the partition $\PPP$ whose elements are the maximal connected sets on which $T$ and $T^{-1}$ are continuous; see Figure \ref{fig:M-partition}.  Elements of $\MMM_{-1}^1$ are the maximal \emph{open} connected sets on which $T$ and $T^{-1}$ are continuous.  Thus, a cell of $\MMM_{-1}^1$ is the interior of some cell in $\PPP$ \cite[Lemma 3.3]{BD20}.

We will show that the coding map $c \colon M \to X$ is injective by showing that the diameter of cells in $\MMM_{-k}^n \to 0$ as $k,n\to\infty$.  Note that $c$ may not be surjective, but it still preserves the entropy $\htop(\sigma) = \htop(T)$.

Recall that $A$ is the index set for $\PPP$, and let $A^\circ$ be the indices $a\in A$ for which the corresponding element of $\PPP$ has nonempty interior $Y_a$.
Observe that
\[
M_{-1}^1 = \bigsqcup_{a\in A^\circ} Y_a.
\]
Write $\QQQ = \MMM_{-1}^1$ for the partition of $M_{-1}^1$ into its open connected components $Y_a$.
The sets $Y_a$ are the maximal open connected sets on which $T$ and $T^{-1}$ are both continuous.

We claim that the connected components of $M_{-k}^n$ are exactly the cells of the partition $\bigvee_{j=-k}^n (T^{-j}\QQQ)|_{M_{-k}^n}$. First we establish some notation, which should be compared with the cylinder notation in \eqref{eqn:cyl} and \eqref{eqn:vw-cyl}. Given $n\in \NN$ and $w\in A^n$, consider the set
\begin{equation}\label{eqn:Yw+}
Y_w^+ := \bigcap_{j=1}^n T^{-(j-1)} Y_{w_j}.
\end{equation}
This set could be empty, so we write
\begin{equation}\label{eqn:bill-lang}
\LLL_n := \{ w\in A^n : Y_w^+ \neq \emptyset \}
\quad\text{and}\quad \LLL = \bigcup_{n=0}^\infty \LLL_n.
\end{equation}
We also write
\begin{equation}\label{eqn:Yw-}
Y_w^- := T^{n-1}Y_w^+ = \bigcap_{i=1}^n T^{n-i} Y_{w_i}.
\end{equation}
Given $v\in \LLL_k'$ and $w\in \LLL_n$, we write
\begin{equation}\label{eqn:Yvw}
Y_{v.w} := Y_v^- \cap Y_w^+,
\end{equation}
observing that this set is empty whenever $v_k \neq w_1$, just as with $\cyl{v}{w}$. 
Writing $v' := v_1 v_2 \cdots v_{k-1}$ as in \eqref{eqn:Lkn}, we see that whenever $v_k = w_1$, we have
\begin{equation}\label{eqn:Yvw-again}
Y_{v.w} = T^{n-1}Y_{v'w}^+ = T^{-(k-1)}Y_{v'w}^-.
\end{equation}

\begin{remark}\label{rmk:X-lang}
Note that $\LLL$ is not necessarily the entire language of $X$.  The language of the shift space $X = \overline{c(M)}$ may actually be larger than $\LLL$ due to the presence of isolated points, as discussed in \cite[Lemma 3.2]{BD20}. We will see in \S\ref{sec:pos-ent} that the set of such points is a null set for any positive entropy ergodic measure.  Indeed, if $\hat\LLL$ is the complete language of $X$, then $\#(\hat\LLL_n \setminus \LLL_n)$ grows at most linearly \cite[Lemma 3.2]{BD20},
which implies that 
\[
\htop(\sigma)=\lim_{n\to\infty} \frac 1n \log(\#\LLL_n).
\]
For the majority of our arguments, it is sufficient to use the partition $\QQQ$ and the language $\LLL$ from \eqref{eqn:bill-lang}.  In particular, the counting estimates in \S\ref{sec:counting}, and the arguments in 
Lemma \ref{lem:positive-leaf} and Proposition \ref{prop:full-support} that the measures $m^{\pm}$ and $\mu$ are fully supported, only apply to the collection $\LLL$.  However, we will use the partition $\PPP$ in \S\ref{sec:rectangles}, and in \S\ref{sec:pos-ent} when we consider arbitrary positive entropy ergodic measures for the billiard map.
\end{remark}

\begin{lemma}\label{lem:Y-open}
For every $w\in \LLL$, the sets $Y_w^{\pm}$ are open and nonempty, and for every $v,w\in \LLL$ such that $v_{|v|} = w_1$ and $v'w \in \LLL$, the set $Y_{v.w}$ is open and nonempty.
\end{lemma}
\begin{proof}
Each $Y_w^+$ is nonempty by definition. To see that they are open, observe that $Y_a$ is open by definition, and if $w\in \LLL_n$ is such that $Y_w^+$ is open, then for every $a\in A$ we have
\[
Y_{aw}^+ = Y_a \cap T^{-1}Y_w^+ = T|_{Y_a}^{-1}Y_w^+.
\]
This is open since $T|_{Y_a}$ is continuous. By induction in $n$, we conclude that $Y_w^+$ is nonempty and open for every $w\in \LLL$.

The result for $Y_w^-$ follows from \eqref{eqn:Yw-} and the fact that $T^{n-1}|_{M_{-1}^{n}}$ is a homeomorphism from $M_{-1}^n$ to $M_{-n}^1$. Similarly, the result for $Y_{v.w}$ follows from \eqref{eqn:Yvw} and the fact that $T^{n-1}|_{M_{-1}^{n+k-1}}$ is a homeomorphism from $M_{-1}^{n+k-1}$ to $M_{-k}^n$.
\end{proof}

Although $\LLL$ is not the entire language of the shift $X$, we still have the following:

\begin{lemma}\label{lem:lang}
The collection $\LLL \subset A^*$ is a factorial and extendable language.
\end{lemma}
\begin{proof}
First we show that $\LLL$ is factorial.
Given any $n\in \NN$, $w\in \LLL_n$, and $i,j \in \{1,\dots, n\}$ with $i\leq j$, let $u = w_{[1,i]}$ and $v = w_{[i,n]}$. Then $w = u'v$ and \eqref{eqn:Yvw} gives
\[
Y_{w_{[i,j]}}^+ \supset Y_{w_{[i,n]}}^+
= Y_v^+ \supset Y_{u.v} = T^{n-i}Y_{u'v}^+
= T^{n-i}Y_w^+ \neq \emptyset,
\]
so $w_{[i,j]} \in \LLL$, and we conclude that $\LLL$ is factorial.

To show that $\LLL$ is extendable, observe that $T^{-1}Y_w^+$ is an open subset of $M$, and $\MMM_{-1}^1 = \bigcup_{a\in A} Y_a$ is dense in $M$, so there exists $a\in A$ such that $Y_a \cap T^{-1}Y_w^+ \neq \emptyset$, hence $aw\in \LLL$. A similar argument shows that there is $b\in A$ such that $wb\in \LLL$.
\end{proof}

\begin{remark}
As in \eqref{eqn:X-from-L}, we could define a two-sided shift space $\check{X} \subset A^\ZZ$ by
\begin{equation}\label{eqn:X-bill}
	\check{X} := \Big\{ x\in A^\ZZ : Y_{x_{[-k,0]},x_{[0,n]}}
	= \bigcap_{i=-k}^n T^{-i}Y_{x_i} \neq \emptyset
	\text{ for all } k,n\in \NN \Big\}.
\end{equation}
Moreover, if we let $X' = c(M')$, then $\check{X} = \overline{X'}$.  However, $\check{X}$ may be properly contained in $X$, because $\check{X}$ does not contain any of the isolated points that $X$ has.
\end{remark}

To describe the relationship between the shift space $X$ and the billiard map, we need more information about the sets $Y_w^{\pm}$ and $Y_{v.w}$. Recall from \eqref{eqn:Lkn} that we write $\LLL_{k,n}$ for the set of pairs $(v,w) \in \LLL_k\times \LLL_n$ such that $Y_{v.w} \neq \emptyset$.

\begin{proposition}\label{prop:P-gen}
For every $w \in \LLL$, the sets $Y_w^\pm$ are connected. The collection $\{ Y_w^+ : w\in \LLL_n\}$ gives the connected components of $M_{-1}^n$, and $\{ Y_w^- : w\in \LLL_n \}$ gives the connected components of $M_{-n}^1$.

The set $Y_{v.w}$ is connected for all $(v,w) \in \LLL_{k,n}$, and such sets are exactly the connected components of $M_{-k}^n$.
Moreover, there exists a constant $C_3 > 0$ such that for every $k,n\in \NN$ and $(v,w) \in \LLL_{k,n}$, the diameter of $Y_{v.w}$ is at most $C_3 e^{-\lambda \min(k,n)}$.
\end{proposition}
\begin{proof}
We will prove that $T^{-1}Y_w^+$ is a connected component of $M_0^{n+1}$ for every $n\in \NN$ and $w\in \LLL_n$. The remaining claims concerning $Y_w^{\pm}$ follow immediately from this result and from the fact that $T^{-(n+1)}$ is continuous on $M_0^{n+1}$. This will also imply the claims concerning $Y_{v.w}$, 
where the diameter estimate uses Lemma \ref{lem:s-contr}.

Now we turn our attention to the proof that $T^{-1}Y_w^+$ is a connected component of $M_0^{n+1}$ for all $n\in \NN$ and $w\in \LLL_n$.
We start with the observation that $P \subset M_0^{n+1}$ is a connected component of $M_0^{n+1}$ if and only if there exists a closed interval\footnote{Throughout this proof, we use ``interval'' to mean ``a set homeomorphic to an interval''.} $I \subset \partial\Omega$ and two continuous nonincreasing functions $\ph^{0,1} \colon I\to [-\frac\pi2,\frac\pi2]$ such that the following are true:
\begin{itemize}
\item the graphs of $\ph^0$ and $\ph^1$ are subsets of $\SSS_{n+1}$;
\item for every $r\in I$, we have $\ph^0(r) \leq \ph^1(r)$, with equality if and only if $r$ is an endpoint of $I$;
\item the set $P$ is the region between the graphs of these functions:
\begin{equation}\label{eqn:P-graphs}
P = \{ (r,\ph) : r\in I,\
\ph^0(r) < \ph < \ph^1(r) \}.
\end{equation}
\end{itemize}
Armed with this characterization, we will prove 
the result for $T^{-1}Y_w^+$ by induction in $n$.

For the case $n=1$, it follows from the definition of $Y_a$ that $T^{-1}Y_a$ is a connected component of $M_0^2$ for every $a\in A$. For the inductive step, fix $w\in \LLL_n$ and suppose that $T^{-1}Y_w^+$ is a connected component of $M_0^{n+1}$. Let $I_w$ and $\ph_w^{0,1}$ be as described above, so that \eqref{eqn:P-graphs} becomes
\begin{equation}\label{eqn:Iw-graphs}
T^{-1}Y_w^+ = \{ (r,\ph) : r\in I_w,\
\ph_w^0(r) < \ph < \ph_w^1(r) \}.
\end{equation}
We will describe $Y_a \cap T^{-1}Y_w^+$ whenever $a\in A$ is such that $aw\in \LLL$.

Since $Y_a$ is a connected component of $M_{-1}^1$, there exists a closed interval $I_a \subset \partial\Omega$ and continuous functions $\ph_a^{0,1}, \psi_a^{0,1} \colon I_a \to [-\frac\pi2,\frac\pi2]$ such that
\begin{itemize}
\item the functions $\ph_a^{0,1}$ are nonincreasing, and their graphs are subsets of $\SSS_1$;
\item the functions $\psi_a^{0,1}$ are nondecreasing, and their graphs are subsets of $\SSS_{-1}$;
\item for every $r\in I_a$, we have $\max(\ph_a^0(r),\psi_a^0(r)) \leq \min(\ph_a^1(r),\psi_a^1(r))$, with equality if and only if $r$ is an endpoint of $I_a$;
\item the set $Y_a$ is given by
\begin{equation}\label{eqn:Ya}
Y_a = \{(r,\ph) : r\in I_a,\
\max(\ph_a^0(r),\psi_a^0(r)) < \ph < \min(\ph_a^1(r),\psi_a^1(r)) \}.
\end{equation}
\end{itemize}
Combining \eqref{eqn:Iw-graphs} and \eqref{eqn:Ya}, the intersection $Y_a \cap T^{-1}Y_w^+$ can be written as
\[
\{ (r,\ph) : r\in I_w \cap I_a,\
\max(\ph_a^0,\psi_a^0,\ph_w^0)(r) < \ph
< \min(\ph_a^1,\psi_a^1,\ph_a^1)(r) \}.
\]
The graphs of $\phi_a^{0,1}$ are subsets of $\SSS_{1}$, so they are disjoint from $T^{-1}Y_w^+ \subset M_0^{n+1}$. This implies that $\ph_a^0 \leq \ph_w^0$ and $\ph_a^1 \geq \ph_w^1$ on $I_w \cap I_a$, and we conclude that
\[
Y_a \cap T^{-1}Y_w^+
= \{ (r,\ph) : r\in I_w \cap I_a,\
\max(\psi_a^0,\ph_w^0)(r) < \ph < \min(\psi_a^1,\ph_w^1)(r) \}.
\]
If $I_w \cap I_a$ were to be disconnected, then $Y_a \cap T^{-1}Y_w^+$ would be disconnected as well; see Remark \ref{rmk:non-interval} below for a (different) choice of coding in which this can occur. For our present coding, we will show that $I_w \cap I_a$ is an interval in $\partial\Omega$, which will then imply that $Y_a \cap T^{-1}Y_w^+$ is a connected component of $M_{-1}^n$, completing the inductive step.

To prove that $I_w \cap I_a$ is an interval, we observe that there exist $i,j\in \{1,\dots, D\}$ such that $Y_a$ and $T^{-1}Y_w^+$ lie in the same connected component of $M_i \cap T^{-1}M_j$. 
Every $\bx$ in this connected component corresponds to a billiard trajectory that hits $B_i$ and then $B_j$. 
By connectedness, 
there are lifts $\tilde{B}_i,\tilde{B}_j \subset \RR^2$ such that this set of trajectories lifts to a set of line segments in $Z := \RR^2 \setminus (\tilde{B}_i^\circ \cup \tilde{B}_j^\circ)$ connecting a point on $\partial\tilde{B}_i$ to a point on $\partial\tilde{B}_j$.
By convexity of $\tilde{B}_i$ and $\tilde{B}_j$, 
there exists a line in $Z \subset \RR^2$ separating them.
Choosing $\tilde{r}_0 \in \partial\tilde{B}_i$ such that the tangent line to $\partial\tilde{B}_i$ at this point 
is parallel to the separating line, we see that no line segment in $Z$ connecting $\tilde{B}_i$ to $\tilde{B}_j$ has $\tilde{r}_0$ as its endpoint. 

In particular, we conclude that 
the connected component of $M_i \cap T^{-1}M_j$ that contains $Y_a$ and $T^{-1}Y_w^+$ is a subset of $(\partial B_i \setminus\{r_0\}) \times [-\frac\pi2,\frac\pi2]$.
This implies that $I_a,I_w \subset (\partial B_i) \setminus \{r_0\}$.  Since this latter set is an interval and so are $I_a,I_w$, it follows that $I_a \cap I_w$ is an interval as well.
\end{proof}

\begin{remark}\label{rmk:non-interval}
If we coded by cells of $\MMM_{-1}^0$ or of $\MMM_0^1$ 
instead of $\MMM_{-1}^1$, then the above proof would fail in the last paragraph, because $I_w \cap I_a$ might not be an interval. Figure \ref{fig:not-generating} shows
what could happen, illustrating a cell of $\MMM_{-1}^0$ and of $\MMM_0^1$ whose intersection is not connected.\footnote{This does not rule out the possibility that the partition into cells of $\MMM_0^1$ could be generating, but a different argument would be needed.}
\end{remark}

\begin{figure}[!tbp]
\centering
\includegraphics[width=0.4\textwidth]{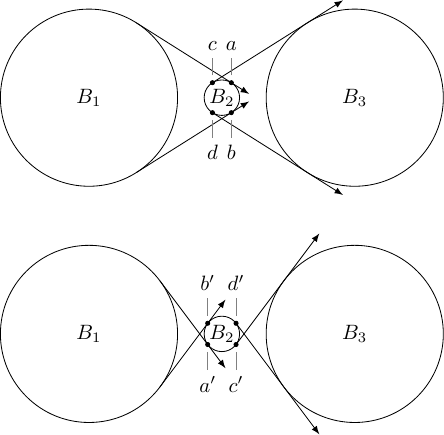}
\hfill
\includegraphics[width=0.48\textwidth]{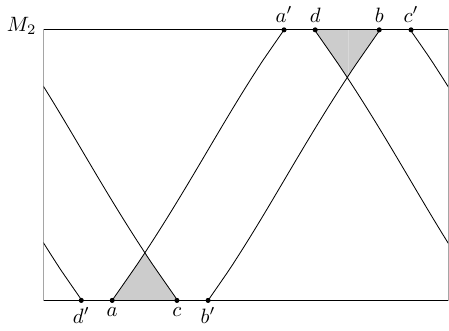}\caption{The shaded region is the intersection of a cell of $\MMM_{-1}^0$ and a cell of $\MMM_0^1$.}\label{fig:not-generating}
\end{figure}

Recalling the definition of the shift space $X\subset A^\ZZ$ in \S\ref{sec:coding-bill} as the closure of $c(M)$, we now define the map
\begin{equation}\label{eqn:piXM}
\pi \colon X \to M
\end{equation}
that we will work with throughout the rest of the paper:

\begin{definition}\label{def:coding}
Start by defining $\pi|_{c(M)} := c^{-1}$.
The diameter estimate in Proposition \ref{prop:P-gen} shows that $\pi$ is uniformly continuous, so it has a unique continuous extension to $X = \overline{c(M)}$.
In particular, given any $x \in X$,
there exists a sequence of points $x^n \in X'$, such that $x^n \to x$, and we have $\pi(x) = \lim_n c^{-1}(x^n)$, where the limit is independent of the choice of sequence.
Another equivalent description is that $\pi(x)$ is the unique point in
\[
\overline{\textstyle\bigcap_{j=-n}^n T^{-j} Y_{x_j}}
= \overline{Y_{x_{[-n,0]}.x_{[0,n]}}}.
\]
We will use the notational convention that unadorned symbols such as $x,y,z$ will denote elements of symbolic space $A^\ZZ$, while $\bx = \pi(x)$, $\bill{y}= \pi(y)$, $\bill{z} = \pi(z)$ will denote the points in the phase space whose trajectory they code.
\end{definition}

\begin{remark}\label{rmk:not-conj}
The map $\pi$ is \emph{not} a semiconjugacy between the shift map $\sigma \colon X\to X$ and the billiard map $T\colon M\to M$.
This can be seen by visualizing the boundary between two cells in $\MMM_{-1}^1$ in terms of the billiard table itself, and looking at the non-grazing trajectories that limit on a grazing trajectory as in Figure \ref{fig:limit-grazing}.  
The coding of the grazing trajectory by $abc$ involves two applications of both $\sigma$ and $T$, but the coding by $ac$ involves a single application of $\sigma$ that corresponds to two applications of $T$.
\end{remark}

\begin{figure}[htbp]
\includegraphics[scale=1]{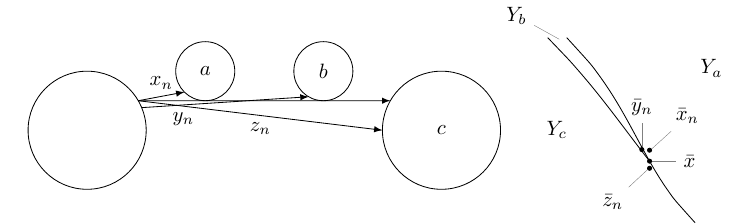}
\caption{Nongrazing collisions can limit on a grazing collision, resulting in tangential trajectories with multiple codings, some of which ``miss'' one or more symbols in the sequence.}
\label{fig:limit-grazing}
\end{figure}

In addition to the problem pointed out in Remark \ref{rmk:not-conj}, the map $\pi$ is not injective on $X$.
The second part of Figure \ref{fig:limit-grazing} illustrates how injectivity can fail: the three regions represent a situation in which three cells of $\MMM_{-1}^1$ have boundaries that intersect in a single point $\bx \in M$; denoting these cells by $Y_a$, $Y_b$, and $Y_c$, where $a,b,c\in A$, there are sequences $\bx_n \in Y_a$, $\bill{y}_n \in Y_b$, and $\bill{z}_n \in Y_c$ that all converge to $\bx$, so $\bx \in \pi[a] \cap \pi[b] \cap \pi[c]$.

Fortunately, both of the problems described above -- failure of semiconjugacy, and failure of injectivity -- can be avoided by passing to a subset that avoids grazing collisions.
Consider the set
\[
M' := M_{-\infty}^\infty 
= M \setminus \bigcup_{n \in \ZZ}\SSS_{n}
= \bigcap_{n\in \ZZ} T^{-n} M^\circ,
\]
which consists of all points in $M$ that never have a grazing collision in their past or future.  
Given a point $\bx\in M'$, we have $T^n \bx \in M_{-1}^1$ for every $n\in \ZZ$, and thus for each $n$ there exists a unique $x_n \in A$ such that $T^n \bx \in Y_{x_n}$. The sequence $x = (x_n)_{n\in \ZZ} \in X \subset A^\ZZ$ is the unique preimage of $\bx$ under $\pi$, and we refer to it as the \emph{coding} of $\bx$. Writing
\[
X' := \pi^{-1}(M') = \{ x \in X : \pi(x) \in M' \},
\]
we see that $\pi|_{X'}$ is injective, and that the diagram
\begin{equation}\label{eqn:comm}
\begin{tikzcd}
X' \arrow[r, "\sigma"] \arrow[d, "\pi"] & X' \arrow[d, "\pi"] \\
M' \arrow[r, "T"] & M'
\end{tikzcd}
\end{equation}
commutes. In particular, $\pi_*$ is an entropy-preserving bijection from the space of $\sigma$-invariant Borel probability measures on $X'$ to the space of $T$-invariant Borel probability measures on $M'$. It becomes important, then, to understand how `large' the sets $M'$ and $X'$ are.
The sets $\SSS_{\pm\infty}$ are dense in $M$, so $M'$ has no interior; on the other hand, we will see that every MME gives full weight to $M'$.

With the symbolic coding in place, our strategy for the remainder of the proof is as follows:
\begin{description}[align=left, leftmargin=0pt, itemsep=1ex]
\renewcommand{\makelabel}[1]{\textbf{#1}}
\item[\fancymarker{\S\ref{sec:rectangles}}]
Describe stable and unstable manifolds for the billiard map via the shift space, along with an important result on families of rectangles.
\item[\fancymarker{\S\ref{sec:pos-ent}}]
Prove the first part of Theorem \ref{thm:strategy} by showing that every positive entropy ergodic measure on $X$ gives full weight to $X' \subset c(M)$.
(In fact, we prove a stronger result involving the \emph{regular set}.)
\item[\fancymarker{\S\ref{sec:counting}}]
Verify the counting bounds necessary to apply Proposition \ref{prop:inv-MME}, concluding that the invariant measure $\mu$ constructed as in \S\ref{sec:construction} satisfies $0 < \mu(X) < \infty$ and that the normalization of $\mu$ is an MME. This proves part of Theorem \ref{thm:strategy}\ref{mu-fin-pos}.
\item[\fancymarker{\S\ref{sec:bill-construction}}]
Prove the rest of Theorem \ref{thm:strategy} by showing that $\pi_*\mu$ is fully supported on $M$, and its normalization is multiply mixing and is the unique MME on $M$.
\end{description}

\subsection{Stable manifolds, unstable manifolds, and rectangles}\label{sec:rectangles}

Consider once again the partition $\PPP$ of $M$ into maximal sets on which $T$ and $T^{-1}$ are continuous.
Given $k,n\in \NN_0$, we will consider the dynamically refined partitions
\begin{equation}\label{eqn:Pkn}
\PPP_{-k}^n := \bigvee_{j=-k}^n T^{-j}\PPP
\end{equation}
into maximal sets on which $T^{-k}$ and $T^n$ are continuous. 
Given any element $P \in \PPP_{-k}^n$, there exists an open connected set $Q \in \MMM_{-k}^n$ such that $Q \subset P \subset \overline{Q}$, so $P$ is connected as well; see also \cite[Lemma 3.1]{BD20}.

We will continue to use the notation \eqref{eqn:Pkn} when one of $k,n$ is infinite, writing
\begin{equation}\label{eqn:P-infty}
\PPP_0^\infty := \bigvee_{j=0}^\infty T^{-j} \PPP
\quad\text{and}\quad
\PPP_{-\infty}^0 := \bigvee_{j=-\infty}^0 T^{-j} \PPP.
\end{equation}
Given $\bx \in M$, let $\PPP(x)$ denote the cell of $\PPP$ that contains $\bx$, and consider the sets
\begin{equation}\label{eqn:Wusl}
\begin{aligned}
\Wsl(\bx) &:= \{ \bill{y} \in M : \PPP(T^n\bill{y}) = \PPP(T^n \bx) \text{ for all } n\geq 0 \}, \\
\Wul(\bx) &:= \{ \bill{y} \in M : \PPP(T^{-n}\bill{y}) = \PPP(T^{-n} \bx) \text{ for all } n\geq 0 \}.
\end{aligned}
\end{equation}
Observe that $\Wsl(\bx) = \PPP_0^\infty(\bx)$ and $\Wul(\bx) = \PPP_{-\infty}^0(\bx)$.
In particular, we have
\begin{equation}\label{eqn:cW}
c(\Wsl(\bx)) = \Wsl(c(\bx)) \cap c(M)
\quad\text{and}\quad
c(\Wul(\bx)) = \Wul(c(\bx)) \cap c(M).
\end{equation}
If $\bx \in M_{-1}^\infty$, then there exists a unique $x\in X^+$ such that $T^n\bx \in Y_{x_n}$ for every $n\geq 0$, so $\bx \in \bigcap_{n=0}^\infty Y_{x_{[0,n]}}^+$, and 
it follows from \cite[\S4.11]{CM06} (especially Lemma 4.57 and Corollary 4.61) that $\Wsl(\bx)$ is a $C^1$ smooth $s$-curve, which may or may not include one or both of its endpoints, and that
\[
(\Wsl(\bx))^\circ \subset \bigcap_{n=0}^\infty Y_{x_{[0,n]}}^+
\quad\text{and}\quad
\overline{\Wsl(\bx)} = \bigcap_{n=0}^\infty \overline{Y_{x_{[0,n]}}^+}.
\]
Thus we refer to $\Wsl(\bx)$ as the \emph{stable manifold} of $\bx$, even though when $\bx \in \SSS_{-1} \cup \SSS_\infty$ we could have $\Wsl(\bx) = \{\bx\}$. 

When $\bx \in M_{-\infty}^1$, similar statements hold for $\Wul(\bx)$, using negative indices, and we call this the \emph{unstable manifold} of $\bx$. 
Even though $\Wsl(\bx)$ or $\Wul(\bx)$ could be a single point for $\bx \in M \setminus M'$, we still refer to $\PPP_0^\infty$ as the \emph{partition into local stable manifolds}, and $\PPP_{-\infty}^0$ as the \emph{partition into local unstable manifolds}.

\begin{remark}
When $\bx \in M'$ so that $\pi^{-1}\bx$ is a single point $x\in X'$, we have
$\bigcap_{n=0}^\infty \overline{Y_{x_{[0,n]}}^+} = \pi(\Wul(x))$, and we see that $\Wsl(\bx)$ is equal to $\pi(\Wsl(x))$ up to possibly removing one or both endpoints. 
\end{remark}

We will work with the following set of \emph{regular} points:
\begin{equation}\label{eqn:Mreg}
\Mreg := \{ \bx \in M : \bx \in (\Wul(\bx))^\circ \text{ and } \bx \in (\Wsl(\bx))^\circ \}.
\end{equation}
That is, $\Mreg$ is the set of points that are not an endpoint of either their stable or unstable manifold.
Observe that stable manifolds cannot cross singularity curves in $\SSS_\infty$, and unstable manifolds cannot cross singularity curves in $\SSS_{-\infty}$, so
\begin{equation}\label{eqn:Mreg'}
\Mreg \subset M' = M \setminus (\SSS_\infty \cup \SSS_{-\infty}).
\end{equation}
This implies that
\begin{equation}\label{eqn:Xreg-1}
\Xreg := \pi^{-1}(\Mreg) = c(\Mreg) \subset X'.
\end{equation}

\begin{figure}[htbp]
\includegraphics{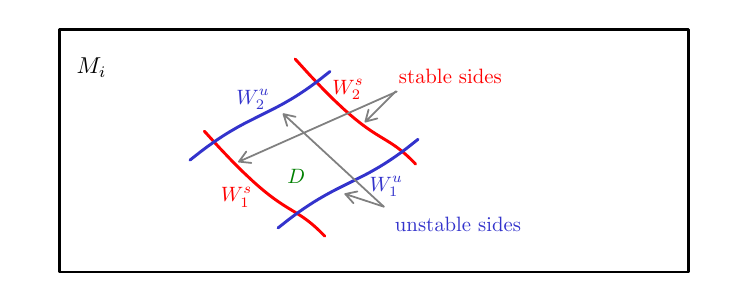}
\caption{A solid rectangle.}
\label{fig:rectangle}
\end{figure}

\begin{definition}\label{def:solid-rect}
A \emph{solid rectangle} is a closed and connected region $D \subset M$ that is bounded by two nontrivial stable manifolds $W_1^s, W_2^s$ and two nontrivial unstable manifolds $W_1^u, W_2^u$ with the property that each $W_i^s$ intersects each $W_j^u$ in a point which is interior to both curves; see Figure \ref{fig:rectangle}. We will refer to $W_i^s \cap D$ for $i=1,2$ as the \emph{stable sides} of $D$, and $W_i^u \cap D$ as the \emph{unstable sides} of $D$.
\end{definition}

An unstable curve $V$ is said to \emph{cross} a solid rectangle $D$ if $V^\circ$ intersects both of the stable sides of $D$.  Likewise, a stable curve $V$ is said to \emph{cross} $D$ if $V^\circ$ intersects both of the unstable sides of $D$.  The \emph{Cantor rectangle} associated to $D$ is the union of all points in $D$ whose stable and unstable manifolds cross $D$.  
In \cite{BD20}, these are referred to as \emph{locally maximal} Cantor rectangles; we will use the shorter term here for brevity, since we do not consider any other sort of Cantor rectangle in $M$.

Cantor rectangles are closed, nowhere dense subsets that naturally have product structure.
We say that a stable or unstable curve $V$ crosses a Cantor rectangle $R$ if it crosses the solid rectangle $D$ that generates $R$.

The next result from the literature plays a crucial role in our arguments. 

\begin{figure}[htbp]
\centering
\includegraphics{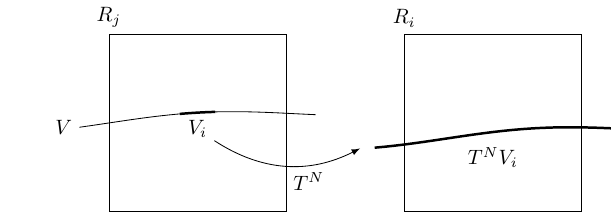}
\caption{Crossing of an image of a subcurve of $V$.}
\label{fig:crossing}
\end{figure}

\begin{proposition}[Sufficient rectangles]\label{prop:suff-rect}
There exists a countable cover $\mathcal{R}$ of $\Mreg$ by Cantor rectangles such that for every $\delta>0$ and every $R_0 \in \mathcal{R}$, there are $R_1,\dots, R_\ell \in \mathcal{R}$ that are \emph{$\delta$-sufficient} in the following sense: there exists $N\in \NN$ such that for every $V\in \Vud = \{V \in \Vu : |V| \geq \delta\}$, we have the behavior shown in Figure \ref{fig:crossing}, namely:
\begin{enumerate}
	\item there exists $j\in \{0,\dots, \ell\}$ such that $V$ crosses $R_j$, and
	\item for every $i\in \{0,\dots, \ell\}$, there exists $V_i \subset V$ such that $T^N V_i$ crosses $R_i$.
\end{enumerate}
\end{proposition}

We will write $c(\RRR) := \{ c(R) : R\in \RRR \}$ for the set of \emph{symbolic rectangles} in $X$ that code Cantor rectangles in $\RRR$.

We omit the full proof of Proposition \ref{prop:suff-rect}, which involves a more careful and detailed use of results from billiard dynamics, and simply sketch the main ideas; see \cite[Lemmas 7.87 and 7.90]{CM06} and \cite[Proof of Proposition 5.5]{BD20} for further details.\footnote{We caution that the notation in the references differs from ours, and they do not use the terminology ``$\delta$-sufficient'', which we introduce for the sake of convenience in later arguments.}

\begin{proof}[Sketch of proof of Proposition \ref{prop:suff-rect}]
The existence of a countable cover of $\Mreg$ by Cantor rectangles follows from Lemma 7.10 of \cite{BD20}.  These rectangles can be chosen such that each is large (in the sense of Liouville measure) in its solid rectangle.  
	
Define a notion of ``proper crossing'' of a Cantor rectangle $R$, which requires not only that $V\in \Vu$ fully crosses $R$, but that it does so near the center of $R$ in a certain sense. Then argue that every $V\in \Vu$ properly crosses some rectangle. Moreover, for a given rectangle, the set of $V\in \Vud$ that properly cross it is open in an appropriate topology, in which $\Vud$ is compact, and thus the finite set of rectangles $R_1,\dots, R_\ell$ can be obtained as a finite subcover of $\Vud$.

Now use the Lebesgue density theorem to identify a ``uniformly dense'' subset $P_i \subset R_i$, and use the mixing property of the Liouville measure $\mu_L$ to deduce that there exists $N$ such that for every $i,j$, we have $\mu_L(R_j \cap T^{-N}P_i) > 0$. Then argue that when $V$ properly crosses $R_j$, we can use the intersection point $\bar{y} \in R_j \cap T^{-N}P_i$ to construct $V_i \subset V$ such that $T^N V$ crosses $R_i$.
\end{proof}

\begin{remark}\label{rem:dense}
Although Cantor rectangles are nowhere dense, they can be ``large"  in the sense of Lebesgue measure.  In particular, they can be constructed to have high density with respect to Liouville (SRB) measure $\mu_L$, which is absolutely continuous with respect to Lebesgue measure $dr\, d\ph$.  The measure $\mu_L$ has been well studied, and many properties about it have been established.  The fact that $\mu_L$ is mixing \cite{yS70} is an important property that is needed to prove Proposition \ref{prop:suff-rect}.  Moreover, the Cantor rectangles described in Proposition \ref{prop:suff-rect} satisfy this density condition; see \cite[Definition 5.8]{BD20} or \cite[\S7.11]{CM06} for more details about this notion of density.
\end{remark}

\begin{remark}\label{rmk:tools-for-lps}
In the next section, we will prove that every MME $\nu$ for $X$ gives full weight to $\Xreg$; see Proposition \ref{prop:reg-meas}. 
Suppose we can show that the measure
$\mu$ constructed in \S\ref{sec:construction} satisfies  
$\mu(X) > 0$. Then by Proposition \ref{prop:inv-MME}, the normalized measure $\bar\mu = \mu/\mu(X)$ is an MME for $X$, and thus satisfies $\bar\mu(\Xreg)=1$.

Writing $\RRR$ for the countable cover of $\Mreg$ by Cantor rectangles provided by Proposition \ref{prop:suff-rect}, we see
that the (symbolic) rectangles $\{ c(R) : R\in \RRR \}$ cover $\Xreg$, and thus by Lemma \ref{lem:lps},
it will immediately follow that $\bar\mu$ has local product structure in the sense of Definition \ref{def:lps}.
Thus once we prove Proposition \ref{prop:reg-meas} below, we will have shown the following: 
\begin{quote}
\emph{If the measure $\mu$ from \S\ref{sec:construction} satisfies $\mu(X) > 0$, then $\bar\mu = \mu/\mu(X)$ is an MME for $X$, and has local product structure.}
\end{quote}
\end{remark}

\subsection{The regular set sees every MME}\label{sec:pos-ent}

Now we prove that:

\begin{proposition}\label{prop:reg-meas}
For each $\nu\in \MMM_\sigma^e(X)$ with $h_\nu(\sigma)>0$, we have $\nu(X \setminus \Xreg) = 0$. In particular, every MME for $X$ gives full weight to $\Xreg$.
\end{proposition}

Before proving the proposition, we note an important consequence: the pushforward map $\pi_*$ gives an entropy-preserving bijection between the sets of positive entropy ergodic measures on $X$ and on $M$. In particular, to prove existence of a unique MME for $T\colon M\to M$, it suffices to prove existence of a unique MME for $\sigma \colon X\to X$.

To prove Proposition \ref{prop:reg-meas}, first recall the following two facts from \S\ref{sec:cond-ent}:
\begin{itemize}
\item \emph{Rokhlin disintegration, Lemma \ref{lem:disint}:} given any Borel probability measure $\nu$ on $X$, 
there exists a full $\nu$-measure set $\CCC_\nu \subset X$ and a family of Borel probability measures $\{\nu_x : x\in \CCC_\nu \}$ such that 
\begin{equation}\label{eqn:nux}
\text{$\nu_x(\Wul(x)) = 1$ for every $x\in \CCC_\nu$, and $\nu = \int_{X} \nu_x \,d\nu$.}
\end{equation}
\item \emph{Conditional Shannon--McMillan--Breiman theorem, \eqref{eqn:SMB}--\eqref{eqn:S'nu}:} for every $\nu \in \MMM_\sigma^e(X)$,
there is a set $\SSS_\nu\subset X$ such that $\nu(\SSS_\nu)=1$ and
\begin{equation}\label{eqn:snu-0}
\lim_{n\to\infty} -\frac 1n \log \nu_x(\cyl{}{x_{[1,n]}}) = h_\nu(\sigma)
\quad\text{for all } x\in \SSS_\nu.
\end{equation}
Then the set $\SSS'_\nu := \{x \in \CCC_\nu : \nu_x(\SSS_\nu) = 1 \}$ also has $\nu(\SSS'_\nu)=1$.
\end{itemize}

Most of the work in proving Proposition \ref{prop:reg-meas} goes into establishing the following:

\begin{proposition}\label{prop:leaf-supp-nu}
For every $\nu \in \MMM_\sigma^e(X)$ with $h_\nu(\sigma)>0$, and every $x\in \SSS'_\nu$, we have $\nu_x(\Wul(x) \setminus X') = 0$.
\end{proposition}

Before proving Proposition \ref{prop:leaf-supp-nu}, we show how to use it to deduce Proposition \ref{prop:reg-meas}.
Consider the following sets:
\begin{equation}\label{eqn:X-endpt}
	\begin{aligned}
		X_\partial^u &:= \{ x\in X : \bx \text{ is an endpoint of } \Wul(\bx) \}, \\
		X_\partial^s &:= \{ x\in X : \bx \text{ is an endpoint of } \Wsl(\bx) \}.
	\end{aligned}
\end{equation}

\begin{lemma}\label{lem:null-endpts}
For every $\nu \in \MMM_\sigma^e(X)$ with $h_\nu(\sigma)>0$, we have $\nu(X_\partial^u) = \nu(X_\partial^s) = 0$.
\end{lemma}
\begin{proof}
Note that  $\Wul(x) \cap X'$ contains at most $2$ points in $X_\partial^u$
since $\pi$ is injective on $X'$. By Lemma \ref{lem:non-atomic}, $\nu_x$ is nonatomic, so Proposition \ref{prop:leaf-supp-nu} implies that
	\[
	\nu_x(X_\partial^u) \leq \nu_x (\Wul(x) \setminus X') + \nu_x(\Wul(x) \cap X' \cap X_\partial^u)
	= 0.
	\]
Since $\nu(\SSS'_\nu)=1$, integrating over $x$ and using the disintegration formula in  \eqref{eqn:u-cond} gives $\nu(X_\partial^u) = 0$. The argument for $X_\partial^s$ is similar.
\end{proof}

Proposition \ref{prop:reg-meas} follows from  Lemma \ref{lem:null-endpts} and the observation that $\Xreg = X \setminus (X_\partial^u \cup X_\partial^s)$.
Before proceeding to the proof of Proposition \ref{prop:leaf-supp-nu}, we note the following consequence of Proposition \ref{prop:reg-meas}:

\begin{corollary}\label{cor:Yreg}
Given any $\nu\in \MMM_\sigma^e(X)$ with $h_\nu(\sigma)>0$, the set
	\begin{equation}\label{eqn:Yreg}
		\Yreg := \{ x\in \Xreg \cap \CCC_\nu : \nu_x(\Wul(x) \setminus \Xreg) = 0 \}
	\end{equation}
	has the property that $\nu(X\setminus \Yreg) = 0$.
\end{corollary}

Proposition \ref{prop:leaf-supp-nu} is an immediate consequence of the following two lemmas, since 
\[
\Wul(x) \setminus X' = \bigcup_{\bill{y} \in \Wul(\pi(x)) \setminus M'} \Wul(x) \cap \pi^{-1}(\bill{y}).
\]

\begin{lemma}\label{lem:ctbl-sing-0}
	For every $\bx\in M'$, the set $\Wul(\bx) \setminus M'$ is countable.
\end{lemma}

\begin{lemma}\label{lem:nonatomic-0}
Given any $\nu \in \MMM_\sigma^e(X)$ with $h_\nu(\sigma)>0$, any $x\in \SSS'_\nu$, and any $y\in \Wul(x) \cap \SSS_\nu$, we have $\nu_x(\pi^{-1}(\pi(y))) = 0$.
\end{lemma}

\begin{proof}[Proof of Lemma \ref{lem:ctbl-sing-0}]
	For every $n\in \NN$ and every $V\in \bbS_n \cup \bbS_{-n}$, 
	we claim that $\Wul(\bx) \cap V$ contains at most one point. 
	Indeed, given any $\bill{y},\bill{z} \in \Wul(\bx)$, the displacement $T^{-n}\bill{y} - T^{-n}\bill{z}$ lies in the unstable cone since $T^{-n}\Wul(\bx)$ is a $u$-curve, but $T^{-n}V$ is a horizontal line, so we can only have $\bill{y},\bill{z} \in V$ if $\bill{y} = \bill{z}$.
	We conclude that $\Wul(\bx) \cap (\SSS_{-n} \cup \SSS_n)$ is finite, and taking a union over all $n$ proves the lemma.
\end{proof}

\begin{proof}[Proof of Lemma \ref{lem:nonatomic-0}]
Fix $x\in \SSS_\nu'$. Then $\nu_x(\SSS_\nu)=1$.
Fix $\zeta>0$. There exists $Z = Z_\zeta \subset \SSS_\nu \cap \Wul(x)$ such that $\nu_x(Z) > 1-\zeta$ and the convergence in \eqref{eqn:snu-0} is uniform on $Z$. In particular, for every $h' \in (0,h_\nu)$, there exists $N\in \NN$ such that for all $n\geq N$ and $y\in Z$, we have
\[
\nu_x(\cyl{}{y_{[1,n]}}) < e^{-nh'}.
\]
Fix $n\in \NN$, let $\delta>0$ be given by Corollary \ref{cor:local-lin}, 
	and let
	\[
	\EEE_n := \{ w\in \LLL_n : \cyl{}{w} \cap \pi^{-1}(\bill{y}) \cap \Wul(x) \neq \emptyset \}
	\]
	be the set of words of length $n$ that begin a forward coding of $\bill{y} := \pi(y)$. 
	
	For each $w\in \EEE_n$, the cell $Y_w^+ \in \MMM_{-1}^{n}$ must intersect $B(\bill{y},\delta) \setminus \SSS_n$, since every $y'\in \pi^{-1}(\bill{y})$ has $y' = \lim_{k\to\infty} \pi^{-1}(\bx_k)$ for some sequence $\bx_k \in M'$ converging to $\bill{y}$.  By Corollary \ref{cor:local-lin}, the number of such cells is at most $Kn+1$, so we have
	\begin{equation}\label{eqn:Kn}
		\#\EEE_n \leq Kn+1,
		\quad\text{and}\quad
		\pi^{-1}(\bill{y}) \cap \Wul(x) \subset \bigcup_{w\in \EEE_n} \pbl w \pbr.
	\end{equation}
From this we deduce that
\begin{align*}
\nu_x(\pi^{-1}(\bill{y}))
\leq \sum_{w\in \EEE_n} \nu_x (\pbl w \pbr)
\leq \zeta + (\#\EEE_n) e^{-nh'}
\leq \zeta + (Kn+1) e^{-nh'}.
\end{align*}
Sending $n\to\infty$ shows that $\nu_x(\pi^{-1}(\pi(y))) \leq \zeta$, and since $\zeta>0$ was arbitrary, this proves the lemma.
\end{proof}

\section{Counting bounds}\label{sec:counting}

\subsection{Overview}\label{sec:overview}

Let $X$ be the shift space constructed in \S\ref{sec:coding-bill} to code the billiard map, and let $\LLL$ be the subset of the language described in \eqref{eqn:bill-lang}.
As in \S\ref{sec:construction}, we denote the topological entropy of $X$ by $h = \htop(X,\sigma) = \lim \frac 1n \log \#\LLL_n$.

In this section we establish the counting bounds necessary for the construction in \S\ref{sec:construction} to produce an MME on $X$.  These bounds will also play a role in showing that the MME is unique.
Recall from \eqref{eqn:lcb} that submultiplicativity of $\#\LLL_n$ automatically gives the lower counting bound $\#\LLL_n \geq e^{nh}$ using Fekete's lemma. 
In this section, we will prove the following results.

\begin{proposition}[Global upper bound]\label{prop:ucb}
There exists $C_4>0$ such that for every $n\in \NN$, we have $\#\LLL_n \leq C_4 e^{nh}$.
\end{proposition}

\begin{proposition}[Local lower bound]\label{prop:loc-lcb}
For every $\delta>0$ there exists
$q := q_\delta > 0$ such that if $x\in X$ is such that $|\Wul(\bx)| \geq \delta$, then
\[
\# \{w \in \LLL_n : \cyl{}{w} \cap \Wul(x) \neq \emptyset \} \geq q e^{nh}
\quad\text{for all }
n\in \NN.
\]
\end{proposition}

The following is an immediate consequence of Propositions \ref{prop:ucb} and \ref{prop:nondegen}:

\begin{corollary}\label{cor:finite}
The measures $\mf$, $\mb$, and $\mu$ are all finite.
\end{corollary}

The bulk of the work in \S\S\ref{sec:good-bad-ugly}--\ref{sec:long-cells} goes into proving the following result, which we use in \S\ref{sec:qsm} to prove Propositions \ref{prop:ucb} and \ref{prop:loc-lcb}, as well as Theorem \ref{thm:length-growth}.

\begin{proposition}\label{prop:LV-geq-L}
For every sufficiently small $\delta>0$, there exists $b_0 > 0$ such that 
given any $u$-curve $V$ with $|V| \geq \delta$, and any $n\in \NN$, we have
\begin{equation}\label{eqn:LV-geq-L}
\# \{ w\in \LLL_n : V \cap T^{-1} Y_w^+ \neq \emptyset \}
\geq b_0 \#\LLL_n.
\end{equation}
\end{proposition}

 The local lower bound in Proposition \ref{prop:loc-lcb} follows immediately from Proposition \ref{prop:LV-geq-L} and \eqref{eqn:lcb}. In \S\ref{sec:positive}, we use this bound to show that every open set with no isolated points has positive measures with respect to each of $\mf$, $\mb$, and $\mu$.  

\begin{remark}
The bounds in this section are analogous to those in \cite{BD20}, and our arguments largely follow the approach there.  In particular, one could obtain the global upper bound from \cite[Proposition 4.6]{BD20} and the local lower bound from \cite[Lemma 5.5]{BD20}.  We include full proofs to emphasize that we do not require the sparse recurrence property, and to give simpler arguments in places where 
the use of anisotropic Banach spaces required
\cite{BD20} to prove more complicated bounds.
\end{remark}

\subsection{Notation: the good, the bad, and the ugly}\label{sec:good-bad-ugly}

Recall that $\Vu$ denotes the set of all $u$-curves, and $\Vud$ the set of all $u$-curves with length at least $\delta$. Now we write $\Vuq$ for the set of all $V\in \Vu$ such that $V$ is contained in some $Y_a\in \QQQ$, and $\Vuqd$ for the set of all $V\in \Vuq$ with $|V| \geq \delta$.

Given $w\in \LLL_n$ and $V\in \Vuq$, we consider the set of ``follower'' words
\begin{equation}\label{eqn:LnV}
\LLL_n^V = \{w \in \LLL_n : |V_w| > 0 \},
\quad\text{where }
V_w := V \cap T^{-1}Y_w^+.
\end{equation}
That is, $V_w$ denotes the subcurve of $V$ whose next $n$ iterates are coded by $w$,
in the sense that $T^j V_w \subset Y_{w_j}$ for $1\leq j\leq n$.
We refer to the curves $T^j V_w$ for $1\leq j < n$ as \emph{intermediate images}, and the curve $T^n V_w$ as the \emph{final image}.
Observe that
\[
T^nV_w \subset T^{n-1}Y_w^+ = Y_w^- \subset Y_{w_n} \in \MMM_{-1}^1,
\]
so the final image $T^nV_w$ is contained in the cell $Y_{w_n}$ of $\QQQ$. In particular, $T^nV_w \in \Vuq$. Since $T$ is continuous on this cell, $T^{n+1}V_w$ is still a single curve, but it may no longer be in $\Vuq$, and
for $k>n+1$, the image $T^kV_w$ can be cut into multiple pieces by the singularities of the billiard map.

For a given length threshold $\delta>0$, we will partition $\LLL_n^V$ into \emph{good} words (long final image) and \emph{bad} words (short final image):
\begin{equation}\label{eqn:good-bad}
	\begin{split}
\GGG_n^\delta(V) & := \{ w\in \LLL_n^V  : |T^{n} V_w| \geq \delta \}, \\
\BBB_n^\delta(V) & := \LLL_n^V \setminus \GGG_n^\delta(V)
= \{w\in \LLL_n^V : |T^{n} V_w| < \delta \}.
	\end{split}
\end{equation}
We also consider ``ugly'' or ``ungrowing'' words, with short intermediate images:
\begin{equation}\label{eqn:ugly}
\UUU_n^\delta(V) := \{w\in \LLL_n^V : w_{[1,k]} \in \BBB_k^\delta(V) \text{ for all } 1\leq k < n \}.
\end{equation}

\begin{remark}\label{rmk:notation}
The image curves $T^{n}V_w$ here play the role of the `descendant' curves used in \cite[\S5]{BD20}, although we use u-curves and they use s-curves. Two further  differences are worth noting. First, we do not subdivide artificially: the curves $T^{n} V_w$ could have length much greater than the threshold $\delta$, unlike in \cite{BD20}. Second, our notation differs from theirs, and in fact is nearly reversed: we already used $\LLL_n$ to refer to the shift space's language, so we use $\LLL_n^V$ to refer to the set of all follower words, and use $\GGG_n^\delta(V)$ to refer to the ``good'' curves,\footnote{This also follows the standard notation from the first author's work with Dan Thompson on non-uniform specification properties \cite{CT12,CT21}, some of whose techniques are echoed below.} whereas \cite[\S5]{BD20} uses $L_n(W)$ for ``long'' curves, and $\GGG_n(W)$ for the set of all descendant curves.
\end{remark}

Every follower word is either ugly itself or can be decomposed into a good beginning and an ugly end (which may be empty):

\begin{lemma}[Decomposition]\label{lem:decomp}
Fix $\delta>0$ and $V\in \Vuq$. 
Given $n\geq 0$ and any $w\in \LLL_n^V \setminus\UUU_n^\delta(V)$,
there exists $k \in \{1,\dots, n\}$ such that $w_{[1,k]} \in \GGG_k^\delta(V)$.
For the maximal such $k$, write $u := w_{[1,k]}$ and $v := w_{(k,n]}$. Then $v \in \UUU_{n-k}^\delta(T^kV_u)$.
In particular, we have
\begin{equation}\label{eqn:decomp}
\LLL_n^V = \UUU_n^\delta(V)
\cup \bigcup_{k=1}^n \bigcup_{u\in \GGG_k^\delta(V)} \{uv : v\in \UUU_{n-k}^\delta(T^k V_u) \}.
\end{equation}
\end{lemma}
\begin{proof}
Fix $w \in \LLL_n^V \setminus \UUU_n^\delta(V)$.
The claim regarding existence of $k$ follows from the definition of $\UUU_n^\delta(V)$ in \eqref{eqn:ugly}.
For the maximal $k\in \{1,\dots, n\}$ with $|T^{k} V_{w_{[1,k]}}| \geq \delta$, maximality implies that $|T^{k+j} V_{w_{[1,k+j]}}| < \delta$ for all $1 \leq j \leq n-k$.  
Writing $u = w_{[1,k]}$ and $v = w_{(k,n]}$, we have
$(T^k V_u)_{v_{[1,j]}} = T^kV_{uv_{[1,j]}}$, so for  $1\leq j \leq n-k$, we have
\[
|T^j(T^k V_u)_{v_{[1,j]}}| = |T^{k+j}V_{uv_{[1,j]}}| = |T^{k+j}V_{w_{[1,k+j]}}| < \delta,
\]
and we conclude that $v\in \UUU_{n-k}^\delta(T^kV_u)$, as claimed.
\end{proof}

The linear complexity bound in Corollary \ref{cor:local-lin} has the following consequence:

\begin{lemma}[Slow fragmentation]\label{lem:tempered}
For every $\alpha >0$, there exist $\rho>0$ and $C_5 > 0$ such that for every $V\in \Vuq$ and $n\geq 0$, we have $\#\UUU_n^\rho(V) \leq C_5 e^{\alpha n}$.
\end{lemma}
\begin{proof}
Given any $\alpha > 0$, there exists $m\in \NN$ large enough that $(2Km+1)^{1/m} \leq e^{\alpha}$.  By Corollary \ref{cor:local-lin}, there exists $\rho > 0$ sufficiently small so that for every $\bx \in M$, the set $B(\bx,\rho/2)\setminus \SSS_{2m}$ has at most $2Km+1$ connected components. 
Since $V$ is a $u$-curve and the components of $B(\bx,\rho/2)\setminus \SSS_{2m}$ are separated by $s$-curves, the intersection of $V$ with each component is connected, and we conclude that:
\begin{itemize}
\item if $|V| < \rho$, then for each $0\leq j\leq 2m$, the set $T^j V$ has at most $Kj+1$ connected components.
\end{itemize}
If $n = km + \ell$ with $0\leq \ell < m$, and $w \in \UUU_{jm}^\rho(V)$ for some $1\leq j\leq k$, then
\[
\# \UUU_m^\rho(T^{jm}V_w) \leq Km + 1.
\]
By iterating this bound we get
\[
\# \UUU_{n}^\rho(V) \leq (Km+1)^{k-1}(K(m+\ell)+1) \leq (2Km+1)^{n/m} \leq e^{\alpha n}.
\]
If $|V| > \rho$, then 
\[
\#\UUU_{n}^\rho(V) = \sum_{u \in \UUU_2^\rho(V)} \# \UUU_{n-2}^\rho(T V_u) \leq  \#\LLL_2 e^{\alpha n}.\qedhere
\]
\end{proof}

Recalling that  $\lambda>0$ is the contraction rate from Lemma \ref{lem:s-contr}, we have
the following estimate, which gives a concrete sense in which fragmentation is dominated by growth, provided we choose a small enough scale.

\begin{lemma}[Expansion--fragmentation]\label{lem:grow-cut}
For every $\gamma \in (0,\lambda)$, there exist $\rho>0$ and $C_6>0$ such that for every $V\in \Vuq$ and $n\geq 0$, we have
\begin{equation}\label{eqn:grow-cut}
\#\UUU_n^\rho(V) \leq C_6 e^{-\gamma n} \frac{|T^n V|}{|V|}.
\end{equation}
\end{lemma}
\begin{proof}
Let $\rho,C_5$ be given by Lemma \ref{lem:tempered} for $\alpha = \lambda - \gamma$. Then the bounds in Lemmas \ref{lem:s-contr} and \ref{lem:tempered} give
\[
|V| \cdot \#\UUU_n^\rho(V) \leq C_2 e^{-\lambda n} |T^n V| \cdot C_5 e^{\alpha n} = C_2 C_5 e^{-(\lambda - \alpha)n} |T^n V|,
\]
which proves the lemma.
\end{proof}

From now on, we will fix $\gamma \in (0,\lambda)$, and let $\rho>0$ be provided by Lemma \ref{lem:grow-cut}. We will eventually fix a smaller threshold $\delta \in (0,\rho)$ to use in our final estimates; see Proposition \ref{prop:control-bad}.

\subsection{Long curves}

Given $V\in \Vuq$, we will establish a uniform relationship between the quantities $\#\LLL_n^V$, $\#\GGG_n^\delta(V)$, and $|T^n V|$. We start by controlling the average length of final images associated with words in $\LLL_n^V$. Figure \ref{fig:average-length} illustrates the following bound.

\begin{proposition}[Average length]\label{prop:average-length}
There exist $b_1,b_2>0$ such that for every $V\in \Vuq$ and every $n\geq 0$, we have
\begin{equation}\label{eqn:average-length}
\frac{|T^n V|}{\#\LLL_n^V} \geq \min\big(b_1,\, b_2 e^{\gamma n} |V|\big).
\end{equation}
\end{proposition}
\begin{proof}
Using the decomposition and expansion--fragmentation lemmas, we have
\begin{align*}
\#\LLL_n^V &\leq 
\#\UUU_n^\rho(V) + \sum_{k=1}^n \sum_{w\in \GGG_k^\rho(V)} \#\UUU_{n-k}^\rho(T^k V_w) 
&&\text{by Lemma \ref{lem:decomp}} \\
&\leq C_6 e^{-\gamma n} \frac{|T^n V|}{|V|}
+ \sum_{k=1}^n \sum_{w\in \GGG_k^\rho(V)} C_6 e^{-\gamma(n-k)} \frac{|T^n V_w|}{|T^k V_w|}
&&\text{by Lemma \ref{lem:grow-cut}}.
\end{align*}
For every $w \in \GGG_k^\rho(V)$, we have $|T^k V_w| \geq \rho$ by definition. Moreover, since the curves $\{V_w : w \in \GGG_k^\rho(V)\}$ have disjoint interiors, we have $\sum_{w\in \GGG_k^\rho(V)} |T^n V_w| \leq |T^n V|$, and we conclude that
\begin{align*}
\#\LLL_n^V &\leq C_6 |T^n V| \Big(
\frac{e^{-\gamma n}}{|V|} + \frac 1\rho \sum_{k=1}^n e^{-\gamma(n-k)} \Big) \\
&\leq 2C_6 |T^n V| \max \big( e^{-\gamma n} |V|^{-1},\, (1-e^{-\gamma})^{-1} \rho^{-1}\big).
\end{align*}
Taking $b_2 = (2C_6)^{-1}$ and $b_1 = b_2 \rho(1-e^{-\gamma})$, this proves Proposition \ref{prop:average-length}.
\end{proof}

\begin{figure}
\centering
\includegraphics{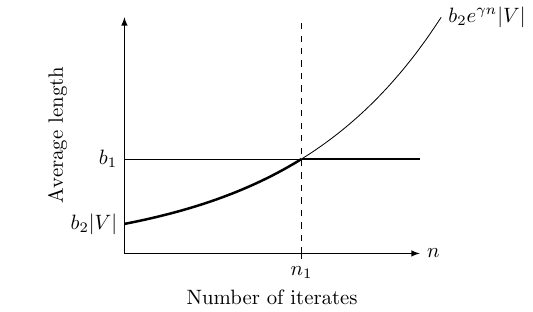}
\caption{Lower bound on the average length of a curve.}
\label{fig:average-length}
\end{figure}

\begin{corollary}\label{cor:ave-length}
Let $b_1,b_2>0$ be as in Proposition \ref{prop:average-length}. Then
given $n_1\in \NN$ large enough that $b_2 e^{\gamma n_1} \delta \geq b_1$ (see Figure \ref{fig:average-length}), we have
\begin{equation}\label{eqn:TnV-geq}
\frac{|T^n V|}{\#\LLL_n^V} \geq b_1
\quad\text{for all } V\in \Vuqd \text{ and } n\geq n_1.
\end{equation}
\end{corollary}

We are now in a position to establish a uniform estimate relating the numbers of short and long curves in $T^n V$.

\begin{proposition}[Many good curves]\label{prop:control-bad}
Given $b_1>0$ as in Proposition \ref{prop:average-length} and $\delta \in (0,b_1)$, let $n_1 \in \NN$ be as in Corollary \ref{cor:ave-length}.
Then there exists $b_3>0$ such that for every $V\in \Vuqd$ and $n\geq n_1$, we have
\begin{equation}\label{eqn:control-bad}
\#\GGG_n^\delta(V) \geq b_3 \#\LLL_n^V.
\end{equation}
\end{proposition}
\begin{proof}
Recall from \S\ref{sec:background} that $L$ denotes the maximum length of a $u$-curve, so given $V\in \Vuqd$, we have
\begin{equation}\label{eqn:TnV-leq}
|T^n V| = \sum_{w\in \GGG_n^\delta(V)} |T^n V_w|
+ \sum_{u\in \BBB_n^\delta(V)} |T^n V_u|
\leq L \#\GGG_n^\delta(V) + \delta \#\BBB_n^\delta(V).
\end{equation}
Writing $r_n := \#\GGG_n^\delta(V) / \#\LLL_n^V$, we see that
\[
\#\GGG_n^\delta(V) = r_n \#\LLL_n^V
\quad\text{and}\quad
\#\BBB_n^\delta(V) = (1-r_n) \#\LLL_n^V,
\]
so we can combine \eqref{eqn:TnV-geq} and \eqref{eqn:TnV-leq} to obtain
\[
b_1 \leq \frac{|T^n V|}{\#\LLL_n^V} \leq L r_n + \delta(1-r_n)
\quad\Rightarrow\quad
(b_1 - \delta) \leq (L-\delta) r_n.
\]
This proves Proposition \ref{prop:control-bad} with $b_3 = (b_1 - \delta)/(L-\delta)$.
\end{proof}

\begin{remark}
The appearance of some threshold time $n_1$ in Proposition \ref{prop:control-bad} is unavoidable: a curve $V\in \Vuqd$ may immediately be cut into shorter pieces so that every component of $TV$ has length $<\delta$, in which case $\#\GGG_1^\delta(V) = 0$, and \eqref{eqn:control-bad} cannot hold until $n$ is large enough that expansion is dominant over fragmentation.
\end{remark}

\begin{remark}\label{rem:many-good}
Proposition \ref{prop:control-bad} continues to hold if we let $\eps>0$ be arbitrarily small and then allow any $V\in \Vuq_\eps$ (instead of only $V\in \Vuqd$): 
one then chooses $n_1$ such that $b_2 e^{\gamma n} \eps \geq b_1$ and proceeds as before, obtaining the result with the same $b_3$ (depending on $\delta$, but independent of $\eps$) but a larger $n_1$ (depending on $\eps$). 
The case $\eps=\delta$ will suffice for our needs.
\end{remark}

\subsection{Boundary components}\label{sec:bdry-comp}

Given $0\leq k\leq n$, we say that a component curve $V \in \bbS_k$ \emph{bounds} a cell $P \in \MMM_{-n}^0$ if $V \cap \partial P$ has positive length.
When $V$ bounds $P$, we can either have $V \subset \partial P$ or $V \not\subset \partial P$. We will need to control when the second of these can happen.

Say that the \emph{order} of a cell $P \in \MMM_{-n}^0$ is the maximal value of $k \in \{1,2,\dots,n\}$ such that there is a curve $V \in \bbS_k$ that bounds $P$.

\begin{lemma}[Full boundary components]\label{lem:bdry-comp}
Let $P \in \MMM_{-n}^0$ be a cell with order $k$.  Suppose that $V\in \bbS_k$ bounds $P$ and that $V \cap \partial P$ does not intersect any other $W\in \bbS_k$ that bounds $P$. Then $V \subset \partial P$.
\end{lemma}
\begin{proof}
Let $\bx_1$ and $\bx_2$ be the endpoints of  $V \cap \partial P$.
Then for each $i\in \{1,2\}$, there exists $j \in \{0,\dots, k-1\}$ and $W_i \in \bbS_j$ such that $\bx_i \in V \cap W_i$, by hypothesis.  By the Noncrossing Lemma \ref{lem:noncrossing}, this implies that  $\bx_i \in \partial V$. It follows that $V \subset \partial P$.
\end{proof}

Given $k\in \NN$, consider the following collection of cells:
\begin{multline}\label{eqn:bbI}
\bbI_k := \{ P \in \MMM_{-k}^0 : P \text{ is a cell of order $k$ and there exist $V,W\in \bbS_k$} \\
\text{such that $V,W$ both bound $P$ and $V\cap W \neq \emptyset$} \}.
\end{multline}
Let $\bbI_k^c$ denote the collection of cells in $\MMM_{-k}^0$ that are not in $\bbI_k$. We also write
\begin{equation}\label{eqn:bbS}
\bbS_k' := \{ V \in \bbS_k : \text{there exists $W\in \bbS_k \setminus \{V\}$ such that $V\cap W \neq \emptyset$} \}.
\end{equation}
Lemma \ref{lem:bdry-comp} has the following consequence.

\begin{proposition}\label{prop:bdry-comp}
Fix $k \in \NN$ and $V \in \bbS_k \setminus \bbS_k'$. Then $V \subset \partial P$ for every cell $P$ of order $k$ that it bounds. Similarly, given any $P \in \bbI_k^c$, every $V\in \bbS_k$ that bounds $P$ must be contained in $\partial P$.
\end{proposition}

Moreover, we have the following bound.

\begin{lemma}\label{lem:bad-curve-cell}
There exists $C_7>0$ such that for every $k\in \NN$, we have $\#\bbS_{k}' \leq C_7$ and $\#\bbI_k \leq C_7$.
\end{lemma}
\begin{proof}
Consider the following (finite) set of points:
\begin{equation}\label{eqn:S1-int}
F := \bigcup\{ V \cap W : V,W \in \bbS_1,\ V\neq W \}.
\end{equation}
Let $K$ be as in Lemma \ref{lem:lin-comp}, so that at most $K$ components of $\bbS_1$ meet at any single point in $F$, and let $C_7 := (K+1)(\#F)$. 

Observe that if $V\in \bbS_{k}$ intersects some other $W\in \bbS_{k}$ at a point $\bx$, then $\bx \in T^{k}F$.  Each point in $T^{k}F$ lies in at most $K$ elements of $\bbS_{k+1}$, so $\# \bbS'_{k+1} \leq (K+1)(\# F)$.  See Figure \ref{fig:small-cells}.

For the second bound, observe that if $P \in \bbI_k$, then it is bounded by two curves $V,W \in \bbS_{k}$ intersecting at $\bx \in \overline{P}$.  For sufficiently small $\delta > 0$, we have that the set $T^{-k}(B(\bx,\delta) \cap P)$ is a connected component of $U \setminus T^{-1}\SSS_0$, where $U$ is a small neighborhood of a point of $F$. Since there can be at most $K+1$ adjacent cells around each point in $F$, we obtain the second inequality.
\end{proof}

\begin{figure}
\centering
\includegraphics{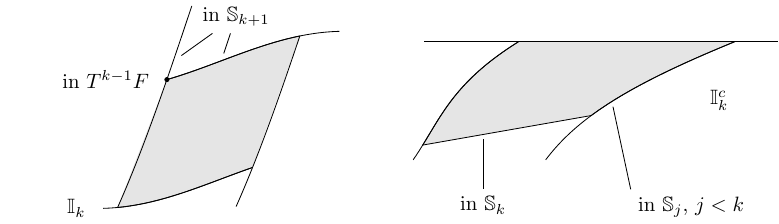}
\caption{Cells in $\bbI_k$ and $\bbI_k^c$.}
\label{fig:small-cells}
\end{figure}

\subsection{Long cells}\label{sec:long-cells}
Following \cite{BD20}, we distinguish between ``long cells'' (which contain some $V\in \Vuq$ that has a large image) and ``short cells'' (which contain no such $V$). In our setting we think of these as good and bad words in the language $\LLL_n$, writing
\begin{align*}
	\GGG_n^{u,\delta} & := \bigcup_{V\in \Vuq} \GGG_n^\delta(V) 
	= \{ w \in \LLL_n : \text{there exists } V \in \Vuq \text{ such that } |T^n V_w| \geq \delta \} \\ 
	& = \{ w \in \LLL_n : Y_w^- \text{ contains a $u$-curve with length at least $\delta$} \},
\end{align*}
and then putting
\[
\BBB_n^{u,\delta} := \LLL_n \setminus \GGG_n^{u,\delta}.
\]
All of the above definitions can also be given with $u$ replaced by $s$, mutatis mutandis.

\begin{remark}
In \S\ref{sec:good-bad-ugly}, we defined 
sets $\GGG_n^{\delta}(V)$ and $\BBB_n^{\delta}(V)$ of ``good'' and ``bad'' words for a $u$-curve $V$.  That partition took place at the level of a curve, rather than at the level of the language $\LLL_n$, which is where we work now. Additionally, the words in $\LLL_n^V$ partition $V$ via the cells $T^{-1}Y_w^+$, which correspond to cylinders $\sigma^{-1}\cyl{}{w}$, whereas for $\GGG_n^{u,\delta}$ and $\BBB_n^{u,\delta}$ we are working with the cells $Y_w^-$, corresponding to the cylinders $\cyl{w}{}$.
\end{remark}

The rest of \S\ref{sec:long-cells} is devoted to proving the following:

\begin{proposition}\label{prop:Bn-Gn}
For every sufficiently small $\delta>0$, there exists $C_8>0$ such that for every $n\geq 0$, we have
\begin{equation}\label{eqn:Bn-Gn}
\#\BBB_n^{u,\delta} \leq C_8 \#\GGG_n^{u,\delta}.
\end{equation}
Consequently, writing $b_4 := (1+C_8)^{-1}$, we have
\begin{equation}\label{eqn:Gn-Ln}
\#\GGG_n^{u,\delta} \geq b_4 \#\LLL_n.
\end{equation}
\end{proposition}

To prove Proposition \ref{prop:Bn-Gn}, we first need some consequences of \S\ref{sec:bdry-comp} concerning the boundary components of cells.  We will need the following notation, which is similar to \eqref{eqn:good-bad}.  Given $w\in \LLL_n$, let
\[
S_w^+ := (\overline{Y_w^+} \cap \SSS_0)^\circ,
\]
and let $\LLL_n^{\SSS_0}$ denote the set of all $w\in \LLL_n$ such that $S_w^+$ is nonempty. Observe that we can now describe $\bbS_n$ as
\[
\bbS_n = \{ \overline{T^n S_w^+} : w\in \LLL_n^{\SSS_0} \}.
\]
As before, we define
\begin{align*}
\GGG_n^\delta(\SSS_0) &= \{w \in\LLL_n^{\SSS_0} : |T^n S_w^+| >\delta\}, \\
\BBB_n^\delta(\SSS_0) &= \LLL_n^{\SSS_0} \setminus \GGG_n^{\delta}(\SSS_0).
\end{align*}
Using the constant $C_7$ from Lemma \ref{lem:bad-curve-cell}, we have the following bounds.

\begin{lemma}\label{lem:cells-curves}
Given $\delta \in (0,b_1)$, we have
\begin{align}
\label{eqn:cell-leq-curve}
\#\BBB_n^{u,\delta} &\leq C_7 n + 2\sum_{k=1}^{n+1} \#\BBB_k^\delta(\SSS_0), \\
\label{eqn:curve-leq-cell}
\#\GGG_{n+1}^\delta(\SSS_0) &\leq C_7 + L\delta^{-1} \#\GGG_n^{u,\delta}.
\end{align}
\end{lemma}
\begin{proof}
Given $w\in \LLL_n$, the set $TY_w^-$ is a cell in $\MMM_{-(n+1)}^0$, and its order (the largest $k$ such that some curve in $\bbS_k$ is a boundary component of the cell) is an element of $\{1,\dots, n+1\}$. Let
\[
\OOO_n^k := \{ w\in \LLL_n : TY_w^- \text{ has order } k \}.
\] 
Given $w\in \OOO_n^k$, the set $TY_w^-$ is also a cell in $\MMM_{-k}^0$. Recall that $\bbI_k$ denotes the set of cells in $\MMM_{-k}^0$ that are bounded by two intersecting elements of $\bbS_k$. Let
\[
\III_n^k := \{ w\in \OOO_n^k : TY_w^- \in \bbI_k \}
\quad\text{and}\quad
\III_n := \bigcup_{k=1}^{n+1} \III_n^k.
\]
For each $k\in \{1,\dots, n+1\}$, we have $\#\III_n^k \leq \#\bbI_k \leq C_7$, and thus
\begin{equation}\label{eqn:In-leq}
\#\III_n \leq (n+1) C_7.
\end{equation}
Given any $w\in \BBB_n^{u,\delta}$, we either have $w\in \III_n$ or $w\notin \III_n$. In the latter case, we have $w\in \OOO_n^k \setminus \III_n^k$ for some $k \in \{1,\dots, n+1\}$, so the cell $TY_w^-$ lies in $\bbI_k^c$, and Proposition \ref{prop:bdry-comp} implies that its boundary contains some $V = V(w) \in \bbS_k$. Since $w\in \BBB_n^{u,\delta}$, we must have $|V(w)| < \delta$.

With $w$ and $V(w)$ as in the previous paragraph, let $r(w) = w_{[n-k+1,n]}$ be the ``reduced'' word that codes $TY_w^- = TY_{r(w)}^-$ as a cell in $\MMM_{-k}^0$. Then $V(w) = T^k S_{r(w)}^+$, and we conclude that $r(w) \in \BBB_k^\delta(\SSS_0)$. Moreover, the map $w \mapsto V(w)$ is at most 2-to-1 on $\BBB_n^{u,\delta} \cap \OOO_n^k \setminus \III_n^k$, so we conclude that
\[
\#\BBB_n^{u,\delta} \cap \OOO_n^k \setminus \III_n^k \leq 2 \#\BBB_k^\delta(\SSS_0).
\]
Summing over $k$ and using \eqref{eqn:In-leq} implies \eqref{eqn:cell-leq-curve}. 

For \eqref{eqn:curve-leq-cell}, 
observe that given $w\in \GGG_{n+1}^\delta(\SSS_0)$, the curve $V = \overline{T^{n+1}S_w^+}$ is an element of $\bbS_{n+1}$ such that $|V| \geq \delta$. We would like to conclude that the long $u$-curve $V$ forces some cells in $\MMM_{-(n+1)}^0$ to contain long $u$-curves. If $V \in \bbS_{n+1}'$ then we cannot draw any such information, but 
by Lemma \ref{lem:bad-curve-cell}, the number of such curves $V$ is at most $C_7$.
If $V \notin \bbS_{n+1}'$, then Proposition \ref{prop:bdry-comp} implies that $V\subset \partial P$ for every cell in $\MMM_{-(n+1)}^0$ that it bounds; thus it bounds exactly two such cells, and is completely contained in both of their boundaries.
At the same time, every cell in $\MMM_{-(n+1)}^0$ has upper and lower boundaries given by $u$-curves together with some (possibly empty) subset of $\SSS_0$. Each such boundary $u$-curve has length $\leq L$, so it can contain at most $L/\delta$ curves $V$ with $|V|\geq \delta$. Thus each cell in $\MMM_{-(n+1)}^0$ can be bounded by at most $2L/\delta$ curves $V\in \bbS_{n+1} \setminus \bbS_{n+1}'$ that are of the form $V = \overline{T^{n+1}S_w^+}$ for some  $w\in \GGG_{n+1}^\delta(\SSS_0)$.
The number of such curves is at least $\#\GGG_{n+1}^\delta(\SSS_0) - C_7$, and \eqref{eqn:curve-leq-cell} follows.
\end{proof}

\begin{proof}[Proof of Proposition \ref{prop:Bn-Gn}]
Let $\bbS_1^1 := \{ V \cap Y_a : a\in A, V \in \bbS_1\}$. This set is finite, so choosing $\delta\in (0,b_1)$ smaller than $\min\{ |V| : V\in \bbS_1^1\}$, we see that every $V\in \bbS_1^1$ lies in $\Vuqd$.

Let $b := \min(b_1,b_2 \delta)$, and observe that by Proposition \ref{prop:average-length}, for every $V\in \bbS_1^1$ and every $k\in \NN$, we have
\[
|T^kV| \geq \min(b_1, b_2 e^{\gamma k} \delta) \#\LLL_k^V \geq b \#\LLL_k^V.
\]
Summing over all $V\in \bbS_1^1$, we get
\begin{equation}\label{eqn:TkS0}
|T^{k+1}\SSS_0| = \sum_{V \in \bbS_1^1} |T^k V| \geq  \sum_{V \in \bbS_1^1} b \#\LLL_k^V = b \#\LLL_{k+1}^{\SSS_0}.
\end{equation}
Using this together with \eqref{eqn:cell-leq-curve} from Lemma \ref{lem:cells-curves} gives
\[
\#\BBB_n^{u,\delta} \leq C_7 n + 2 \sum_{k=1}^{n+1} \#\LLL_k^{\SSS_0}
\leq C_7 n + 2\sum_{k=1}^{n+1} b^{-1} |T^k \SSS_0|.
\]
Applying Lemma \ref{lem:s-contr}, we obtain
\begin{align*}
\#\BBB_n^{u,\delta}
&\leq C_7 n + 2b^{-1} \sum_{k=1}^{n+1} C_2 e^{-\lambda(n+1-k)} |T^{n+1} \SSS_0| \\
&= C_7 n + 2b^{-1} C_2 (1-e^{-\lambda})^{-1} |T^{n+1}\SSS_0|.
\end{align*}
Since $|T^n\SSS_0|$ grows exponentially fast, this implies that there exists $C_9>0$, depending only on $b$, $C_2$, $C_7$, $\lambda$, and $|\SSS_0|$, such that
\begin{equation}\label{eqn:Bnud-2}
\#\BBB_n^{u,\delta} \leq C_9 |T^{n+1}\SSS_0|
\quad\text{for all } n\in \NN.
\end{equation}
Because $\bbS_1^1 \subset \Vuqd$, we can apply Proposition \ref{prop:control-bad} to deduce that for every $n\geq n_1$, we have
\begin{align*}
|T^{n+1}\SSS_0| & = \sum_{V\in \bbS_{1}^1} |T^n V| \leq \sum_{V\in \bbS_{1}^1} L \# \LLL_{n}^V 
\leq \sum_{V\in \bbS_{1}^1} L b_3^{-1} \#\GGG_{n}^\delta(V)\\
& \leq \#\bbS_1^1  L b_3^{-1} \#\GGG_{n+1}^\delta(\SSS_0).
\end{align*}
We can combine this with \eqref{eqn:Bnud-2} and \eqref{eqn:curve-leq-cell} to obtain
\begin{equation}\label{eqn:Bnud}
\#\BBB_n^{u,\delta} \leq C_9 \#\bbS_1 L b_3^{-1} \#\GGG_n^\delta(\SSS_0)
\leq  C_9 \#\bbS_1 L b_3^{-1} \big( C_7 + L\delta^{-1} \#\GGG_n^{u,\delta} \big).
\end{equation}
Take $C_8$ sufficiently large that \eqref{eqn:Bn-Gn} holds for all $n\leq n_1$, and such that 
\[
C_9 \#\bbS_1 L b_3^{-1}\big( C_7 + L\delta^{-1} N \big) \leq C_8 N
\text{ for all } N \geq 1.
\]
Since $\#\GGG_n^{u,\delta} \geq 1$ for all $n$, combining this with \eqref{eqn:Bnud} proves \eqref{eqn:Bn-Gn} for all $n\geq n_1$, and hence for all $n\in \NN$. To derive \eqref{eqn:Gn-Ln} from this, it suffices to observe that
\[
\#\LLL_n = \#\GGG_n^{u,\delta} + \#\BBB_n^{u,\delta}
\leq (1 + C_8) \#\GGG_n^{u,\delta}.\qedhere
\]
\end{proof}

\subsection{Quasisupermultiplicativity and Fekete}\label{sec:qsm}

Using the family of sufficient Cantor rectangles from Proposition \ref{prop:suff-rect}, we can now prove Proposition \ref{prop:LV-geq-L}: for every $V \in \Vuqd$ and every $n\in \NN$, we have $\#\LLL_n^V \geq b_0 \#\LLL_n$, where $b_0>0$ is independent of $V$ and $n$.

\begin{proof}[Proof of Proposition \ref{prop:LV-geq-L}]
Using Proposition \ref{prop:suff-rect}, there exists a $\delta$-sufficient family of $\ell$ Cantor rectangles. 
Given $n\geq 0$, every cell in $\MMM_{0}^{n+1}$ that has a stable diameter of at least $\delta$ fully crosses some rectangle $R_i$.  Observe that the number of such cells is equal to $\#\GGG_n^{u,\delta}$ by symmetry of the singularity curves.  Thus, there exists $i_n$ such that $R_{i_n}$ is fully crossed by $\ell^{-1} \#\GGG_n^{u,\delta}$ cells in $\MMM_{0}^{n+1}$.

Moreover, by Proposition \ref{prop:suff-rect}, for each $i_n$, 
there exists $w \in \LLL_N^V$ such that $T^{N-1} V_w$ fully crosses $R_{i_n}$, and intersecting this image with the cells of $\MMM_{0}^{n+1}$ from the previous paragraph proves that
\begin{equation}\label{eqn:LnV-geq}
\#\LLL_{n+N}^V \geq \ell^{-1} \#\GGG_n^{u,\delta}.
\end{equation}
Combining this with Proposition \ref{prop:Bn-Gn}, we have
\[
\#\LLL_n \leq b_4^{-1} \#\GGG_n^{u,\delta}
\leq \ell b_4^{-1} \#\LLL_{n+N}^V
\leq \ell b_4^{-1} \#\LLL_n^V \#\LLL_N.
\]
This proves Proposition \ref{prop:LV-geq-L} with $b_0 := b_4 \ell^{-1} (\#\LLL_N)^{-1}$.
\end{proof}

As described in \S\ref{sec:overview}, Proposition \ref{prop:LV-geq-L} implies the local lower counting bound in Proposition \ref{prop:loc-lcb} by using \eqref{eqn:lcb}. 
We can also obtain the global upper counting bound in Proposition \ref{prop:ucb}. First we show:

\begin{proposition}[Quasisupermultiplicativity]\label{prop:qsm}
There exists $b_5>0$ such that for every $k,n\in \NN$, we have $\#\LLL_{n+k} \geq b_5 (\#\LLL_k)(\#\LLL_n)$.
\end{proposition}
\begin{proof}
Fix $k,n\in \NN$. Given $w\in \GGG_k^{u,\delta}$, there exists $V(w) \in \Vuqd$ such that $V(w) \subset Y_w^-$, so Proposition \ref{prop:LV-geq-L} gives
\begin{equation}\label{eqn:LVw}
\#\LLL_n^{V(w)} \geq b_0 \#\LLL_n,
\end{equation}
and using \eqref{eqn:Gn-Ln} from Proposition \ref{prop:Bn-Gn}, we conclude that
\[
\#\LLL_{n+k} \geq \sum_{w\in \GGG_k^{u,\delta}} \#\LLL_n^{V(w)}
\geq \sum_{w\in \GGG_k^{u,\delta}} b_0 \#\LLL_n
\geq b_0 b_4 \#\LLL_k \#\LLL_n.
\]
Taking $b_5 := b_0 b_4$ completes the proof.
\end{proof}

\begin{proof}[Proof of Proposition \ref{prop:ucb}]
To prove the global upper counting bound, let $a_n := b_5 \#\LLL_n$ and observe that Proposition \ref{prop:qsm} gives
\[
a_{n+k} = b_5 \#\LLL_{n+k} \geq b_5^2 \#\LLL_n \#\LLL_k = a_n a_k.
\]
Applying Fekete's Lemma to $-\log a_n$ gives $-\log a_n \geq -nh$ for all $n$, so $a_n \leq e^{nh}$, which suffices.
\end{proof}

\begin{corollary}\label{cor:counting-big-language}
If $\hat\LLL$ is the complete language of $X$, then there exists a $C_{10}$ such that for all $n$
\begin{equation}\label{eqn:counting-big-language}
e^{nh} \leq \#\hat\LLL_n \leq C_{10} e^{nh}.
\end{equation}
\end{corollary}
\begin{proof}
Clearly $e^{nh} \leq \#\LLL_n \leq \#\hat\LLL_n$.  By Lemma 3.2 of \cite{BD20}, there exists a $C>0$ that depends only on the table such that 
\[\#\hat\LLL_n \leq \#\LLL_n + Cn.\]
There exists an $N$ such that for $n\geq N$, we have $Cn \leq e^{nh}$. Thus,
\[\#\hat\LLL_n \leq \#\LLL_n + Cn \leq \#\LLL_n +e^{nh} \leq 2\#\LLL_n\leq 2C_4 e^{nh}.\]
Since there are only finitely many terms to correct for, there must exist a $C_{10}$ satisfying \eqref{eqn:counting-big-language} for all $n$.
\end{proof}

\begin{proof}[Proof of Theorem \ref{thm:length-growth}]
The lower bound in \eqref{eqn:TnV} follows from the local lower counting bound in Proposition \ref{prop:loc-lcb} together with Corollary \ref{cor:ave-length}.
For the upper bound in \eqref{eqn:TnV}, observe that the upper bound in Proposition \ref{prop:ucb} together with the upper bound on length of $u$-curves gives
$
|T^n V| \leq L \# \LLL_n^V \leq L C_4 e^{nh}.
$
\end{proof}

\section{The billiard MME}\label{sec:bill-construction}

Let $X \subset A^\ZZ$ be the coding space for a finite horizon Sinai billiard, as in \S\ref{sec:coding-bill}.
Consider the associated one-sided shifts $X^{\pm}$, and the measures $m^{\pm}$ constructed in \S\ref{sec:one-sided}.
Let $\mu$ be the invariant measure on $X$ constructed in \S\ref{sec:exterior-product}. 
We have shown in Corollary \ref{cor:finite} that the measures $m^{\pm}$ and $\mu$ are finite.
In this section, we use the counting bounds from \S\ref{sec:counting} to deduce 
the following.
\begin{itemize}
\item $\mu$ gives positive weight to $c(U)$ for every open $U \subset M$ (\S\ref{sec:positive}). In particular, it is nonzero, verifying Theorem \ref{thm:strategy}\ref{mu-fin-pos}; this in turn implies that $\bar\mu = \mu/\mu(X)$ is an MME 
(Proposition \ref{prop:inv-MME}) and has local product structure (Remark \ref{rmk:tools-for-lps}), verifying Theorem \ref{thm:strategy}\ref{mu-lps-supp}.
\item $\bar\mu$ is multiply mixing (\S\ref{sec:ergodic}), verifying Theorem \ref{thm:strategy}\ref{mu-erg}.
\item $\bar\mu$ is the unique MME (\S\ref{sec:unique}), verifying Theorem \ref{thm:strategy}\ref{mu-uniq}.
\end{itemize}

\subsection{Almost all local stable and unstable sets are non-trivial}

Consider the sets
\begin{equation}\label{eqn:X0}
	\begin{aligned}
		X_0^u &:= \{x\in X : |\Wul(\bx)| = 0 \}, \\
		X_0^s &:= \{x\in X : |\Wsl(\bx)| = 0 \}.
	\end{aligned}
\end{equation}
Recalling \eqref{eqn:X*su}, observe that $X_0^u \subset X_*^u$ and $X_0^s \subset X_*^s$ since $\mf$ and $\mb$ are nonatomic by Corollary \ref{cor:Gibbs}. Thus \eqref{eqn:m-X*} implies that
\begin{equation}\label{eqn:m-X0}
	\mf_x(X_0^u) = \mb_x(X_0^s) = 0
	\quad\text{for all } x\in X
\end{equation}
by Lemma \ref{lem:nonzero-mu-ae}.

In order to prove that $\mu$ gives positive weight to all open sets, we will need the following variant of \eqref{eqn:m-X0}, which requires more work.

\begin{lemma}\label{lem:m-X0-again}
	With $X_0^{u,s}$ as in \eqref{eqn:X0}, we have
	\begin{equation}\label{eqn:m-X0-again}
		\mb(\pim X_0^u) = \mf(\pip X_0^s) = 0.
	\end{equation}
\end{lemma}
\begin{proof}
We prove that $\mf(\pip X_0^s) = 0$; the proof for $\mb(\pim X_0^u)$ is similar. 
Given $x\in X^+$ and $n\in \NN$, let
\[
P_n(x) := \{ w\in \LLL_n : wx\in X^+ \}.
\]
Thus $\#P_n(x)$ is the number of cylinders of length $n$ required to cover all possible codings of points in $T^{-n}\Wsl(\bx)$. If $x\in \pip(X_0^s)$, then $\Wsl(\bx)$ is a single point, so as in Lemma \ref{lem:nonatomic-0} and \eqref{eqn:Kn}, we have 
	\begin{equation}\label{eqn:Kn-pre}
		\#P_n(x) \leq Kn+1
		\quad\text{for all } n\in \NN.
	\end{equation}	
	Given $w\in \LLL_n$, consider the set
	\[
	Z_w := \{ x\in \pip X_0^s \subset X^+ : w\in P_n(x) \}
	\]
	of all forward-infinite sequences that code points with trivial stables and have a preimage of order $n$ coded by $w$. Observe that
	\[
	\pip(X_0^s) = \bigsqcup_{w\in \LLL_n} \pip (X_0^s \cap \sigma\mbl w\mbr)
	= \bigsqcup_{w\in \LLL_n} \sigma^n|_{\pbl w \pbr}^{-1}Z_w,
	\]
	and that by Proposition \ref{prop:scaling}, we have
	\[
	\mf(\sigma^n|_{\pbl w \pbr}^{-1}Z_w)
	= e^{-nh} \mf(Z_w).
	\]
	Combining these and applying Fubini's theorem to the measure $\#\times \mf$ on $\LLL_n\times X^+$, we can use \eqref{eqn:Kn-pre} to obtain
	\begin{align*}
		\mf(\pip X_0^s) &= e^{-nh} \sum_{w\in \LLL_n} \mf(Z_w)
		= e^{-nh} \int_{\pip X_0^s} \#P_n(x) \,d\mf(x) \\
		&\leq (Kn+1) e^{-nh} \mf(\pip X_0^s).
	\end{align*}
	For large $n$, we have $(Kn+1) e^{-nh} < 1$, so $\mf(\pip X_0^s) = 0$.
\end{proof}

\subsection{Positivity on open sets}\label{sec:positive}

\begin{lemma}\label{lem:positive-leaf}
For all $\delta \in (0,b_1)$ there exists $c_\delta>0$ such that
\begin{equation}\label{eqn:mfbVus}
\begin{gathered}
\text{for every $V\in \Vud$, we have $\mf(\pip(\pi^{-1}V)) \geq c_\delta$}, \\
\text{for every $V\in \Vsd$, we have $\mb(\pim(\pi^{-1}V)) \geq c_\delta$}.
\end{gathered}
\end{equation}
In particular, for every $x\in X$, we have:
\begin{equation}\label{eqn:positive-leaf}
\begin{gathered}
\text{if $|\Wul(\bx)| \geq \delta$, then } \mf_x(\Wul(x)) \geq c_\delta, \\
\text{if $|\Wsl(\bx)| \geq \delta$, then } \mb_x(\Wsl(x)) \geq c_\delta.
\end{gathered}
\end{equation}
Moreover, both $\mf$ and $\mb$ are fully supported: 
\begin{equation}\label{eqn:mfb-full}
\text{given any $w\in \LLL$, we have $\mf(\pbl w \pbr)>0$ and $\mb(\mbl w \mbr) > 0$.}
\end{equation}
\end{lemma}
\begin{proof}
It suffices, to prove \eqref{eqn:mfbVus} for $\mf$. Indeed:
\begin{itemize}
\item \eqref{eqn:positive-leaf} follows from \eqref{eqn:mfbVus} by taking $V = \Wul(\bx)$;
\item \eqref{eqn:mfb-full} follows by observing that given any $w\in \LLL$, the open cell $Y_w^+ \subset M$ is nonempty and thus contains some curve $V\in \Vu$ with $|V|>0$;
\item the results for $\mb$ are proved similarly.
\end{itemize}
Given $V\in \Vud$, 
we will apply Proposition \ref{prop:Z-compact-lower} to the compact set $Z=V$.
Observe that \eqref{eqn:ucb} is satisfied with $C_1 = b_5^{-1}$ by Proposition \ref{prop:ucb}, so it suffices to verify \eqref{eqn:Z-lcb}.  
This follows immediately from Proposition \ref{prop:LV-geq-L}, which gives
\[
\#\{w \in \LLL_n : V \cap T^{-1} Y_w^+ \neq \emptyset\}
= \#\LLL_n^V \geq b_0 \#\LLL_n \geq b_0 e^{nh},
\]
and applying Proposition \ref{prop:Z-compact-lower} completes the proof with $c_\delta = b_0 b_5$.
\end{proof}

\begin{corollary}\label{cor:pos-rectangle}
	If $|\Wul(\bx)| > 0$, then there exists a symbolic rectangle $R \in c(\RRR)$ such that $\mf_x(R) > 0$.  
\end{corollary}
\begin{proof}
	If $|\Wul(\bx)| > 0$, then $\mf_x(\Wul(x)) > 0$ by Lemma \ref{lem:positive-leaf}.  Moreover, if $X_0^s = \{x \in X : |\Wsl(\bx)| = 0\}$, then $\mf_x(X_0^s) = 0$ by Lemma \ref{lem:m-X0-again}.  There is a countable cover of $\Wul(x)\setminus X_0^s$ by rectangles in $c(\RRR)$, so there must be at least one rectangle $R$ such that $\mf_x(R) > 0$.
\end{proof}

A similar statement holds for stable manifolds.

\begin{corollary}\label{cor:nonuniform-lower-gibbs-leaf}
For all $\delta \in (0,b_1)$, if $w \in \GGG_n^\delta(\Wul(\bx))$, then 
\[\mf_x(\cyl{}{w}) \geq c_\delta e^{-nh}.\]
\end{corollary}
\begin{proof}
This is an immediate consequence of the scaling property and \eqref{eqn:positive-leaf}.  For convenience, suppose $x \in \cyl{}{w}$.  Since $|T^{n}V_w| \geq \delta$,
\[\mf_x(\cyl{}{w}) = e^{-(n-1)h} \mf_{\sigma^{n-1}x} (\Wul(\sigma^{n} x)) \geq c_\delta e^{-(n-1)h} \geq c_\delta e^{-nh}.\qedhere\]
\end{proof}

\begin{proposition}\label{prop:full-support}
For every $w\in \LLL$, we have $\mu(\cyl{}{w})>0$. In particular, $\mu(X)>0$, and $\mu$ has full support on $c(M)$.
\end{proposition}
\begin{proof}
Given $w\in \LLL$, 
we recall the disintegration formula for $\mu$ from \eqref{eqn:s-cond} and obtain
\begin{equation}\label{eqn:mu-w}
\mu(\cyl{}{w}) = \int_{X^+} \mb_x(\cyl{}{w}) \,d\mf(x)
= \int_{\pbl w\pbr} \mb_x(\Wsl(x)) \,d\mf(x),
\end{equation}
where the last equality holds because
$
\mb_x(\cyl{}{w}) = \mb(\pim(\cyl{}{w} \cap \Wsl(x)))
$
and
\[
\cyl{}{w} \cap \Wsl(x) = \begin{cases} 
\Wsl(x) &\text{if } x\in \pbl w\pbr, \\
\emptyset &\text{if } x\notin \pbl w\pbr.
\end{cases}
\]
By Lemma \ref{lem:positive-leaf}, we have $\mb_x(\Wsl(x)) > 0$ for every $x\in \pip(\cyl{}{w} \setminus X_0^s)$, and by Lemma \ref{lem:m-X0-again}, we have
$\mf(\pip X_0^s) = 0$, so the function $x \mapsto \mb_x(\Wsl(x))$ is positive $\mf$-a.e.

Moreover, \eqref{eqn:mfb-full} from Lemma \ref{lem:positive-leaf} gives $\mf(\pbl w \pbr)>0$, so the last expression in \eqref{eqn:mu-w} is the integral of a nonnegative function that is strictly positive on a set of positive measure. It follows that the integral is positive, which completes the proof.
\end{proof}

As discussed as the start of \S\ref{sec:bill-construction}, Proposition  \ref{prop:inv-MME}) and Remark \ref{rmk:tools-for-lps} now imply that $\bar\mu$ is an MME with local product structure, proving Theorem \ref{thm:strategy}\ref{mu-fin-pos}--\ref{mu-lps-supp}.

\subsection{Multiple mixing}\label{sec:ergodic}

In this section, we prove Theorem \ref{thm:strategy}\ref{mu-erg} by showing that $\bar\mu$ is multiply mixing.
Given $\nu \in \MMM_\sigma(X)$, recall that:
\begin{itemize}
\item $\nu$ is \emph{totally ergodic} if the measure-preserving transformation $(X, \sigma^n, \nu)$ is ergodic for every $n\in \NN$;
\item $\nu$ is \emph{mixing} if for all Borel sets $E_1, E_2 \subset X$, we have
\[
\nu(E_1 \cap \sigma^{-n}E_2) \to \nu(E_1)\nu(E_2)
\quad\text{as }n\to\infty;
\]
\item $\nu$ is \emph{multiply mixing}, or \emph{mixing of all orders}, if for every $\ell\in \NN$ and all Borel sets $E_1,\dots, E_\ell \subset X$, we have 
\[
\lim_{\min(n_1,\dots, n_{\ell-1}) \to \infty}
\nu \Big( \bigcap_{j=1}^\ell \sigma^{-N_j}E_j \Big)
= \prod_{j=1}^\ell \nu(E_j),
\quad\text{where } N_j = \sum_{i=1}^{j-1} n_i.
\]
\end{itemize}
These satisfy the following implications:
\[
\text{multiply mixing}
\quad\Rightarrow\quad
\text{mixing}
\quad\Rightarrow\quad
\text{totally ergodic}
\quad\Rightarrow\quad
\text{ergodic}.
\]
In \cite{CHT16}, Coud\`ene, Hasselblatt, and Troubetzkoy 
described a rather general setting in which the Hopf argument can be used to establish some or all of these properties. We recall the relevant definitions and results from their work, modifying their notation slightly to fit our billiard shift space $X$.

Given $x\in X$, consider the \emph{global stable and unstable sets}
\begin{align*}
\Ws(x) &:= \{ y\in X : d(\sigma^n x, \sigma^n y) \to 0 \text{ as } n\to\infty \}, \\
\Wu(x) &:= \{ y\in X : d(\sigma^n x, \sigma^n y) \to 0 \text{ as } n\to -\infty \}.
\end{align*}
Observe that $\Ws(x) = \bigcup_{k=0}^\infty \sigma^{-k} \Wsl(x)$ is the set of all $y\in X$ for which there exists $k=k(y)$ such that $x_n = y_n$ for all $n\geq k$, and $\Wu(x)$ admits a similar characterization. The sets $\Ws$ form a partition of $X$, as do the sets $\Wu$. 

Now fix a $\sigma$-invariant Borel probability measure $\nu \in \MMM_\sigma(X)$.
With respect to $\nu$, a function $f \colon X\to \RR$ is \emph{subordinate to $\Ws$}, or \emph{$\Ws$-saturated}, if there exists $G\subset X$ with $\nu(G) = 1$ such that for every $x,y\in G$ satisfying $y\in \Ws(x)$, we have $f(x) = f(y)$. A similar definition applies for $\Wu$.

\begin{definition}\label{def:erg-joint}
The partition $\Ws$ is \emph{ergodic} for $\nu \in \MMM_\sigma(X)$ if every $\Ws$-saturated function $f\in L^2(\nu)$ is constant $\nu$-a.e., meaning that there exists a set $Z \subset X$ such that $\nu(Z) = 1$ and $f|_Z$ is constant.

The partitions $\Ws$ and $\Wu$ are \emph{jointly ergodic} for $\nu \in \MMM_\sigma(X)$ if for every function $f\in L^2(\nu)$ that is both $\Ws$-saturated and $\Wu$-saturated, $f$ is constant $\nu$-a.e.
\end{definition}

\begin{theorem}\label{thm:Hopf}
Let $X$ be a shift space on a finite alphabet, and fix $\nu \in \MMM_\sigma(X)$. 
\begin{enumerate}[label=\upshape{(\alph{*})}]
\item\label{get-mix}
If $\Ws$ and $\Wu$ are jointly ergodic for $\nu$, then $\nu$ is mixing.
\item\label{get-mult}
If $\Ws$ is ergodic for $\nu$, then $\nu$ is multiply mixing.
\item\label{get-Ws-erg}
If $\nu$ is totally ergodic and $\nu$ has local product structure in the sense of Definition \ref{def:lps}, then $\Ws$ is ergodic for $\nu$.
\end{enumerate}
\end{theorem}
\begin{proof}
Statement \ref{get-mix} follows from \cite[Theorem 3.3]{CHT16}. Statement \ref{get-mult} follows from \cite[Theorem 2.2]{CHT16}. Statement \ref{get-Ws-erg} follows from \cite[Theorem 5.1]{CHT16}, together with the observation that local product structure in the sense of Definition \ref{def:lps} implies the absolute continuity property of $\Ws$ formulated in \cite[Definition 4.2]{CHT16}.
\end{proof}

Since mixing measures are totally ergodic, we can combine the three conclusions of Theorem \ref{thm:Hopf} to obtain:

\begin{corollary}\label{cor:Hopf}
Let $X$ be a shift space on a finite alphabet. Suppose that  $\nu \in \MMM_\sigma(X)$ has local product structure in the sense of Definition \ref{def:lps}, and that $\Ws$ and $\Wu$ are jointly ergodic with respect to $\nu$. Then $\nu$ is multiply mixing.
\end{corollary}

We remark that \cite[Theorem 5.5]{CHT16} uses this result to provide a proof of multiple mixing for the Liouville measure on a Sinai billiard. Our task now is to show that it can be applied to the MME $\bar\mu$ as well.
Since we have already proved in \S\ref{sec:positive} that $\bar\mu$ has local product structure in the sense of Definition \ref{def:lps}, it suffices to show that $\Ws$ and $\Wu$ are jointly ergodic with respect to $\bar\mu$.

To this end, let $f\in L^2(\bar\mu)$ be both $\Ws$-saturated and $\Wu$-saturated; then there exists $G_0 \subset X$ such that
\begin{gather}
\label{eqn:G0a}
\bar\mu(G_0) = 1,\text{ and} \\
\label{eqn:G0b}
\text{for all $x,y\in G_0$ such that $y\in \Ws(x) \cup \Wu(x)$, we have $f(x) = f(y)$.}
\end{gather} 
Let $G = \Xreg \cap \bigcap_{n\in \ZZ} \sigma^{-n}G_0$; then  $\sigma^{-1}G = G$, and \eqref{eqn:G0a}--\eqref{eqn:G0b} continue to hold with $G_0$ replaced by $G$.
Consider the set 
\begin{equation}\label{eqn:AAA}
\AAA = \{x \in G : \mf_x(\Wul(x) \setminus G) = 0\}.
\end{equation}
Recalling \eqref{eqn:u-cond}, we have
\[
0 = \mu(X\setminus G) = \int_{X^-} \mf_x(X\setminus G) \,d\mb(x),
\]
from which we conclude that $\mf_x(X\setminus G)=0$ for $\mb$-a.e.\ $x$, and thus $\bar\mu(\AAA)=1$. To show that $\Ws$ and $\Wu$ are jointly ergodic, it now suffices to show that $f|_\AAA$ is constant. We will need the following lemma.

\begin{lemma}\label{lem:A-inv}
	If $x \in \AAA$, then $\sigma^n x \in \AAA$ for all $n \geq 0$.
\end{lemma}
\begin{proof}
	First observe that if $x \in G$, then $\sigma^n x \in G$. Suppose $x \in \AAA$.  Let $u \in \LLL_{n+1}$ with $\cyl{}{u}\cap\Wul(x)\neq\emptyset$.  Observe that if $y \in \cyl{}{u} \cap \Wul(x)$, then $\sigma^{n}(\Wul(x)\cap\cyl{}{u}) = \Wul(y)$.  By the scaling property from Proposition \ref{prop:scaling}, we have that
	\[0 = \mf_x(\cyl{}{u} \setminus G) = e^{-nh} \mf_y(\Wul(y) \setminus G). \qedhere\]
\end{proof}

\begin{remark}
The proof of Lemma \ref{lem:A-inv} fails for negative $n$, so it is important that we go forward, not backward, when we use it in the arguments below.
\end{remark}

Now we prove that $f|_\AAA$ is constant. First, fix $z\in \Xreg$, so that $|\Wul(\bill{z})|>0$. By Corollary \ref{cor:pos-rectangle}, there exists a symbolic rectangle $R\in c(\mathcal{R})$ such that $\mf_z(R) > 0$. Without loss of generality, assume that $z\in R$.

Given $x,y\in \AAA$, we have  $\delta := \min(|\Wul(x)|, |\Wul(y)|) > 0$.
By Proposition \ref{prop:suff-rect}, there exists $N\in \NN$ and subsets $V_1 \subset \Wul(x)$, $V_2 \subset \Wul(y)$ such that $T^N(\pi(V_i))$ crosses $\pi(R)$ for $i=1,2$, as shown in Figure \ref{fig:Hopf}.

\begin{figure}[htbp]
\centering
\includegraphics{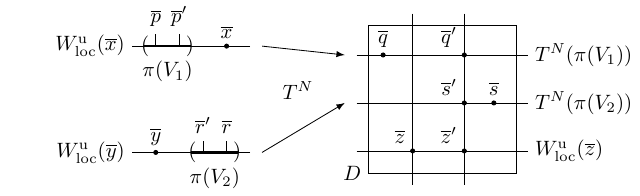}
\caption{Proving that $\Ws$ and $\Wu$ are jointly ergodic.}
\label{fig:Hopf}
\end{figure}

Since $x \in \AAA$, we have $\mf_x(V_1 \cap \AAA) = \mf_x(V_1)>0$, where the inequality uses \eqref{eqn:positive-leaf} from Lemma \ref{lem:positive-leaf}.  Thus there exists a point $p \in V \cap \AAA \subset \Wul(x)$, and by Lemma \ref{lem:A-inv}, we have $q := \sigma^Np \in \sigma^N (V_1) \cap \AAA$. A similar argument allows us to fix a point $r \in V_2 \cap \AAA \subset \Wul(y)$, for which we have $s := \sigma^N r \in \sigma^N (V_2) \cap \AAA$. Note that $q$ and $s$ need not lie in $R$.

Let $B_i := (\sigma^N (V_i) \cap R) \setminus G$ for $i=1,2$.
Since $\mf_x(V_1\setminus G) = 0$, we can use the invariance of $G$ and the scaling properties of $\mf$ from Proposition \ref{prop:scaling} to conclude that
\[
\mf_q(B_1) = \mf_q(\sigma^N(V_1 \setminus G)) = 0,
\]
and similarly, $\mf_s(B_2) = 0$. By Lemma \ref{lem:holonomy}, this implies that
\[
\mf_z(\pi_{q,z}(B_1) \cup \pi_{s,z}(B_2)) = 0.
\]
Since $\mf_z(R)>0$, there exists $z' \in (\Wul(z) \cap R) \setminus (\pi_{q,z}(B_1) \cup \pi_{s,z}(B_2))$. 
Since $z' \in R$, the s-curve $\Wsl(\bill{z'})$ crosses the solid rectangle $D$ that generates the Cantor rectangle $\pi(R)$, and thus this curve intersects $\pi(\sigma^NV_i)$ for $i=1,2$. This implies that the points
\[
q' := \pi_{z,p}^s(z')
\qquad\text{and}\qquad
s' := \pi_{z,q}^s(z')
\]
are both defined, and both lie in $R \cap G$.
Moreover, $q', s' \in \Wsl(z')$, so $s' \in \Wsl(q')$.

Let $p' := \sigma^{-N}q'$ and $r' := \sigma^{-N}s'$. Then we must have $p',r' \in G$ with $r' \in \sigma^{-N}\Wsl(\sigma^N r') \subset \Ws(p')$.

We now have four points $x,p',r',y \in G$ such that
\[
p' \in \Wul(x),
\qquad r' \in \Ws(p'),
\qquad y \in \Wul(r').
\]
By \eqref{eqn:G0b}, this implies that
\[
f(x) = f(p') = f(r') = f(y).
\]
Since $x,y\in \AAA$ were arbitrary, this shows that $f$ is constant on $\AAA$. 
This completes the proof that $\Ws$ and $\Wu$ are jointly ergodic; by Corollary \ref{cor:Hopf}, we conclude that $\bar\mu$ is multiply mixing, which proves Theorem \ref{thm:strategy}\ref{mu-erg}.

\subsection{Uniqueness}\label{sec:unique}

To prove that the measure of maximal entropy is unique, we will continue to follow the leafwise approach taken in \cite{CPZ2}, which naturally adapts to our setting since we construct $\bar\mu$ via its leaf measures. However, \cite{CPZ2}, which was written in the context of partially hyperbolic diffeomorphisms, assumes that there are open sets with product structure and a uniform lower Gibbs bound, which is not true in the billiard setting.  To overcome this issue, we will truncate words according to the leaf-wise decomposition of the language from Lemma \ref{lem:decomp}, using an argument analogous to one from \cite{BD20}; similar estimates also appeared in \cite{CT12}. In the terminology of \S\ref{sec:good-bad-ugly}, the idea will be that while there may be many subcurves whose final image is short, ``nearly all'' subcurves of order $2n$ have a long intermediate image that occurs between the times $n$ and $2n$.

Given a continuous function $f\colon X \to \RR$ we denote the forward and backward Birkhoff averages by
\[
f^-(x)  = \lim_{n \to \infty} \frac{1}{n} \sum_{j=0}^{n-1} f(\sigma^{-j} x)
\quad\text{and}\quad
f^+(x)  = \lim_{n \to \infty} \frac{1}{n} \sum_{j=0}^{n-1} f(\sigma^j x)
\]
wherever the limits exist.
A point $x \in X$ is \emph{Birkhoff regular for $f$} if these limits exist and agree, in which case we write $\uf(x) := f^+(x) = f^-(x)$.  
We say that $x$ is \emph{Birkhoff regular} if it is Birkhoff regular for every continuous $f$, and we denote the set of such points by $\BBB$. By the Birkhoff Ergodic Theorem and separability of $C(X)$, we have $\nu(X\setminus \BBB) = 0$ for every $\nu \in \MMM_\sigma(X)$.

Define an equivalence relation on $\BBB$ by
\begin{equation}\label{eqn:sim}
x\sim y \quad\Leftrightarrow\quad
\uf(x) = \uf(y)
\text{ for all } f\in C(X).
\end{equation}
By \eqref{eqn:d-to-0}, we see that $x\sim y$ whenever $y\in \Wul(x)$ or $y\in \Wsl(x)$. Similarly, $x\sim \sigma^n x$ for every $n\in \ZZ$.

Let $\Breg = \BBB \cap \Xreg$ be the set of Birkhoff regular points that correspond to interior points of their stable and unstable manifolds. Given an ergodic invariant measure $\nu$ on $X$ with $h(\nu) > 0$, it follows from Proposition \ref{prop:reg-meas}
that $\nu(\Breg) = 1$, and thus by the same argument as for \eqref{eqn:AAA} in \S\ref{sec:ergodic}, the set
\begin{equation}\label{eqn:Anu}
\AAA_{\nu} := \{x \in \Breg : \nu_x(\Wul(x) \setminus \Breg) = 0\}	
\end{equation}
has full $\nu$-measure. 
Similarly, by the Birkhoff Ergodic Theorem and separability of $C(X)$, the set
\begin{equation}\label{eqn:generic}
G_\nu := \Big\{ x\in \Breg : \uf(x) = \int f \,d\nu \text{ for all } f\in C(X) \Big \}
\end{equation}
has full measure, and thus
\begin{equation}\label{eqn:AG-full}
\nu(\AAA_\nu \cap G_\nu) = 1.
\end{equation}

\begin{lemma}\label{lem:singular-leaf}
Let $\nu$ be an ergodic measure on $X$ with positive entropy, and let $\{ \nu_x : x\in \CCC_\nu \}$ be its system of conditional measure from Lemma \ref{lem:disint} in \S\ref{sec:cond-ent}.
If $\nu$ and $\mu$ are mutually singular, then for any $x \in \AAA_\nu \cap G_\nu \cap \CCC_\nu$, the measures $\nu_x$ and $\mf_x$ are mutually singular.
\end{lemma}
\begin{proof}
It suffices to prove the contrapositive. Let $\nu$ be a positive entropy ergodic measure, and suppose that there exists $x\in \AAA_\nu \cap G_\nu \cap \CCC_\nu$ such that $\nu_x$ and $\mf_x$ are \emph{not} mutually singular. We will show that this implies $\nu = \bar\mu$.

To this end, consider the set
\[
Z = \bigcup_{y \in \Wul(x) \cap \Breg} \Wsl(y).
\]
Since $\Breg$ has full $\nu_x$-measure in $\Wul(x)$ and we assumed that $\nu_x \not\perp \mf_x$, we must have $\mf_x(\Breg) > 0$.
Recalling from Lemma \ref{lem:positive-leaf} that $\mb_y(\Wsl(y)) > 0$ for every $y\in \Breg$, we see that $\bar\mu(Z) > 0$.  Since $\bar\mu(\AAA_{\bar\mu} \cap G_{\bar\mu}) = 1$, there exists a point $z \in Z \cap \AAA_{\bar\mu} \cap G_{\bar\mu}$. 
We have $z \in \Wsl(y)$ for some $y\in \Wul(x) \cap \Breg$, so we see that
$x,y,z \in \BBB$, and $x\sim y \sim z$. At the same time, we have $x\in G_\nu$ and $z\in G_{\bar\mu}$, so for every $f\in C(X)$, we have
\[
\int f \,d\nu = \uf(x) = \uf(y) = \uf(z) = \int f \,d{\bar\mu}.\qedhere
\]
\end{proof}

We need one more lemma before proceeding to the proof of uniqueness and establishing Theorem \ref{thm:strategy}\ref{mu-uniq}.

Let $\alpha = 3h/4$, and fix $\delta_0 > 0$ sufficiently small to satisfy the conditions of Lemma \ref{lem:tempered} for this choice of $\alpha$.  Fix $x\in X$ and let $V = \Wul(x)$, and recall the definitions of $\LLL_n^V$ and $\GGG_n^{\delta_0}(V)$ from \S\ref{sec:good-bad-ugly}.  We consider the following sets of words:
\begin{equation}\label{eqn:LnGB}
\begin{aligned}
\LLL_{2n}^G(x) &:= \{ w\in \LLL_{2n}^V : w_{[1,k]} \in \GGG_k^{\delta_0}(V) \text{ for some } k>n  \}, \\
\LLL_{2n}^B(x) &:= \LLL_{2n}^V \setminus \LLL_{2n}^G.
\end{aligned} 
\end{equation}
Observe that $\LLL_{2n}^B(x)$ corresponds
to subcurves in $\Wul(\bx)$ whose final $n$ iterates are always short.  The following lemma is a leaf-wise analog of Lemma 7.22 of \cite{BD20}; similar estimates also appeared in \cite{CT12}.

\begin{lemma}\label{lem:counting-bad-end}
There exists a constant $C_{11}$ such that for all $x\in X$, we have
	\begin{equation}\label{eqn:L2n}
		\LLL_{2n}^B(x) \leq C_{11} e^{7nh/4}.
	\end{equation}
\end{lemma}
\begin{proof}

Given $w\in \LLL_{2n}^B(x) \setminus \UUU_{2n}^{\delta_0}(V)$, let $k = k(w)$ be maximal such that $v := w_{[1,k]} \in \GGG_k^{\delta_0}(V)$, and let $t = 2n-k$; then Lemma \ref{lem:decomp} gives $w_{(k,2n]} \in \UUU_t^{\delta_0}(T^k V_v)$, and \eqref{eqn:LnGB} gives $t \geq n$, so we have
\[
\LLL_{2n}^B(x) \setminus \UUU_{2n}^{\delta_0}(V) \subset \bigcup_{t=n}^{2n-1} \bigcup_{v \in \GGG_{2n-t}^{\delta_0}(V)} \UUU_t^{\delta_0}(T^{2n-t}V_v).
\]
Combining the upper bounds from Lemma \ref{lem:tempered} and Proposition \ref{prop:ucb}, we obtain
\begin{align*}
\#\LLL_{2n}^B(x) &= \#\UUU_{2n}^{\delta_0}(V) + \sum_{t=n}^{2n-1} \sum_{v \in \GGG_{2n-t}^{\delta_0}(V)} \#\UUU_t^{\delta_0}(T^{2n-t}V_v) \\
&\leq
C_5 e^{2n \cdot 3h/4} + \sum_{t=n}^{2n-1} (\#\GGG_{2n-t}^{\delta_0}(V)) C_5 e^{t \cdot 3h/4} \\
&\leq \sum_{t=n}^{2n} Q e^{(2n-t) h} C_5 e^{t \cdot 3h/4}
= Q C_5 e^{2nh} \sum_{t=n}^{2n} e^{-th/4}.
\end{align*}
The sum is bounded above by $e^{-nh/4} (1-e^{-h/4})^{-1}$, so this implies \eqref{eqn:L2n}.
\end{proof}

Now we prove that $\bar\mu$ is the unique MME. Since we already proved that $\bar\mu$ is multiply mixing, and hence ergodic, it suffices to show that there are no ergodic MMEs that are mutually singular with respect to $\bar\mu$. To this end, let $\nu$ be any ergodic positive entropy measure such that $\nu \perp \bar\mu$. 
By Lemma \ref{lem:singular-leaf}, this implies that 
\begin{equation}\label{eqn:nxmx}
\nu_x \perp \mf_x \quad\text{for every } x\in \AAA_\nu \cap G_\nu \cap \CCC_\nu.
\end{equation}
We will use this fact to show that $h_\nu < h$, recalling from Lemma \ref{lem:h-int} that for every $n\in \NN$, we have
\begin{equation}\label{eqn:h-int-sum}
\begin{aligned}
2nh_\nu &= \int_{x\in X^-} \sum_{w\in \LLL_{2n}} \psi(\nu_x(\sigma^{-1}\cyl{}{w})) \,d\nu(x) \\
&= \int_{x\in X^-} \sum_{w\in \LLL_{2n}} \psi(\tnu_x(\cyl{}{w})) \,d\nu(x),
\end{aligned}
\end{equation}
where we write $\tnu_x := \sigma_* \nu_x$, and 
\[
\psi(t) = \begin{cases} -t\log t & \text{if } t>0, \\ 0 & \text{if } t=0. \end{cases}
\]
(The appearance of $\sigma^{-1}$ in \eqref{eqn:h-int-sum} is due to the fact that the cylinder $\cyl{}{w}$ starts indexing at $0$, instead of at $1$.)
We will use the mutual singularity property in \eqref{eqn:nxmx} to
obtain an upper bound for the sum in \eqref{eqn:h-int-sum} by splitting the sum into two pieces: one that carries nearly full weight for $\tnu_x$, and one that carries nearly full weight for $\mtf_x := \sigma_* \mf_x$.

To carry this out, we adopt the following notation:
given $\DDD \subset \LLL$, denote the corresponding union of cylinders by
\[
\cylu{\DDD} := \bigcup_{w \in \DDD} \cyl{}{w}.
\]
Fix $x\in \AAA_\nu \cap G_\nu$. By \eqref{eqn:nxmx}, for each $n\in \NN$ there exists $\HHH_n(x) \subset \LLL_n$ such that 
\begin{equation}\label{eqn:Hn-to-10}
\tnu_x(\cylu{\HHH_n(x)}) \nearrow 1
\qquad \text{and}\qquad 
\mtf_x(\cylu{\HHH_n(x)}) \searrow 0
\qquad \text{as $n\to\infty$.}
\end{equation}
Now we consider the set of words
\begin{equation}\label{eqn:D-2n}
\DDD_{2n}(x) = \{w \in \LLL_{2n} : \cyl{}{w} \subset \cylu{\HHH_n(x)}\},
\end{equation}
and writing $\DDD_{2n}^c(x) := \LLL_{2n} \setminus \DDD_{2n}(x)$, we obtain the following bound:
\begin{align*}
\sum_{w\in \LLL_{2n}} \psi(\tnu_x(\cyl{}{w}))
& = \sum_{w\in \DDD_{2n}(x)} \psi(\tnu_x(\cyl{}{w}))
+ \sum_{w\in \DDD_{2n}^c(x)} \psi(\tnu_x(\cyl{}{w})) \\
&\leq \tnu_x(\cylu{\HHH_n(x)}) \log(\#\DDD_{2n}(x)) + \tnu_x(\cylu{\HHH_n^c(x)}) \log(\#\DDD_{2n}^c(x)) + \frac{2}{e},
\end{align*}
where the last line uses a standard entropy inequality; see \cite[(20.3.5)]{Kat}.  Recalling the upper counting bound from Proposition \ref{prop:ucb}, we have
\[
\tnu_x(\cylu{\HHH_n^c(x)}) \log(\#\DDD_{2n}^c(x))
\leq (1-\tnu_x(\cylu{\HHH_n(x)})) (2nh + \log Q),
\]
and combining this with the previous estimate gives
\begin{equation}\label{eqn:partway}
\sum_{w\in \LLL_{2n}} \psi(\tnu_x(\cyl{}{w}))
\leq \tnu_x(\cylu{\HHH_n(x)})\big(\log (\#\DDD_{2n}(x)) - 2nh - \log Q\big) + 2nh + Q',
\end{equation}
where $Q' = \log Q + \frac 2e$.
To bound $\#\DDD_{2n}(x)$, we write
\[
\DDD_{2n}^B(x) := \DDD_{2n}(x) \cap \LLL_{2n}^B(x),
\qquad
\DDD_{2n}^G(x) := \DDD_{2n}(x) \cap \LLL_{2n}^G(x),
\]
so that $\DDD_{2n}(x) = \DDD_{2n}^B(x) \cup \DDD_{2n}^G(x)$.
By Lemma \ref{lem:counting-bad-end}, we have
\begin{equation}\label{eqn:DB-leq}
\#\DDD_{2n}^B(x) \leq C_{11} e^{7nh/4}.
\end{equation}
To bound $\#\DDD_{2n}^G(x)$, we observe that by \eqref{eqn:LnGB}, for each $w\in \DDD_{2n}^G(x)$, the set
\[
k^G(w) := \{ k\in \{n+1,n+2,\dots, 2n\} : w_{[1,k]} \in \GGG_k^{\delta_0}(V) \}
\]
is nonempty. Let $k(w) := \min k^G(w)$, and let $\trunc(w) := w_{[1,k(w)]}$.
The  truncation map $\trunc \colon \DDD_{2n}^G(x) \to \bigcup_{k=n+1}^{2n} \GGG_k^{\delta_0}(V)$
need not be injective; however, for each 
$u\in \GGG_k^{\delta_0}(V)$, we have
\[
\#\phi^{-1}(u) \leq \#\LLL_{2n-k} \leq Q e^{(2n-k)h} = Q e^{2nh} e^{-|u|h}.
\]
At the same time, Corollary \ref{cor:nonuniform-lower-gibbs-leaf} gives
\[
\mtf_x(\cyl{}{u}) \geq c_{\delta_0} e^{-|u|h},
\]
and we obtain the bound
\begin{equation}\label{eqn:sum-Gibbs}
\#\DDD_{2n}^G(x) \leq 
\sum_{u\in \phi(\DDD_{2n}^G(x))} \#\phi^{-1}(u)
\leq Q c_{\delta_0}^{-1} e^{2nh} \sum_{u\in \phi(\DDD_{2n}^G(x))} \mtf_x(\cyl{}{u}).
\end{equation}
For each $u\in \trunc(\DDD_{2n}^G(x))$, we see from \eqref{eqn:D-2n} that there is some $w\in \LLL_{2n}$ such that $w_{[1,n]} \in \HHH_n(x)$ and $w_{[1,|u|]} = u$. This implies that for \emph{every} $v\in \LLL_{2n}$ with $v_{[1,|u|]} = u$, we have $v_{[1,n]} = w_{[1,n]} \in \HHH_n(x)$, so $v\in \DDD_{2n}(x)$, and we conclude that
\begin{equation}\label{eqn:trunc-cyl}
\cyl{}{u} \subset \bigcup_{w\in \DDD_{2n}(x)} \cyl{}{w} \subset \cylu{\HHH_n(x)}.
\end{equation}
We want to use \eqref{eqn:trunc-cyl} to replace the sum in \eqref{eqn:sum-Gibbs} by $\mtf_x(\cylu{\HHH_n(x)})$. To do this, we must show that distinct elements of $\trunc(\DDD_{2n}^G(x))$ have disjoint cylinders. 

Suppose $v,w \in \DDD_{2n}^G(x)$ are such that $\cyl{}{\trunc(v)}$ and $\cyl{}{\trunc(w)}$ are not disjoint; then these cylinders must be nested. Without loss of generality, suppose $\cyl{}{\trunc(w)} \subset \cyl{}{\trunc(v)}$. This implies that $\trunc(v)$ is a prefix of $\trunc(w)$, and thus $k(v) \leq k(w)$. Moreover,
\[
w_{[1,k(v)]} = (\trunc(w))_{[1,k(v)]} = \trunc(v) \in \GGG_{k(v)}^{\delta_0}(V),
\]
so $k(v) \in k^G(w)$, and since $k(w) = \min k^G(w)$, we conclude that $k(w) \leq k(v)$, so the two are in fact equal, and we have $\trunc(v) = \trunc(w)$.

Having proved that the cylinders $\{ \cyl{}{u} : u\in \trunc(\DDD_{2n}^G(x)) \}$ are disjoint, we can use \eqref{eqn:sum-Gibbs} and \eqref{eqn:trunc-cyl} to conclude that
\begin{equation}\label{eqn:DG-leq}
\#\DDD_{2n}^G(x) \leq Qc_{\delta_0}^{-1} e^{2nh} \mtf_x(\cylu{\trunc(\DDD_{2n}^G(x))})
\leq Q c_{\delta_0}^{-1} e^{2nh} \mtf_x(\cylu{\HHH_n(x)}).
\end{equation}
Using the estimates from \eqref{eqn:DB-leq} and \eqref{eqn:DG-leq}, we get
\[
\#\DDD_{2n}(x) \leq C_{11} e^{7nh/4} + Qc_{\delta_0}^{-1} \mtf_x(\cylu{\HHH_n(x)}) e^{2nh},
\]
and writing $d_n(x) := \#\DDD_{2n}(x) e^{-2nh}$, we see from \eqref{eqn:Hn-to-10} that
\begin{equation}\label{eqn:dnx}
d_n(x) \leq C_{11} e^{-nh/4} + Qc_{\delta_0}^{-1} \mtf_x(\cylu{\HHH_n(x)})
\searrow 0
\quad\text{as } n\to\infty.
\end{equation}
From \eqref{eqn:partway}, we have
\[
\sum_{w\in \LLL_{2n}} \psi(\tnu_x(\cyl{}{w}))
\leq \tnu_x(\cylu{\HHH_n(x)})\big(\log d_n(x) - \log Q\big) + 2nh + Q',
\]
and rearranging gives
\begin{equation}\label{eqn:leafwise-geq}
2nh - \sum_{w\in \LLL_{2n}} \psi(\tnu_x(\cyl{}{w}))
\geq \tnu_x(\cylu{\HHH_n(x)}) (-\log d_n(x) + \log Q) - Q'.
\end{equation}
By \eqref{eqn:Hn-to-10} and \eqref{eqn:dnx}, the right-hand side goes to $\infty$ as $n\to\infty$, and we conclude that
the function
\[
\Delta_n(x) := 2nh - \sum_{w\in \LLL_{2n}} \psi(\tnu_x(\cyl{}{w}))
\]
has the property that
\begin{equation}\label{eqn:Delta-h}
\lim_{n\to\infty} \Delta_n(x) = \infty
\quad\text{for all } x\in \AAA_\nu \cap G_\nu.
\end{equation}
Using \eqref{eqn:h-int-sum} (which was a consequence of Lemma \ref{lem:h-int}), we have
\[
2n(h-h_\nu) = \int_{x\in X^-} \Big( 2nh - \sum_{w\in \LLL_{2n}}
\psi(\tnu_x(\cyl{}{w}))\Big) \,d\nu(x)
= \int_{x\in X^-} \Delta_n(x) \,d\nu(x)
\]
for every $n\in \NN$. Using the upper counting bound in Proposition \ref{prop:ucb}, we have
\[
\Delta_n(x) \geq 2nh - \log\#\LLL_{2n} \geq -Q
\quad\text{for all }n\in \NN,
\]
so we can apply Fatou's lemma and deduce that
\[
2n(h-h_\nu) = \lim_{n\to\infty} \int_{x\in X^-} \Delta_n(x) \,d\nu(x)
\geq \int_{x\in X^-} \lim_{n\to\infty} \Delta_n(x) \,d\nu(x) = \infty.
\]
It follows that $h-h_\nu > 0$, so $\nu$ is not a measure of maximal entropy. This completes the proof of Theorem \ref{thm:strategy}\ref{mu-uniq}, and hence of Theorem \ref{thm:mme-unique}.


\bibliographystyle{amsalpha}
\bibliography{billiard-mme-ref}

	
\end{document}